\documentclass[12pt]{amsart}
\usepackage[dvipsnames]{xcolor}

\usepackage{etoolbox}

\makeatletter
\let\old@tocline\@tocline
\let\section@tocline\@tocline
\newcommand{\subsection@dotsep}{4.5}
\newcommand{\subsubsection@dotsep}{4.5}
\patchcmd{\@tocline}
  {\hfil}
  {\nobreak
     \leaders\hbox{$\m@th
        \mkern \subsection@dotsep mu\hbox{.}\mkern \subsection@dotsep mu$}\hfill
     \nobreak}{}{}
\let\subsection@tocline\@tocline
\let\@tocline\old@tocline

\patchcmd{\@tocline}
  {\hfil}
  {\nobreak
     \leaders\hbox{$\m@th
        \mkern \subsubsection@dotsep mu\hbox{.}\mkern \subsubsection@dotsep mu$}\hfill
     \nobreak}{}{}
\let\subsubsection@tocline\@tocline
\let\@tocline\old@tocline

\let\old@l@subsection\l@subsection
\let\old@l@subsubsection\l@subsubsection

\def\@tocwriteb#1#2#3{%
  \begingroup
    \@xp\def\csname #2@tocline\endcsname##1##2##3##4##5##6{%
      \ifnum##1>\c@tocdepth
      \else \sbox\z@{##5\let\indentlabel\@tochangmeasure##6}\fi}%
    \csname l@#2\endcsname{#1{\csname#2name\endcsname}{\@secnumber}{}}%
  \endgroup
  \addcontentsline{toc}{#2}%
    {\protect#1{\csname#2name\endcsname}{\@secnumber}{#3}}}%

\newlength{\@tocsectionindent}
\newlength{\@tocsubsectionindent}
\newlength{\@tocsubsubsectionindent}
\newlength{\@tocsectionnumwidth}
\newlength{\@tocsubsectionnumwidth}
\newlength{\@tocsubsubsectionnumwidth}
\newcommand{\settocsectionnumwidth}[1]{\setlength{\@tocsectionnumwidth}{#1}}
\newcommand{\settocsubsectionnumwidth}[1]{\setlength{\@tocsubsectionnumwidth}{#1}}
\newcommand{\settocsubsubsectionnumwidth}[1]{\setlength{\@tocsubsubsectionnumwidth}{#1}}
\newcommand{\settocsectionindent}[1]{\setlength{\@tocsectionindent}{#1}}
\newcommand{\settocsubsectionindent}[1]{\setlength{\@tocsubsectionindent}{#1}}
\newcommand{\settocsubsubsectionindent}[1]{\setlength{\@tocsubsubsectionindent}{#1}}

\renewcommand{\l@section}{\section@tocline{1}{\@tocsectionvskip}{\@tocsectionindent}{}{\@tocsectionformat}}%
\renewcommand{\l@subsection}{\subsection@tocline{2}{\@tocsubsectionvskip}{\@tocsubsectionindent}{}{\@tocsubsectionformat}}%
\renewcommand{\l@subsubsection}{\subsubsection@tocline{3}{\@tocsubsubsectionvskip}{\@tocsubsubsectionindent}{}{\@tocsubsubsectionformat}}%
\newcommand{\@tocsectionformat}{}
\newcommand{\@tocsubsectionformat}{}
\newcommand{\@tocsubsubsectionformat}{}
\expandafter\def\csname toc@1format\endcsname{\@tocsectionformat}
\expandafter\def\csname toc@2format\endcsname{\@tocsubsectionformat}
\expandafter\def\csname toc@3format\endcsname{\@tocsubsubsectionformat}
\newcommand{\settocsectionformat}[1]{\renewcommand{\@tocsectionformat}{#1}}
\newcommand{\settocsubsectionformat}[1]{\renewcommand{\@tocsubsectionformat}{#1}}
\newcommand{\settocsubsubsectionformat}[1]{\renewcommand{\@tocsubsubsectionformat}{#1}}
\newlength{\@tocsectionvskip}
\newcommand{\settocsectionvskip}[1]{\setlength{\@tocsectionvskip}{#1}}
\newlength{\@tocsubsectionvskip}
\newcommand{\settocsubsectionvskip}[1]{\setlength{\@tocsubsectionvskip}{#1}}
\newlength{\@tocsubsubsectionvskip}
\newcommand{\settocsubsubsectionvskip}[1]{\setlength{\@tocsubsubsectionvskip}{#1}}

\patchcmd{\tocsection}{\indentlabel}{\makebox[\@tocsectionnumwidth][l]}{}{}
\patchcmd{\tocsubsection}{\indentlabel}{\makebox[\@tocsubsectionnumwidth][l]}{}{}
\patchcmd{\tocsubsubsection}{\indentlabel}{\makebox[\@tocsubsubsectionnumwidth][l]}{}{}

\newcommand{\@sectypepnumformat}{}
\renewcommand{\contentsline}[1]{%
  \expandafter\let\expandafter\@sectypepnumformat\csname @toc#1pnumformat\endcsname%
  \csname l@#1\endcsname}
\newcommand{\@tocsectionpnumformat}{}
\newcommand{\@tocsubsectionpnumformat}{}
\newcommand{\@tocsubsubsectionpnumformat}{}
\newcommand{\setsectionpnumformat}[1]{\renewcommand{\@tocsectionpnumformat}{#1}}
\newcommand{\setsubsectionpnumformat}[1]{\renewcommand{\@tocsubsectionpnumformat}{#1}}
\newcommand{\setsubsubsectionpnumformat}[1]{\renewcommand{\@tocsubsubsectionpnumformat}{#1}}
\renewcommand{\@tocpagenum}[1]{%
  \hfill {\mdseries\@sectypepnumformat #1}}

\let\oldappendix\appendix
\renewcommand{\appendix}{%
  \leavevmode\oldappendix%
  \addtocontents{toc}{%
    \protect\settowidth{\protect\@tocsectionnumwidth}{\protect\@tocsectionformat\sectionname\space}%
    \protect\addtolength{\protect\@tocsectionnumwidth}{2em}}%
}
\makeatother

\makeatletter
\settocsectionnumwidth{2em}
\settocsubsectionnumwidth{2.5em}
\settocsubsubsectionnumwidth{3em}
\settocsectionindent{1pc}%
\settocsubsectionindent{\dimexpr\@tocsectionindent+\@tocsectionnumwidth}%
\settocsubsubsectionindent{\dimexpr\@tocsubsectionindent+\@tocsubsectionnumwidth}%
\makeatother

\settocsectionvskip{10pt}
\settocsubsectionvskip{0pt}
\settocsubsubsectionvskip{0pt}
    
\settocsectionformat{\bfseries}
\settocsubsectionformat{\mdseries}
\settocsubsubsectionformat{\mdseries}
\setsectionpnumformat{\bfseries}
\setsubsectionpnumformat{\mdseries}
\setsubsubsectionpnumformat{\mdseries}

\let\oldtableofcontents\tableofcontents
\renewcommand{\tableofcontents}{%
  \vspace*{-\linespacing}
  \oldtableofcontents}

\usepackage[bookmarks=true, bookmarksopen=true,%
bookmarksdepth=3,bookmarksopenlevel=2,%
colorlinks=true,%
linkcolor=blue,%
citecolor=blue,%
filecolor=blue,%
menucolor=blue,%
urlcolor=blue]{hyperref}

\usepackage{graphicx}
\usepackage{amssymb}
\usepackage{amsmath, amscd}
\usepackage{xcolor}

\newtheorem{thm}{Theorem}[section]

\newtheorem{theorem}[thm]{Theorem}
\newtheorem*{theorem*}{Theorem}
\newtheorem{proposition}[thm]{Proposition}
\newtheorem{corollary}[thm]{Corollary}
\newtheorem{lemma}[thm]{Lemma}

\theoremstyle{definition}
\newtheorem{definition}[thm]{Definition}

\theoremstyle{remark}
\newtheorem{remark}[thm]{Remark}

\newtheorem{example}[thm]{Example}

\numberwithin{equation}{section}

\newcommand{\D}{{\mathbb{D}}}
\newcommand{\R}{{\mathbb{R}}}
\newcommand{\C}{{\mathbb{C}}}

\newcommand{\Q}{{\mathbb{Q}}}
\newcommand{\Z}{{\mathbb{Z}}}
\renewcommand{\P}{{\mathbb{P}}}

\newcommand{\Mm}{{\mathcal{M}}}

\newcommand\bB{\mathbf{B}}

\newcommand\bD{\mathbf{D}}

\newcommand\bG{\mathbf{G}}

\newcommand\bN{\mathbf{N}}

\newcommand\bP{\mathbf{P}}

\newcommand\bU{\mathbf{U}}
\newcommand\bV{\mathbf{V}}
\newcommand\bW{\mathbf{W}}

\newcommand\bZ{\mathbf{Z}}

\newcommand{\M}{{\mathcal{M}}}

\newcommand{\ev}{\operatorname{ev}}

\newcommand{\pa}{\partial}

\newcommand{\ind}{\operatorname{index}}

\newcommand{\area}{\operatorname{area}}

\newcommand{\Sk}{\operatorname{Sk}}

\usepackage{xspace}

\usepackage[margin=1in,marginparwidth=0.8in, marginparsep=0.1in]{geometry}

\begin{document}

\title{Counting bare curves}
\author{Tobias Ekholm}
\address{Department of mathematics and Centre for Geometry and Physics, Uppsala University, Box 480, 751 06 Uppsala, Sweden \and
Institut Mittag-Leffler, Aurav 17, 182 60 Djursholm, Sweden}
\email{tobias.ekholm@math.uu.se}
\author{Vivek Shende}
\address{Center for Quantum Mathematics, Syddansk Univ., Campusvej 55
5230 Odense Denmark \and 
Department of mathematics, UC Berkeley, 970 Evans Hall,
Berkeley CA 94720 USA}
\email{vivek.vijay.shende@gmail.com}

\thanks{TE is supported by the Knut and Alice Wallenberg Foundation, KAW2020.0307 Wallenberg Scholar and by the Swedish Research Council, VR 2022-06593, Centre of Excellence in Geometry and Physics at Uppsala University and VR 2020-04535, project grant. \\ \indent  
VS is supported by  Villum Fonden Villum Investigator grant 37814, Novo Nordisk Foundation grant NNF20OC0066298, and Danish National Research Foundation grant DNRF157. 
}

\begin{abstract}
We construct a class of perturbations of the Cauchy-Riemann equations for maps from curves to a Calabi-Yau threefold.  Our perturbations vanish on components of zero symplectic area. For generic 1-parameter families of perturbations, the locus of solution curves without zero-area components is compact,  transversely cut out, and satisfies certain natural coherence properties.  

For curves without boundary, this yields a reduced Gromov-Witten theory in the sense of Zinger.   That is, we produce a well defined invariant given by counting only maps without components of zero symplectic area, and we show that this invariant is related to the usual Gromov-Witten invariant by the expected change of variables.

For curves with boundary on Maslov zero Lagrangians, our construction provides an `adequate perturbation scheme' with the needed properties to set up the skein-valued curve counting, as axiomatized in our previous work \cite{SOB}.

The main technical content is the construction, over the Hofer-Wysocki-Zehnder Gromov-Witten configuration spaces, of perturbations to which the `ghost bubble censorship' argument can be applied.  Certain local aspects of this problem were resolved in our previous work \cite{ghost}.  The essential work of the present article
is to ensure inductive compatibilities, despite the non-existence of marked-point-forgetting maps for the configuration spaces. 
\end{abstract}
\maketitle

\thispagestyle{empty}

\newpage
\small
\renewcommand{\contentsname}{}
\setcounter{tocdepth}{2}
\settocsectionvskip{2pt}
$\phantom{N}$ \vspace{-10mm}
\tableofcontents
\normalsize
\newpage

\section{Introduction}\label{sec:intro}

The moduli of pseudo-holomorphic stable maps from possibly nodal Riemann surfaces is not the closure of the locus of such maps from smooth surfaces.  There are other components, and, for enumerative problems, it is natural to ask whether contributions from the different components can be meaningfully separated, or, more generally, whether contributions from maps with various other geometric properties can be separated.  In particular, a holomorphic map $u\colon S \to X$ leads to infinitely many others obtained by composition with some ramified cover $S' \to S$, and/or attaching new irreducible components collapsed by the map.  We refer to such collapsed components as \emph{ghost bubbles}, and say that a map without ghost bubbles is {\em bare}.  Of particular interest is the case of Calabi-Yau 3-folds, where it is natural to consider separating these two contributions out of the Gromov-Witten invariants, leaving some more primitive invariant \cite{GV-MtopII, MNOP, Ionel-Parker-GV, Doan-Ionel-Walpulski, Pardon-MNOP}.

A starting point for approaching this problem is to ask whether, as the almost complex structure $J$ varies in a 1-parameter family, a family of $J$-holomorphic curves without any multiple covers or ghost bubbles can limit to a curve with these behaviors. In fact, multiple covers do generically appear in such families; see for instance \cite{Ionel-Parker-GV, Bai-Swaminathan}, and more sophisticated ideas are required to account for their contributions.  However, in situations where multiple covers are excluded a priori, e.g.\ for topological reasons, it is possible to exclude ghost bubble formation \cite{ionel-genus1, zinger-sharp, niu2016refined, SOB, doan-walpulski-embedded}.  

In \cite{SOB}, we used this idea in the context of a study of the wall-crossing behavior of counts of holomorphic curves in a Calabi-Yau 3-fold $X$ with Maslov zero Lagrangian boundary conditions $L$. Let us recall the ``ghost bubble censorship'' argument in this context. 

By a general position argument, one can show:  for generic $J$, if a $J$-holomorphic map $u\colon (S, \partial S) \to (X, L)$ is generically (on $S$) an embedding, then $S$ is smooth (i.e.\ no nodes), $u$ is an embedding, and $u(\mathrm{int}(S))$ is disjoint from $L$.  We call maps with these properties \emph{0-generic}.  Now consider maps which appear as $J$ varies in a general 1-parameter family.  One certainly expects isolated instances of the following: boundary nodes, general position double points at the boundary, and interior crossings with $L$.
We call 1-parameter families with only such degeneracies \emph{1-generic}.
Again, a general position argument establishes 1-genericity for families of maps which are all generically embeddings.  

The ghost bubble censorship argument excludes one additional situation: a degeneration in which a separating boundary node develops and the curve on one side of the node maps to a point.  Our version of this argument \cite{ghost} works as follows.  
Consider any 1-parameter family $J_t$ of complex structures, and a corresponding 1-parameter family of $J_t$ holomorphic maps $u_t\colon (S_t, \partial S_t) \to (X, L)$, with no ghost components away from $t=t_0$, such that $u_{t_0}\colon S_{t_0} \to X$ has a ghost component, and such that the non-ghost part $(u_{t_0})_+\colon (S_{t_0})_+  \to X$ is generically an embedding. One shows that in this case, $(u_{t_0})_+$ must exhibit some  singular behavior at the attaching point(s) of the ghost components, of a sort which does not appear in a 1-generic family.  But then the holomorphic map $(u_{t_0})_+$ could not have occured in a generic 1-parameter family $J_t$. 

That is, when multiple covers can be excluded, one can show that for generic 1-parameter families $J_t$, the locus of bare maps is closed.  We showed in \cite{SOB} that the remaining phenomena in 1-generic families
can be identified with the framed HOMFLYPT skein relations from quantum topology, and, consequently, that there is a deformation invariant count of open curves \emph{valued in the skein module} of the Lagrangian boundary condition.  (Counting curves in the skein module is a rigorous mathematical shadow of the idea that Lagrangian branes in the topological A-model should carry the Chern-Simons quantum field theory where holomorphic curves ending on them should introduce line defects \cite{Witten}.)  In fact, while most counting problems involve multiple covers, in \cite{SOB} we found one striking application where they could be excluded a priori, and established that the HOMFLYPT invariant of a knot is equal to a count of curves on a certain Lagrangian (depending on the knot) in the resolved conifold, as conjectured by Ooguri and Vafa \cite{OV}. 

More generally, in \cite{SOB} we established the existence of a skein-valued curve counting theory conditional on the existence of a class of perturbations satisfying certain axiomatics, including that solutions to the perturbed holomorphic curve equation could not be multiply covered.  Our axiomatics were such that when multiple covers can be excluded a priori, choosing generic $J$ already provides such a class. 

In the present article, we construct a class of perturbations satisfying the axiomatics of \cite{SOB}, for any fixed homology class of maps to a target Calabi-Yau 3-fold $X$ with (possibly empty) Maslov zero Lagrangian boundary condition $L$.
This completes the necessary foundations for our work  \cite{ekholm-shende-unknot, ekholm-shende-colored} establishing the full Ooguri-Vafa conjecture relating (what are before perturbation multiply covered) curve counts and colored HOMFLYPT invariants, several works \cite{scharitzer-shende, scharitzer-shende-2, hu-schrader-zaslow}
around skein-valued cluster algebra associated to Legendrian mutations, 
and our study of the open topological vertex \cite{ekholm-shende-unknot, ekholm-longhi-nakamura-hopf, Ekholm-Shende-Longhi}. 

Our previous work \cite{ghost} contains 
the necessary analysis locally on the moduli of maps.  Our task here is to patch it together.  
We  work in the Hofer-Wysocki-Zehnder polyfold setting, which we review in Section \ref{sec:polyfoldforopencurves}.  The main result of the paper is the following:

\begin{theorem} \label{main theorem intro}
    Let $X$ be a Calabi-Yau 3-fold, and $L \subset X$ a (possibly empty) Lagrangian submanifold of Maslov index zero. 
    Fix a class $d \in H_2(X, L)$, and choose some integer
    $\chi_{\min}$.  We write $\mathbf{Z}$ for the configuration space of maps realizing class $d$ with possibly disconnected domain of formal Euler characteristic $\ge \chi_{\min}$, see Section \ref{sec:stablemapmoduli} and $\mathbf{W}$ for the bundle of formal complex anti-linear differentials over $\mathbf{Z}$, see Section \ref{sec:formaldiff}.  

    Then there is a (weighed multi-)section $\lambda$ of $\mathbf{W} \to \mathbf{Z}$  with the following properties. For any stable map $(u,S)$, $\lambda(u)$ vanishes on any irreducible component of $S$ with zero symplectic area and there is a neighborhood $N(L)$ of $L\subset X$ such that $\lambda(u)$ vanishes at any point mapping to $N(L)$. All bare solutions to $\bar \partial_J u = \lambda(u)$ are 0-generic and transversely cut out, and the locus of bare solutions is compact. 

    Moreover, any two sections constructed by our methods may be connected by a 1-parameter family of sections, again vanishing on components of zero symplectic area and in $N(L)$,
    such that all bare solutions are 1-generic and transversely cut out, the locus of bare solutions is compact, and a bare curve is a solution if and only if its normalization is. 
\end{theorem} 

Theorem \ref{main theorem intro} is proved in Section \ref{sec: main res}.

\begin{remark}
    Theorem \ref{main theorem intro} concerns maps 
    \emph {of fixed degree} and makes no claims 
    about how the solutions which arise from perturbing 
    multiply covered curves relate with the original unperturbed multiple cover, or whether or not they stay together in families, etc.  
    
    Questions around multiple covers have been studied in many other works, most notably in the work of Ionel and Parker \cite{Ionel-Parker-GV} where they established the Gopakumar-Vafa integrality conjecture. Note however that their proof does  not (as one might have naively hoped) show how to identify the GV invariants with some count of simpler embedded curves and does not separate in families the contributions from embedded curves and multiple covers. We refer to \cite{Bai-Swaminathan} for further discussion and a detailed study in the case of double covers.
\end{remark}

While we have advertised Theorem \ref{main theorem intro} in terms of its foundational relevance to counting holomorphic curves with boundary, it has new consequences already for closed curves: 
    
    \begin{corollary}
        For $(X, \omega)$ a symplectic Calabi-Yau 3-fold, and $\beta \in H_2(X)$, there is a well defined invariant given by counting bare solutions to $\bar \partial_J u = \lambda(u)$: 
        \begin{equation}  \label{bare curve count defined}
            Z_{\mathrm{bare}}(\beta, z) := \sum_{\mbox{\tiny {$\begin{matrix}(u, S) \, \mathrm{bare}\\ u_\ast[S]=\beta\end{matrix}$}}} [\lambda \cap \bar \partial_J]_{(u, S)} \cdot z^{-\chi(S)}, 
        \end{equation}
    where $[\lambda \cap \bar \partial_J]_{(u, S)}$ is the (signed, rational) intersection number between the multi-section $\lambda$ and the $\bar \partial_J$ section.
    \end{corollary}     
    We view this as a differential-geometric definition of the (hitherto conjectural) all-genus reduced Gromov-Witten invariants in the sense of Zinger \cite{Zinger-Li-reduced}.  Such invariants are supposed to be defined by counting only the contributions of the component of the stable map moduli space arising as the closure of maps from smooth curves, and are supposed be related to usual Gromov-Witten partition function, in the language above
    \[
Z_{\mathrm{GW}}(\beta, g_s) := \sum_{\mbox{\tiny{$\begin{matrix}(u, S)\\ u_\ast[S]=\beta\end{matrix}$}}} [\sigma \cap \bar \partial_J]_{(u, S)} \cdot g_s^{-\chi(S)},
    \]
    where $\sigma$ is a perturbation for Gromov-Witten counts (non-trivial on area zero curves), 
    by the formula 
    \begin{equation} \label{zinger-li} Z_{\mathrm{red}}(\beta, e^{g_2/2}-e^{-g_2/2}) = Z_{\mathrm{GW}}(\beta, g_s)\end{equation} 
    In genus one, such invariants were constructed in \cite{Zinger-reduced, Zinger-reduced-CY} (in fact, by proving a ghost bubble censorship result), and  were
    conjectured to exist more generally in \cite{Zinger-Li-reduced}.  There is an algebraic approach \cite{Hu-Li-reduced, RMMP} to define a contribution from the closure of the smooth map locus, but the resulting invariant is only known to satisfy \eqref{zinger-li} in genus 1. 

Our bare curve counting captures only contributions near smooth curves, and we prove:  
\begin{theorem}[\ref{c : bare to GW}]\label{t : bare to GW intro}
        Let $(X,\omega)$ be a symplectic Calabi-Yau 3-fold. Then
        \begin{equation}\label{eq : bare to GW}
        Z_{\mathrm{bare}}(\beta,e^{g_s/2}-e^{-g_s/2}) = Z_{\mathrm{GW}}(\beta,g_s).
        \end{equation} 
    \end{theorem}
\begin{remark}
    The orientation of moduli spaces of holomorphic curves in $X$ depends on an initial choice of orientation on the complex plane $\C$. If the orientation of the complex plane is reversed the change of variables in \eqref{eq : bare to GW} should instead be 
    $
    iz = e^{ig_s/2}-e^{-ig_s/2}
    $.
\end{remark}

\vspace{2mm} \noindent
{\bf Acknowledgements.} We thank Shaoyun Bai, Eleny Ionel, Kenji Fukaya, Penka Georgieva, Melissa Liu, Cristina Manolache,  John Pardon, Brett Parker, Jake Solomon, and Aleksey Zinger for helpful discussions.

\section{Perturbations for bare curve counting}\label{sec : methods and difficulties}
This section contains an informal recollection of the properties of the Hofer-Wysocki-Zehnder perturbation scheme for holomorphic curve counting \cite{HWZ, HWZ-GW}, followed by an informal discussion of the modifications necessary for the bare curve counts that we construct in this paper. We briefly describe the problems we will face and the ideas used to solve them.  The reader who does not like informal discussions may freely skip this section entirely.

\subsection{The Hofer-Wysocki-Zehnder approach to transversality} 
Consider a symplectic manifold $X$ with a compatible almost complex structure $J$. We will have to perturb the Cauchy-Riemann equations $\bar\partial_J=0$ in the presence of multiple covered solutions. Various setups were developed for this purpose \cite{RT,HWZ-GW,FOOO,Pardon};
we will work in the Hofer-Wysocki-Zehnder `polyfold' framework \cite{HWZ}.   We do this for two reasons, both related to the desired applications to skein-valued curve-counting \cite{SOB}.  First, we will need to know that the perturbed solutions to the holomorphic curve equation are actually maps, not just abstract elements of some formal space the counting of which defines a virtual number or class, and that, moreover, said maps have sufficiently good regularity properties that the boundaries trace out curves along the Lagrangian.  
Second, we do not want to patch together data on different charts (as is done in approaches like \cite{FOOO,Pardon}).  Indeed,  while patching together fundamental classes is natural (taking chains commutes with homotopy colimits), a non-trivial discussion would be needed to patch together skein-valued counts.\footnote{Another setting in which the perturbations are explicit and global is provided by the `global Kuranishi charts' of \cite{Siebert, Abouzaid-McLean-Smith, Hirschi-Swaminathan}; these would plausibly also  provide a suitable home for our desired constructions.}  

In the Hofer-Wysocki-Zehnder setup, the moduli spaces of interest are cut out of infinite dimensional configuration spaces $\mathbf{Z}$, parameterizing not-necessarily-holomorphic maps from complex curves,
by sections of infinite dimensional bundles $\mathbf{W}\to\mathbf{Z}$, with fiber at $u\colon S \to X$ being an appropriate functional-analytic space of sections of $u^* TX \otimes \Omega^{0,1}_S$.  

The configuration spaces are built from 
Banach spaces, but the notion of differentiability is not the classical
notion. Instead, the Banach spaces are equipped with a filtration (typically the Banach space is a space of functions and the filtration is by regularity) and, at each filtration level, differentiability is required only with respect
to directions of strictly greater filtration level.  This notion is called sc-differentiability (`sc' is for scale), and satisfies usual properties of calculus.  Unlike in the usual calculus, sc-Banach spaces may admit 
retracts which are not themselves locally modeled on sc-Banach spaces; by idempotent completion, the operations
of calculus may be extended to such retracts \cite[Section 2]{HWZ}.
E.g., the tangent space to a retract is by definition the image of 
the differential of the retraction.  

The use of sc-calculus to study holomorphic curve theory has two key advantages. The first
is that reparameterization of holomorphic curves is an sc-smooth action, hence one can form good quotients.
The second is that the configuration space of maps from the family of curves giving a versal deformation of a nodal domain
is a retract of a sc-smooth space \cite[Section 2]{HWZ-models}.  
These properties allow the construction of a configuration space $\mathbf{Z}$ of maps from possibly nodal domains, to which the sc-calculus applies  \cite[Theorem 1.7]{HWZ-GW}.  

In this setting, the Cauchy-Riemann operator is a section $\bar \partial_J$ of the aforementioned infinite dimensional bundle $\mathbf{W}\to\mathbf{Z}$.  Because the configuration spaces $\mathbf{Z}$ already contains maps from nodal domains (and because the topology on $\mathbf{Z}$ restricts to the correct topology on the solution space), Gromov's compactness theorem asserts that $\bar\partial_J^{-1}(0)$ is compact.

To construct transversely-cut-out moduli spaces,
one perturbs by multi-sections $\lambda$ of the bundle $\mathbf{W}\to\mathbf{Z}$.  (Notationally, one considers $\lambda$ as a subset of the total space of $\mathbf{W}$, and correspondingly denotes the perturbed solutions as  
$\bar\partial_J^{-1}(\lambda)$.)  
The $\bar \partial_J$-section is shown in \cite[Theorem 4.6]{HWZ-GW} to be 
`sc-Fredholm' in the sense of \cite[Definition 3.8]{HWZ}.  
This property is an amalgalm of conditions; in particular, both those necessary to transfer desired consequences of usual notion of Fredholm through sc-retracts, and additionally a formal abstraction of elliptic bootstrapping.  
Additionally, as explained in \cite[Section 5]{HWZ}, there are key properties one can require of perturbations $\lambda$ to ensure desirable consequences: a formal regularizing condition `$\mathrm{sc}^+$', which ensures that the formal elliptic bootstrapping applies, and the property of being small in an `auxiliary norm', which  ensures that the compactness of the initial solution space $\bar\partial_J^{-1}(0)$ persists to $\bar\partial_J^{-1}(\lambda)$.  
Perturbations with both these properties are shown to exist in great generality in \cite{HWZ-perturb}.  However, as we clarify in the next subsection, this abstract existence result does not immediately produce perturbations suitable for our purposes, and we will construct our own by hand.  

We will need to generalize this setup to allow 
maps with Lagrangian boundary conditions.  
This can be done using 
standard techniques for adapting holomorphic curve arguments to allow boundary conditions (see e.g.\ 
\cite[Section 5.2]{ESS}).  With these modifications the original arguments of \cite{HWZ-GW} may be repeated verbatim; we give the statements and explain the necessary modifications to the proofs in Section \ref{sec:polyfoldforopencurves} below. 

\subsection{Our desired perturbations} \label{our perturbations}
We are not in fact interested in transverse sections $\lambda$ of $\mathbf{W} \to \mathbf{Z}$.  Instead, we want to produce a (multi-)section $\lambda$ with the property that for a map $u\colon S \to X$, the value of $\lambda(u)$ (recall, $\lambda(u)$ is a section of a bundle over $S$) vanishes on all components of $S$ with zero symplectic area, and, for other reasons, also on the preimage of a neighborhood  of the Lagrangian $L$.  For such $\lambda(u)$, any solution to $\bar \partial_J u = \lambda(u)$ will contract all zero area components.  

In order for the ghost bubble censorship argument (recalled  in Section \ref{sec:intro}) to apply, we also need two additional properties of the section $\lambda$.  

The first is that if a sequence of bare solutions to $\bar \partial_J u = \lambda(u)$ converges to a solution with ghost bubbles, the bare part of the limit has some singular behavior (that makes it non 1-generic) and  bare parts of solutions in 0- and 1-parameter families in general position are 0-generic and 1-generic, respectively.
(Here, the appropriate generalization of the 1-genericity conditions to perturbed holomorphic maps is to forbid certain singularity conditions on the holomorphic derivative $\partial_J u$.  Indeed, since we are solving a perturbed equation 
$\bar \partial_J u = \lambda(u)$, in general with $\lambda(u)$ supported where the bubble is forming, we cannot expect $du = 0$ at this point.)
This first condition is local in moduli, and 
in \cite[Section 10]{ghost} we constructed, locally in moduli, a class of perturbations that satisfy it. 

The second condition is not local in moduli. Consider a solution with ghost bubbles $u \in \bar \partial_J^{-1}(\lambda)$.  If no non-local-in-moduli conditions are imposed on $\lambda$, then the corresponding positive area map $u_+$ -- whose domain has different topology and so lies in a different component of moduli space --  need not be among the solutions $\bar \partial_J^{-1}(\lambda)$.  We will need to impose a condition ensuring that there is a corresponding solution $u_+' \in \bar \partial_J^{-1}(\lambda)$, near enough to $u_+$ that non-singularity of $u_+'$ ensures non-singularity of $u_+$.  

The main work of the present article consists in showing that there are $\lambda$ for which these two conditions hold simultaneously. 

\subsection{Absence of forgetful maps for perturbed solutions}

The aforementioned non-local constraint on $\lambda$ would most naturally be expressed in terms of maps between configurations spaces 
which `forget a marked point'.  Unfortunately, 
for the Hofer-Wysocki-Zehnder configuration spaces, these maps do not exist. 

The most basic obstruction to forgetful maps is common to any setting where one both perturbs constant spheres, and restricts attention to maps with the property that any component with vanishing symplectic area has finite automorphism group.  Indeed, let $S$ be a sphere with three marked points $0,1,\infty$, and consider a map $u\colon S \to X$ with vanishing symplectic area $\int_{S} u^* \omega = 0$, but where $\bar \partial_J u \ne 0$
(i.e., it solves a perturbed equation $\bar \partial_{J} u = \lambda$, where $\lambda$ is not identically zero).  Assume that the images $u(0)$, $u(1)$, and $u(\infty)$ are distinct points.  
Consider other maps $v_0\colon C_0 \to X$ and $v_1\colon C_1 \to X$, where the image of $v_j(C_{j})$ passes through $u(j)$, $j=0,1$.  Assume now that forgetful maps exist.  On the one hand, we should be able to attach $u$ to $v_0$ and $v_1$ to 
get a map $w\colon (C_0 \cup_0 S \cup_1 C_1) \to X$, where the domain is regarded as a curve with 
a single marked point $\infty$.  On the other hand, we should be able to forget the marked point $\infty$ for $w$. However, the 
domain resulting from forgetting $\infty$ is unstable, because of the irreducible area zero sphere component $S$ with only two marked points, 
and yet we cannot collapse $S$, since 
the points $0,1 \in S$ where $C_0$ and $C_1$ are attached have different 
images.  

The absence of forgetful maps for $\mathbf{Z}$, along with an approach to route around them, is discussed in detail in \cite{Schmaltz}.  (We do not use the approach of \cite{Schmaltz}, although it would be interesting to compare to our approach.)

\subsection{Approximate compatibility}
Our method to construct perturbations for which ghost bubble censorship holds is as follows.  
Consider a locally closed locus in $\mathbf{Z}$ with some fixed topological type of map; in particular, where the ghost components have some fixed topology, and are attached to the positive area components in some topologically fixed manner.  {\em On this stratum} there is a well defined forgetful map which forgets all nodes and ghost components.\footnote{Strictly speaking, forgetting the nodes is only defined on a locus in $\mathbf{Z}$ of maps satisfying certain stronger regularity conditions at nodes; we axiomatize the class of maps to which this distinction is irrelevant through our notion of `extendability' described in Section \ref{sec: extendability} below, and verify later this property for the perturbations we construct in Lemma \ref{lem: extendable}.} So, given a perturbation defined already in higher Euler characteristics, we can `extend' it in some neighborhood of such a stratum as follows.  On the stratum itself, we take the perturbation defined previously on the positive area part, and zero on the ghost components.  To extend to a neighborhood, we use (a version of) the gluing of perturbations defined in \cite{HWZ-GW}, and then cut off.  We do this along each stratum of nodal curves, and cut off using some function restricting to a partition of unity along the boundary, and sum the results. (See Remark \ref{r : multisections and partitions of unity} for a discussion of multi-sections and partitions of unity).  

The sum-of-stratawise-extensions described above does not  involve the use of the (non-existent) forgetting-a-marked-point maps.  As a consequence, the perturbation on a given map -- say with some ghosts -- is not generally the same as the corresponding perturbation on the positive part of the same map.   Let us explain in a simple example.  

\begin{example} \label{perturbation overlap example}
Begin with some bare map $u\colon S \to X$.  Form a map $v\colon T \to X$ by attaching a rational ghost bubble at some $p \in S$, and then some further (higher genus) ghost components to the bubble.  Consider a map $\widetilde{v}\colon \widetilde{T} \to X$ obtained by smoothing the node at $p$.  Note the positive area part of the domain of $\widetilde{v}$ is again just $S$, to which some higher genus ghost bubbles have been attached near $p$.  

Now suppose we had begun with some perturbation defined at Euler characteristic $\chi(S)$, so that $\lambda(u)$ is non-vanishing in a neighborhood of $p \in S$.  We write $\widetilde{u} := \widetilde{v}|_S$.  Now the extension method from the previous paragraph gives {\em two different contributions} to $\lambda(\widetilde{v})|_{S}$.  One, corresponding to $\widetilde{v}$ being obtained from $\widetilde{u}$ by just attaching ghosts, is just $\lambda(\widetilde{u})$.   
The second, corresponding to $\widetilde{v}$ being a smoothing of $v$, will be obtained by gluing $\lambda(u)$ to the zero perturbation on the rational bubble, which has the effect of cutting $\lambda(u)$ off to zero in a neighborhood of $p$, with radius (some monotonic function of) the gluing parameter $\epsilon$ at $p$, see Figure \ref{fig:Forget1}. 

\begin{figure}
    \centering
    \includegraphics[width=0.5\linewidth]{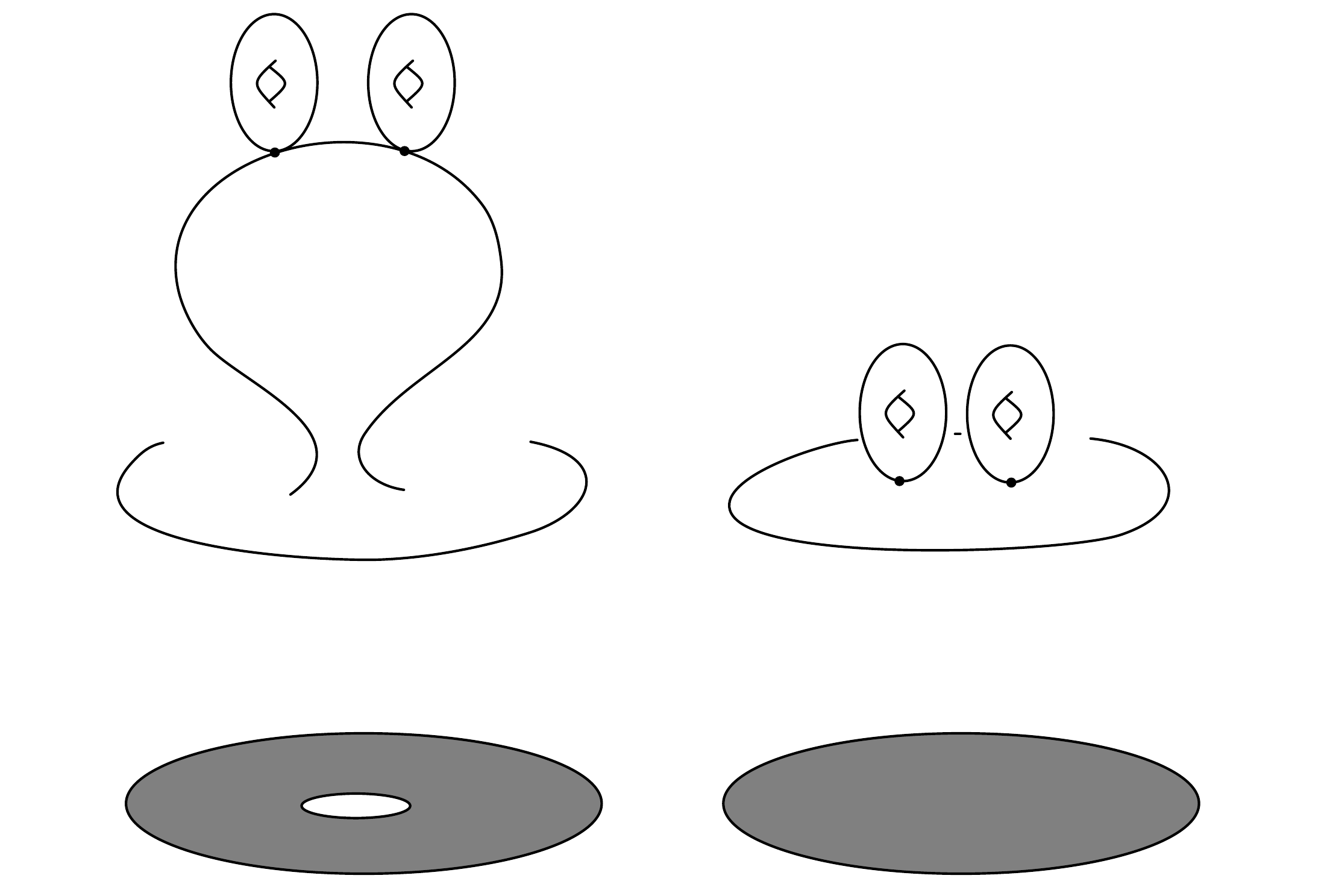}
    \caption{An original perturbation is extended to lower Euler characteristic in two different ways corresponding to different gluing parameter strata. The shaded regions below indicate the support of the corresponding perturbation, restricted to the domain of the positive area part of the new map.}
    \label{fig:Forget1}
\end{figure}

Let us write this result (just here) as 
$\lambda^\bullet(\widetilde u)$. Then we have 
\begin{equation} \label{toy perturbation interpolation}
\lambda(\widetilde{v})|_{S} = (1- \kappa(\epsilon)) \lambda(\widetilde{u}) + \kappa(\epsilon) \lambda^\bullet(\widetilde u),
\end{equation}
where $\kappa$ is a cutoff function in the gluing parameter at $p$.  In the region where $\kappa(\epsilon) \ne 0$, this differs from  $\lambda(\widetilde{u})$,
the perturbation we had seen before when we considered this map, so we cannot (yet) argue for ghost bubble censorship by induction. 
\end{example}

Our resolution of this difficulty is the following.  Since the bare solutions in higher Euler characteristic were transversely cut out, their genericity properties 
are preserved by sufficiently small perturbation.  That is, it is enough
to ensure that $\lambda(\widetilde{v})|_{S}$ can be made to be sufficiently close to $\lambda(\widetilde{u})$.  Now the point is that the (auxiliary) norm of the difference of $\lambda(\widetilde{v})|_{S}$ and $\lambda(\widetilde{u})$ is controlled by the radius of the disk where the perturbation is cut off, and consequently by $\epsilon$.  Thus, it is enough that $\kappa(\epsilon)$ goes to zero soon enough.  
From this sort of argument, we establish our main result (Theorem \ref{main theorem intro}).

\section{Curves, maps, and regularity conditions} \label{sec:basics} 
In this section we fix notation for domains, 
for stable maps, and for boundary degenerations.  We largely follow \cite{HWZ-GW}.

\subsection{Notation for functional analytic norms}
We write $\|f\|$ for the $L^2$-norm of functions, and more generally $\|f\|_k$ for the sum of $L^2$-norms of weak derivatives of order up to $k$.  We use the same notation for the corresponding norms of sections of orthogonal vector bundles over Riemannian manifolds $f\colon M\to E$. More generally, we use standard constructions of Sobolev norms for maps to non-linear targets, see e.g.\ \cite[Section 2]{Moser}. The precise value of such norms depends on choices of metrics on the spaces, bundles, and associated jet- bundles. For our purposes such choices are either immaterial or made explicitly by working in local coordinates and will therefore not be discussed.

We will also sometimes weight the $L^2$ norm by some function on $M$, which we indicate by notations like $\|f\|_{k, \delta}$, to be 
explained later.  When $f \in C^k(M, E)$, we also write $\| f \|_{C^k}$ for the supremum of 
the point wise operator norms of the derivatives of order up to $k$. 

We recall that the Sobolev embedding theorem implies that 
functions on a two-dimensional domain, with $k$ weak derivatives locally in $L^2$, are in fact $k-2$ times continuously differentiable. We also use the corresponding Sobolev inequality 
$ \| f \|_{C^{k-2}} < K\cdot \| f \|_{k}$, where $K$ is a constant depending only on the geometry used to define the norms.

\subsection{Bordered, noded, surfaces}
By a (smooth) Riemann surface, we mean a pair $(S, j)$ where $S$ is a two-dimensional compact real 
manifold, and $j$ is a smooth almost complex structure.  We do not require that $S$ is connected.
When we want to allow or emphasize that $S$ may have boundary, we say that it is a \emph{bordered} Riemann surface.  

For a connected Riemann surface, the topological Euler characteristic $\chi$ is related to the genus $g$ and number of boundary components $h$ by $\chi = 2-2g-h$.  In the disconnected case, we have 
$\chi=2c-2g-h$, where $c$ denotes the number of connected components. We will typically 
use the Euler characteristic as our measure of the complexity of a surface.  Note that surfaces of greater complexity have smaller (more negative) Euler characteristics.  

Our {\em definition} of a nodal surface will be just the data of a smooth surface, plus the data of which points are identified to form nodes.  More
precisely, we have the following.  

\begin{definition} \label{def:surface}
A noded, bordered Riemann surface is a tuple $(S, j, M_{\mathrm{in}}, M_{\mathrm{bd}}, D_{\mathrm{in}}, D_{\mathrm{hy}}, D_{\mathrm{el}})$ 
where $(S, j)$ is a bordered Riemann surface, $M_{\mathrm{in}}$ and $M_{\mathrm{bd}}$ are finite sets of  
points in the interior and boundary respectively (which may be ordered or unordered), 
$D_{\mathrm{in}}$ is a finite set of unordered pairs of interior points, $D_{\mathrm{hy}}$ a finite set of unordered pairs of boundary points, and $D_{\mathrm{el}}$ is a finite set of interior points.  There are no coincidences
among the points in $M_{\mathrm{in}}$, $M_{\mathrm{bd}}$, $D_{\mathrm{in}}$, $D_{\mathrm{hy}}$, and $D_{\mathrm{el}}$. 
\end{definition}

The points in $M_{\mathrm{in}}$ and $M_{\mathrm{bd}}$ we regard as interior and boundary marked points, respectively. 
The pairs in $D_{\mathrm{in}}$ are interior points which we think of as being joined to form a single interior node.  Likewise the pairs in $D_{\mathrm{hy}}$ are joined to form `hyperbolic' boundary nodes.  
Finally, the points in $D_{\mathrm{el}}$ are regarded as `elliptic' boundary nodes, i.e., boundary circles
which have shrunk to radius zero. For simpler notation, we sometimes abbreviate $(S, j, M_{\mathrm{in}}, M_{\mathrm{bd}}, D_{\mathrm{in}}, D_{\mathrm{hy}}, D_{\mathrm{el}})$ to $(S, j, M, D)$ or just $S$.

We write $\chi(S)$ for the topological Euler characteristic of a smoothing of $S$.  The topology on any moduli space of curves or maps is such that 
$\chi(S)$ is locally constant.  (As usual, we will often refer to Riemann surfaces as (holomorphic) curves.)

\begin{remark}
Let us record how to compute $\chi(S)$ in terms of the data defining $S$.   
If $\widetilde{S}$ is the disjoint union of smooth curves underlying $S$, then 
$\chi = \chi(\widetilde{S}) - 2|D_{\mathrm{in}}| - |D_{\mathrm{hy}}| - |D_{\mathrm{el}}|$
(Recall that $D_{\mathrm{in}}$ and $D_{\mathrm{hy}}$ are sets of pairs, and in the above formula $|D_{\mathrm{in}}|$ and $|D_{\mathrm{hy}}|$ count the number of pairs.)  
Indeed, smoothing an interior node either decreases the number of components by one or increases the genus by one, smoothing a hyperbolic boundary node either decreases the number of connected components by one or increases the number of holes by one, and that
smoothing a boundary elliptic node increases the number of holes by one.  
Finally, we may express $\chi(\widetilde{S})$ in terms of the genus $\tilde{g}$, number of holes $\tilde{h}$, and number of components $\tilde{c}$ as
$\chi(\tilde{S}) = 2\tilde c -2\tilde g -\tilde h$.
\end{remark}

\subsection{Stable curves and maps}
A smooth or nodal curve is \emph{stable} if it has no continuous family of automorphisms
preserving its marked points.  
For a smooth connected curve, 
with genus $g$ and $h$ boundary components, and $m$ interior and $b$ boundary marked
points, this amounts to requiring $4g + 2h + 2m + b \ge 5$.  The unstable irreducible closed curves
are tori without marked points, and  
spheres with at most two marked points.  There are the following unstable irreducible bordered curves: 
cylinders without marked points, and disks with either no marked points, or a single interior marked point, or one or two boundary marked points.

A disconnected curve is stable if its connected components all are. A nodal curve is stable if the underlying smooth curve, with  node data regarded as marked points, 
is stable. 

We consider continuous maps of bordered Riemann surfaces $(S,\partial S)\to (X,L)$ into 
a symplectic manifold $X$ with symplectic form $\omega$ and a Lagrangian submanifold $L\subset X$. The precise definition is as follows.

\begin{definition}
Let $S=(S, j, M_{\mathrm{in}}, M_{\mathrm{bd}}, D_{\mathrm{in}}, D_{\mathrm{hy}}, D_{\mathrm{el}})$ be a nodal, bordered Riemann surface. A map of $S$ is a continuous map $u\colon (S, \partial S) \to (X, L)$ such that: 
\begin{itemize}
\item if $\{x, y\}\in D_{\mathrm{in}}$ or $\{x, y\}\in D_{\mathrm{hy}}$ then $u(x) = u(y)$, and  
\item if $x \in  D_{\mathrm{el}}$ then $u(x) \in L$. 
\end{itemize}
\end{definition}

We say two maps $u$ and $u'$ of nodal Riemann surfaces $(S,j)$ and $(S',j')$ are equivalent if there is a $(j,j')$-holomorphic isomorphism $\phi\colon (S,j)\to (S',j')$, which respects the nodal decorations and marked points, such that $u=u'\circ \phi$.

\subsection{Ghosts and bare maps}
We introduce notation for the components of stable maps that are our main objects of study. Let $u\colon(S, \partial S) \to (X, L)$ be a map. 

\begin{definition} \label{def:stable}
We say that $(u,S)$ is \emph{numerically symplectic} if
for each component $C$ of $S$, $\int_C\omega\ge 0$.  A numerically symplectic map is \emph{stable} if $\int_C u^*\omega > 0$ for each unstable irreducible component
$C$ of the curve $S$. (Here nodes of $S$ on $C$ are regarded as marked points on $C$.)
When $\int_C u^*\omega = 0$, we say that $C$ is a \emph{ghost component} or \emph{ghost bubble}.
\end{definition}

\begin{definition} \label{def:bare}
We say that $(u,S)$ is \emph{bare} if it has no ghost components. 
We say that $(u,S)$ has \emph{constant ghosts} if the restriction of $u$ to any of its ghost components is constant. 
\end{definition} 

We  often  separate our maps into a `positive area part' and the rest.  
In our conventions for disconnected maps, it will be useful to allow the unique map from the empty domain to $(X, L)$. 

\begin{definition} \label{def: bare part} 
For a numerically symplectic map $u\colon (S, \partial S) \to (X, L)$: 
\begin{itemize}
\item We write
$u_+\colon (S_+, \partial S_+) \to (X, L)$ for the restriction of $u$ to the union $S_{+}$ of the irreducible components $C$ of $S$ for which $\int_{C} u^*\omega > 0$.  We leave punctures where said components were previously attached.   
We call $(u_+,S_{+})$ the \emph{positive area part} of $u$. 
\item We write $u_0\colon (S_0, \partial S_0) \to (X, L)$
for the restriction of $u$ to the union $S_{0}$ of irreducible components $C$ of $S$ with  $\int_{C} u^*\omega = 0$, with punctures where $S_{0}$ was attached to $S_+$ in $S$. We call $(u_0,S_{0})$ the \emph{ghost part} of $u$.
\end{itemize}
\end{definition}

\begin{lemma} \label{bare part larger Euler characteristic} 
Let $(u, S)$ be a stable numerically symplectic map, and $(u_+, S_+)$ its positive area part.  
Then $\chi(S) \le \chi(S_+)$ with equality if and only if $u$ is bare (i.e.\ $S = S_+$). 
\end{lemma}
\begin{proof}
By stability, deleting a component of symplectic area zero will increase the Euler characteristic of the domain.  
\end{proof}

Fix a strictly convex closed cone $P$ in $H_2(X,L)$ such that $\int_\beta\omega>0$ for all $\beta\in P$.

\begin{lemma}\label{l:Euler char from above}
	Consider an integer homology class $\beta\in P\subset H_{2}(X,L)$. 
	The Euler characteristic of $S$ is bounded above as $(u,S)$ ranges over stable numerically symplectic maps  
	in homology class $\beta$, and if $(u,S)$ achieves the maximum then $u$ is bare and $S$ has no nodes. 
\end{lemma}
\begin{proof}
By Lemma \ref{bare part larger Euler characteristic}, it suffices to consider bare maps.  Passing to the normalization only increases the Euler characteristic of the domain, 
so we may as well assume that the domain is a disjoint union of smooth curves $S=S_1\cup\dots\cup S_m$.  Let $a_0$ be the minimum positive value of $\omega$ on $P\cap\{\int_\beta \omega\le 2\int_S u^\ast\omega\}$. 
Then each component $u(S_k)$ has symplectic area $\ge a_0$ and Euler characteristic $\le 2$.  
Thus, 
\[
\int_{\beta}\omega = \int_S u^\ast\omega =\sum_{k=1}^m \int_{S_k} u^\ast\omega \ge m a_0  
\quad\text{ and }\quad 
\chi(S)=\sum_{k=1}^m \chi(S_k)\le 2m.  
\]
Therefore,
\[
\chi(S)\le \frac{2}{a_0} \int_\beta \omega. 
\]
The last statement follows from Lemma \ref{bare part larger Euler characteristic}.
\end{proof}

We will sometimes write $\chi_{\max}(\beta)$ for the maximum of the Euler characteristic of stable numerically symplectic maps in homology class $\beta$

For an almost complex structure $J$ on $X$, we say as usual that a map $u$ is \emph{$J$-holomorphic} if it satisfies the Cauchy-Riemann equation,
$\bar\partial_{J}u=\frac12(du +  J \circ du \circ j) =0$.  So long as $J$ is tamed by $\omega$, any $J$-holomorphic map $(u,S)$ is numerically symplectic and the ghost part $(u_{0},S_{0})$ of any $J$-holomorphic map is constant.

\subsection{Holomorphic polar coordinates}
In this section we discuss local coordinate conventions and associated elementary regularity properties that will be used when building polyfold neighborhoods of stable maps in later sections. 

The domains of our maps are (possibly disconnected, nodal) bordered Riemann surfaces $(S, \partial S)$ in the sense of Definition \ref{def:surface}.  
We will often want to work in local coordinates.  
We fix two parameterizations of the punctured unit disk: 
$$\D^\circ =  \{z\in\C\colon 0 < |z| \le 1\},$$ 
\begin{alignat}{4}\label{eq:exppolar}
&[0, \infty) &&\times \R / \Z  \to  \D^\circ,  \qquad\qquad &&(-\infty, 0] &&\times \R / \Z\to  \D^\circ,  \\\notag
&(s, t)  &&\mapsto \quad e^{- 2 \pi (s + i t)}   \qquad\qquad    &&(s, t) &&\mapsto \quad e^{2 \pi (s + i t)}. 
\end{alignat} 
Restricting the maps to $[0,\infty)\times [0,\frac12]$ and $(-\infty,0]\times[-\frac12,0]$, respectively, we get parameterizations of the punctured half-disk $\D_{+}^{\circ}=\left\{z\in\C\colon 0<|z|<1,\mathrm{Im}(z)\ge 0\right\}$. For convenient notation below we will often write 
\begin{equation}\label{eq:notationI}
I=\R/\Z,\qquad I=[0,\tfrac12],\qquad \text{or }\quad I=[-\tfrac12,0],
\end{equation}
to cover all these cases simultaneously.

Consider an interior (respectively boundary) point $p \in S$ where $S$ is a Riemann surface.  By \emph{holomorphic polar coordinates around $p$}
we mean the choice of a holomorphic map $(\D, 0) \to (S, p)$ (respectively $(\D_{+},\partial\D_{+} 0) \to (S,\partial S, p)$), which is composed with one of the parameterizations in \eqref{eq:exppolar} 
to give a map $[0, \infty) \times I \to S \setminus p$, or 
$(-\infty, 0] \times I \to S \setminus p$.

\subsection{Regularity conditions} \label{sec:regularity} 

We recall the regularity conditions of \cite[Definition 1.1]{HWZ-GW}:

\begin{definition}  \label{def:regularity}
We say a map $u\colon(S, \partial S) \to (X, L)$ is: 
\begin{itemize}
 \item of class $m$ at a point $z \in S$ if it has $m$ weak derivatives in $L^2_{\mathrm{loc}}$ near $z$.   
  \item of class $(m, \delta)$, $\delta\in (0,2\pi)$ at a point $z \in S$ if, in holomorphic polar coordinates
  in some punctured disk neighborhood of $z$ and smooth local coordinates $(\R^{2n},0)\to (X, u(z))$, the map $u\colon [0, \infty) \times I \to \R^{2n}$ 
  has $m$ weak derivatives in $L^2_{\delta}$, where $L_{\delta}^{2}$ is $L^{2}$ weighted by $e^{\delta s}$, i.e. 
  $$\int_{[0,\infty)\times I}|D^{\mathbf k}u|^{2}e^{2\delta s} \,dsdt<\infty,$$ 
  for all multi-indices $\mathbf k$ with $| \mathbf k|\le m$.
\end{itemize}
For the marked points on $S$, we will further distinguish them into two types: \emph{marked points} where we require the map to be of class $m$, and
and \emph{punctures} where we require the map to be of regularity $(m, \delta)$.  We say that a map $u$ is of class $(m, \delta)$ if it is of class $(m, \delta)$ at nodes and punctures, and of class $m$ elsewhere. 
\end{definition}

\begin{remark}\label{r : delta weight well def}
The notion of class $(m,\delta)$ is independent of holomorphic coordinates.  Indeed, such coordinates correspond to maps 
$\phi\colon [0,\infty)\times I\to\C$ of form
$\phi(z) = \sum_{k<0} c_k e^{2\pi k(s+it)}$, where $c_1\ne 0$. Thus after rescaling that affects integral norms also by rescaling we have 
\[
\phi(z) = e^{-2\pi(s+it)} + \sum_{k\le 2} c_k e^{2\pi k (s+it)} = e^{-2\pi (s+it)} + \mathcal{O}(e^{-4\pi s}).
\]
Taking logarithms we find that two local coordinates $(s,t)$ and $(\sigma,\tau)$ are related by the following change of coordinates in a neighborhood of $\infty$:
\[
(\sigma,\tau) = (s+\gamma_1(s,t),t+\gamma_2(s,t)), \qquad |\gamma_j(s,t)|_{C^3}=\mathcal{O}(e^{-2\pi s}).
\]
Consider now a map $v\colon [\sigma_0,\infty)\times I\to\R^{2n}$ in $(\sigma,\tau)$-coordinates. The corresponding map in $(s,t)$-coordinates is $u(s,t) = v(s+\gamma_1(s,t),t+\gamma_2(s,t))$. Then
\begin{align*}
    \int_{[s_0,\infty)\times I} u^2(s,t)e^{2\delta s}\,dsdt &= 
    \int_{[s_0,\infty)\times I} v^2(s+\gamma_1(s,t),t+\gamma_2(s,t))e^{2\delta s}\,dsdt \\
    &\le 
    \int_{[\sigma_0,\infty)\times I} v^2(\sigma,\tau)e^{2\delta(\sigma - \mathcal{O}(e^{-2\pi\sigma}))}\,(1+\mathcal{O}(e^{-2\pi\sigma}))d\sigma d\tau \\
    &= \mathcal{O}(1)\cdot 
    \int_{[\sigma_0,\infty)\times I} v^2(\sigma,\tau)e^{2\delta\sigma}\,d\sigma d\tau.
\end{align*}
Derivatives of $u(s,t)$ are estimated in a similar way. Thus
composition with the holomorphic coordinate changes of the domain near the puncture intertwines the corresponding $H^{3,\delta}$ spaces.
\end{remark}

\begin{remark}
The Sobolev embedding theorem implies that wherever a map $u$ is of class $m \ge 3$, it 
 is continuously differentiable.  This is one reason why $m \ge 3$ is used in applications, in particular in \cite{HWZ-GW} 
which we will later follow.  For class $(3, \delta)$ maps, we note below in Lemma \ref{l:H3deltatoC0} that at the puncture (or node) itself, one retains some H\"older continuity $C^{0, \alpha}$, $\alpha<\frac{\delta}{2\pi}$.  
\end{remark}

\begin{remark}
	The Fourier expansion of a map $u\colon [0,\infty)\times S^1\to\R^{2n}$ in $L^{2}_\delta$ with an interior node at $0\in\R^{2n}$ is  
    \[
    u(s,t)=\sum_{k<0} c_k(s)e^{2\pi i kt}.
    \]
    Here $2\pi k$, $k\in\Z$, are the eigenvalues of the operator $i \partial_t$ which are then also the critical weights of $\bar\partial = \partial_s + i\partial_t$, or more generally for the linearized $\bar\partial_J$-operator $L\bar\partial_J$ for general $J$ with $J(0)$ equal to the standard complex structure on $\R^{2n}\approx\C^n$, i.e., the weights where  the Fredholm index of the linearized operator $H^1_\delta\to L^2_\delta$ jumps.
    At a boundary node $u\colon [0,\infty)\times [0,1]\to\R^{2n}$ with boundary values in $\R^n\subset\R^{2n}\approx\C^n$ the operator $i\partial_t$ has eigenvalues $\pi k$. To uniformize notation below we take the circle to have length $1$ and the interval to have length $\frac12$ which gives eigenvalues $2\pi k$ in both cases, and we will use weights $\delta\in(0,2\pi)$ below, unless otherwise stated.
\end{remark}

We next compare the regularity requirements $m$ and $(m,\delta)$. Let $s+it\in[0,\infty)\times I$.  The case $(-\infty,0]\times I$ is directly analogous and obtained by the change of variables $s+it\mapsto -(s+it)$ in what follows below.  
Let $z=e^{-2\pi(s+it)}$ as in \eqref{eq:exppolar} be coordinates in a neighborhood $D$ of $p\in S$. Consider a map $u\colon D\setminus\{p\}\to\R^{2n}$ and let $u_{\infty}\colon [0,\infty]\times I\to\R^{2n}$ and $u_{0}\colon \D^{\circ}\to\R^{2n}$ be the corresponding maps in local coordinates. Take the weight $\delta\in(0,2\pi)$. We have the following basic relations between regularity requirements. 

\begin{lemma}\label{l:markedvspuncturenorms}
The integral norms are related as follows:
\begin{align}\label{eq:ref pull-back L^2} 
\int_{D}|u_{0}|^{2} \,dxdy  \ &= \ \int_{[0,\infty)\times I} |u_{\infty}|^{2} e^{-4\pi s} \,dsdt,\\\label{eq:push forward L^2}
\int_{[0,\infty)\times I}|u_{\infty}|^{2}\, dsdt \ &= \ \int_{D} |z|^{-2}|u_{0}|^{2} \,dxdy.
\end{align}
and
\begin{align}\label{eq:ref pull-back L^2 on derivatives}
\int_{D}|D^{(k)}u_{0}|^{2} \,dxdy \ &\asymp \ \sum_{j=1}^{k}\int_{[0,\infty)\times I} |D^{(j)}u_{\infty}|^{2} e^{4\pi(j-1)s}\,dsdt,\\\notag 
\int_{[0,\infty)\times I}|D^{(k)}u_{\infty}|^{2}\,dsdt \ &\asymp \ 
\sum_{j=1}^{k}\int_{D} |z|^{j-2}|Du_{0}^{(j)}|^{2} \,dxdy, 
\end{align}
where for positive functions $a,b$, we write $a \asymp b$ to mean that there exists constants $c,C>0$ such that $c a(x)< b(x)< Ca(x)$ for all $x$.
\end{lemma}

\begin{lemma}\label{l:compareL2}
Assume that $u_{\infty}$ and $Du_{\infty}$ both lie in $L^{2}_{\delta}$ then the following hold.
\begin{itemize}
\item[$(i)$] $u_{0}$ and $Du_{0}$ both lie in $L^{2}$.
\item[$(ii)$] For any $|\alpha|\ge 2$, if $D^{\beta}u_{0}$ lies in $L^{2}$ for $|\beta|\le|\alpha|$ then $D^{\alpha}u_{\infty}$ lies in $L^{2}_{\delta}$. 
\end{itemize}	
\end{lemma}

\begin{proof}
We have 
\begin{alignat*}{2}
d \log z &= \frac{dz}{z} = - 2 \pi (ds + i dt),\qquad
z \partial_{z} &&= - \frac{1}{2\pi} \left(\partial_{s} 
- i \partial_{t} \right),\\
\partial_{s} &= 
 - 2 \pi \left( z\partial_{z}  + \bar z \partial_{\bar z}\right),\qquad\qquad
\partial_{t} &&= - 2 \pi i \left( z \partial_{z}  - \bar z \partial_{\bar z} \right),
\end{alignat*}
Also $|z| = e^{-2 \pi s}$ and so $e^{\delta s} = |z|^{-\delta/2\pi}$. Thus, 
$$
\int_{[0,\infty)\times I}|u_{\infty}|^{2}e^{2\delta s} \,dsdt \ \sim \
\int_{|z| < 1} |u_{0}|^2 |z|^{-2\left(\frac{\delta}{2\pi} \right)}  |z|^{-2} \,dz d \bar{z}
$$
and, if $|\alpha|>0$,
$$
\int_{[0,\infty)\times I}|D^{\alpha}_{s,t} u_{\infty}|^{2}e^{2\delta s} \,dsdt \ \sim \
\sum_{|\beta|=1}^{|\alpha|}
\int_{|z| < 1} |z|^{2 |\beta|} |D^{\beta}_{z, \bar{z}} u_{0}|^2 |z|^{-2\left(\frac{\delta}{2\pi} \right)}  |z|^{-2} \,dz d \bar{z},  
$$
where $\sim$ means up to multiplication by constant.
The result follows.
\end{proof}

Using the Sobolev embedding theorem we get more information. 
\begin{lemma}\label{l:H3toH3delta}
Let $m\ge 3$. If $u_{0}$ has $m$ derivatives in $L^{2}$ then $u_{\infty}$ has $m$ derivatives in $L_{\delta}^{2}$.
\end{lemma}

\begin{proof}
By Lemma \ref{l:compareL2} $(ii)$, it suffices to estimate the $L^{2}_{\delta}$-norm of $u_{\infty}$ and $Du_{\infty}$. By Sobolev embedding, $u_{0}$ lies in $C^1(\D)$ and we have $|Du_{0}| < C$ for some constant $C$. Therefore, $|u_{0}(z)| < C|z|$ and thus, 
$$\int |u_{\infty}(s+it)|^2 e^{2 \delta s}\, ds dt \ < \ \int C e^{- 4 \pi s} e^{2 \delta s}\, ds dt \ = \ 
C \int e^{2 (\delta - 2\pi)s}\, ds dt  \ < \ \infty $$
and 
$$ 
\int \left|Du_{\infty} \right|^2 e^{2 \delta s}\,  ds dt \ \sim \  
\int \left(\left |z \partial_{z} u_{0} \right| +  \left|\bar z \partial_{\bar z} u_{0}\right| \right) ^2 
e^{2 \delta s}\,  ds dt  \ \sim \ \int e^{-4\pi s} e^{2 \delta s}\, ds dt \ < \ \infty.$$
\end{proof}

For $m \ge 3$, as will always be the case in our applications,
Lemma \ref{l:H3toH3delta} implies that  a stable map $(u,S)$ of class $m$ with a marked point at some $z \in S$ is always also a stable map of class $(m,\delta)$ with a puncture at $z \in S$.  When we wish to view such $(u,S)$ in the second way rather than the first, we say that we \emph{consider the marked point as a puncture}.

In the direction opposite to that of Lemma \ref{l:H3toH3delta} we have the following.
 
\begin{lemma}\label{l:H3deltatoC0}
Let $m\ge 3$. If $u_{\infty}$ has $m$ derivatives in $L_{\delta}^{2}$ then $u_{0}$ has two derivatives in $L^{2}$. Furthermore, if $u_{0}$ is extended by $u_{0}(0)=0$, then $u_{0}\in C^{0,\alpha}$ (i.e., $u_{0}$ is H{\"o}lder continuous with exponent $\alpha$) for any $\alpha<\frac{\delta}{2\pi}$.
\end{lemma}
\begin{proof}
Since $u_{\infty}(s+it)\cdot e^{\delta s}$ has three derivatives in $L^{2}$, it lies in $C^{1}$ and $|D(u_{\infty}(s+it)\cdot e^{\delta s})|<C$. Let $0<\epsilon<\delta$. Then for all $s$ sufficiently large, if $|u_{\infty}(s+it)\cdot e^{\delta s}|^{2}> e^{2\epsilon s} $ then
\[ 
\int_{[s,s+1]\times S^{1}}|u_{\infty}|^{2}e^{2\delta}\,dsdt \ \sim \ e^{2\epsilon s} \ \to \infty.
\] 
We conclude that for $s$ such that $e^{\epsilon s}>C$ we have $|u_{\infty}(s+it)\cdot e^{\delta s}|\le e^{\epsilon s}$ and hence
\[ 
|u_{0}(z)|=\mathcal{O}\left(|z|^{\frac{1}{2\pi}(\delta-\epsilon)}\right).
\]
The lemma follows.  
\end{proof}

\begin{remark} 
    Lemma \ref{l:H3deltatoC0} implies that if we have a map from a domain ``with punctures'', i.e. points that the map is required to be of regularity $(m, \delta)$, then the underlying map is at least continuous, though not necessarily of regularity $m$.  This means that when we consider moduli, if we attempt to ``forget a puncture'', e.g. to forget the punctures on the bare/ghost components in Definition \ref{def: bare part}, we will need to be careful that we do not violate our prescribed regularity conditions in the process.  This causes the issues discussed and resolved in Definition \ref{def: extendable} and Lemma \ref{lem: extendable} below. 
\end{remark}

\section{The configuration space of maps from curves with boundary}\label{sec:polyfoldforopencurves} 
Hofer, Wysocki, and Zehnder have explained in \cite{HWZ-GW} how to construct the closed Gromov-Witten moduli spaces  in their general `polyfold' framework. In this section we review some aspects of their construction and explain how to adapt their arguments to the case of curves with boundary in a Lagrangian.  We concern ourselves mostly with describing  the  local models of gluing at nodes, in part because this is one of the more fundamental and novel parts of the approach, but also because our construction of perturbations will rely on the coordinate charts introduced by this description.  

It is our impression that the adaptation of the \cite{HWZ-GW} to the case of curves with boundary is straightforward, and as such we have been very brief in our treatment and indicated only the (rather minor) modifications needed to the original arguments.  The reader wanting a detailed treatment may be interested to consult \cite{jemison}. 

\subsection{Gluing at interior nodes}\label{ssec:gluing} 
Let us recall the 
description of normal coordinates to a stratum of nodal curves in Deligne-Mumford space in the polyfold formalism, see \cite[Section 2]{HWZ-GW}.  For a nodal curve $S$, choose holomorphic polar coordinates near each node.  
Fix attention on one such node, with charts  $C^\pm = \R^\pm \times S^1$. 

To describe domain curves near $S_0$, one fixes once and for all the `gluing profile' (choice of parameterization of the gluing length) $\varphi(r)=e^{\frac{1}{r}-1}$, 
For $a= r e^{i\theta}$, one writes $R = R(a) = \phi(r)$, 
and defines $Z_a^+ = [0, R] \times S^1 \subset \R^+ \times S^1 = C^+$ and $Z_a^- = [-R, 0] \times S^1 \subset \R^- \times S^1 = C^-$.  
We fix an identification $Z_a^+ \sim Z_a^-$ by 
$(s, t) \sim (s' + R, t' + \theta)$.  
 
One defines the gluing of domains: 
\begin{equation} \label{domain gluing}
S_a := (S_0 \setminus (C^\pm \setminus Z_a^\pm)) / (Z_a^+ \sim Z_a^-)
\end{equation}
This construction identifies $a$ as a local coordinate on moduli space along the normal direction to the locus where the node persists. 

We now recall from \cite[Section 2.4]{HWZ-models} and \cite[Section 2.4]{HWZ-GW} the corresponding description of gluing maps, or in other words, of normal coordinates to the space of maps from nodal curves.  Consider a map $u_0\colon S_0 \to X$ and a point $\zeta_0\in S_0$. In suitably chosen holomorphic polar coordinates around $\zeta_0$ and local coordinate system in the target $X$, one arrives at a local description near a node as a pair of maps from semi-infinite cylinders:
$$u^\pm\colon C^\pm \to \R^{2n}.$$
By definition of maps from a nodal curve, these maps should have well defined and common asymptotic constant $c \in \R^{2n}$: 
$$\lim_{s \to \infty} u^+(s, t) = c = \lim_{s \to -\infty} u^-(s, t).$$ 
In the eventual application to giving charts on moduli, the maps will be taken to be of class $(3, \delta)$, i.e., $u^+ - c \in H^{3, \delta}(C^+, \R^{2n})$, and similarly for $u^-$, where the topology of the mapping space is given by $H^{3,\delta}(C^\pm)\times \R^{2n}$ .  We write $C_0 = C^+ \sqcup C^-$.  

For $a\in \D$, we (temporarily) write  $\mathrm{Maps}(S_a,X)$ for the space of maps $S_a\to X$. 
The identification \eqref{domain gluing} presents $\mathrm{Maps}(S_a, X)$ as a subquotient of $\mathrm{Maps}(S_0, X)$: take the subset of maps which match under the identification $Z_a^+ \sim Z_a^-$, and forget what the map does along $C^\pm \setminus Z_a^\pm$.  The gluing map is an endomorphism $\oplus_a$ of $\mathrm{Maps}(S_0, X)$ which is a linear idempotent projection to a subset which maps bijectively to said subquotient.  The endomorphism $\oplus_a$ does not change the maps outside the region described in coordinates, and is the identity for $a =0$.  It is given by interpolating between $u^+$ and $u^-$ along $Z^+\sim Z^-$, and cutting off to zero outside.  

Let us recall the formula for $\oplus_a$.  Make a universal choice of cut off function $\beta\colon \R\to[0,1]$ with the following properties: $\beta(s)=1$ for $s<-1$, $\beta'(s)\le 0$ for $-1\le s\le 1$, and $\beta(s)+\beta(-s)=1$.  
We denote a translate of this cut-off function by $\beta_{a}(s)=\beta(s-\frac12R)$, where $R=\varphi(a)$.
Then
\begin{equation} \label{gluing of maps formula}
(u^{+}\oplus_{a} u^{-})(s,t) :=
\beta_{a}(s)u^{+}(s,t)+(1-\beta_{a}(s))u^{-}(s,t) 
\end{equation}
Here, $(s, t)$ are the coordinates on $Z_a^+$; recall these are identified by coordinate on $Z_a^-$ by $(s, t) = (s'+R, t'+\theta)$.  The function $(u^+ \oplus_a u^-)$ vanishes outside $Z_a^+ \sim Z_a^-$, either by definition, or by continuing the coordinates $(s, t)$ on the regions $C^{\pm}$ in the natural way.

By this construction, one can define what the neighborhood in moduli of a given map from a nodal curve is: all maps which are obtained by taking the image of $\oplus_a$ and reinterpreting the resulting map having as domain $S_a$.  

The following result is fundamental:

\begin{theorem} \label{lem: retract} 
\cite[Theorem 1.28]{HWZ-models}
	For appropriate spaces of maps, the idempotent
    \begin{eqnarray*}
    \pi:  D \times H^{3, \delta}(C_0, \R^{2n}) & \to & D \times  H^{3, \delta}(
    C_0, \R^{2n})  \\
	(a,(u^{+},u^{-})) &\mapsto & (a, u^+ \oplus_a u^-)  
    \end{eqnarray*}
	is sc-smooth. 
\end{theorem}

Theorem \ref{lem: retract} asserts that the  image of $\pi$ is  an `sc-retract' (in fact, of a special class of such called `sc-splicings'), which is the sort of space to which the methods of \cite{HWZ} have extended the differential calculus.  This result is the key local ingredient in  the construction of the `configuration spaces' of \cite{HWZ-GW}. 
Theorem \ref{lem: retract}  is  proven by direct estimates in terms of the formula for $\oplus_a$.  The essential content of the assertion is at $a=0$. 

To generalize Theorem \ref{lem: retract} to the bordered case, we define below analogous maps $\oplus_a$ by sufficiently similar formulas that the proof of \cite[Theorem 1.28]{HWZ-models} can be repeated verbatim.   

\subsection{Stable maps with boundary conditions}

\begin{definition}
	We write $\bZ^{(k, \delta)} = \bZ^{(k, \delta)}(X, L, \omega)$ for the set of stable maps $u: (S, \partial S) \to (X, L)$, where: 
    \begin{enumerate}
        \item $(S, \partial S)$ is a (possibly bordered, noded, disconnected) Riemann surface as in Definition \ref{def:surface}. 
        \item The map $u$ has regularity conditions $(k, \delta)$ (see Definition \ref{def:regularity}) and is  numerically symplectic (see  Definition \ref{def:stable}).
        \item The map $u$ is  $J$-holomorphic along the
        boundary.  
    \end{enumerate}
	We write  $\bZ_{A, g,h,m,b} = \bZ^{(k, \delta)}_{A, g,h,m,b}(X, L, \omega)$ for the subset consisting
	of maps with image of class $A \in H_2(X, L)$, domain of genus $g$ with $h$ boundary components, $m$ interior marked points, and $b$ boundary marked points. 
\end{definition} 

\begin{remark}
    Let us recall why it is natural to demand $J$-holomorphicity
    along the boundary on the functional-analytic space of maps where we will eventually impose the Cauchy-Riemann equation.  The point is to arrange that the $n$-dimensional Lagrangian boundary condition gives $n$ local coordinate functions with  Dirichlet boundary condition (the fiber directions along $L$ in $T^\ast L$) and $n$ coordinate functions with Neumann boundary conditions (the directions along the zero-section $L\subset T^\ast L$).  Such boundary conditions enusre the elliptic estimates in the bordered case are directly analogous to the interior case, see e.g., \cite[Section 5]{Taylor} for the basics and \cite[Appendix B]{McDuffSalamon} for the adaption to Lagrangian boundary conditions. 

    Now, for a map $u\colon(S,\partial S)\to (X,L)$, the Dirichlet conditions are automatic. The Neumann condition is equivalent to asking that the map is $J$-holomorphic along the boundary: if $\tau$ is the tangent vector of $\partial S$ and $\nu$ the normal and if $u=(q,p)\in T^\ast L$, where $(q,p)$ are adapted to $J|_L$, then along the boundary $\partial_\tau p=0$, so holomorphicity along the boundary amounts to $\partial_\nu q = -J\partial_\tau p = 0$. 
\end{remark}

\subsection{Gluing maps} \label{ssec : bordred maps general}
In this section we adapt the $\oplus_a$ gluings to the bordered case.
\begin{remark}
In the following constructions we will need cutoff functions.  In order that these cutoffs preserve the Lagrangian boundary condition and boundary $J$-holomorphicity,  it is sufficient that the normal derivative along the boundary of the cut off function vanishes. We will assume that the cut-off functions we use will meet this condition whenever needed.  (Such functions are easy to arrange.)       
\end{remark}

\subsubsection{Hyperbolic nodes}  The gluing setup for hyperbolic boundary nodes is almost identical to that for interior nodes.  
We parameterize a  neighborhood of a hyperbolic node by
\[ 
C^+ \sqcup C^- = \R^{+}\times[0,1] \ \sqcup \ \R^{-}\times[0,1]. 
\]
We define the loci $Z_a^{+} = [0, R(a)] \times [0,1]$, and similarly $Z_a^-$; we identify coordinates 
$(s, t)$ on $Z_a^+$ with
coordinates $(s', t')$ on $Z_a^-$ by $(s, t) = (s'+R, t')$.

Given a pair of maps $(u^{+},u^{-})$,
\[ 
u^{\pm}\colon\left(\R^{\pm}\times[0,1],\partial(\R^{\pm}\times[0,1])\right)\to (\R^{2n},\R^{n}),
\]
we define $u^{+}\oplus_{a} u^{-}$ again to be the identity when $a=0$, and by the formula \eqref{gluing of maps formula} otherwise.  Note the result defines a map 
$(Z_{a},\partial Z_{a})\to(\R^{2n},\R^{n})$

The analogous result to Theorem \ref{lem: retract} then holds, as can be seen e.g.\ by doubling across $\R^n$ and passing to a $\Z/2$-invariant subset in the proof. 

\subsubsection{Elliptic nodes} \label{sec: elliptic gluing} 
We next consider the case of an elliptic node. As in the case of an interior node we consider a map $u_0\colon S_0 \to X$ and a point $\zeta_0\in S_0$, but here we assume that $u(\zeta_0)\in L$. The elliptic node can be thought of as usual node invariant under complex conjugation. We use this to define our polyfold chart. We parameterize the elliptic node by
$C^+=[0,\infty)\times S^1$ and, in order to see the action by complex conjugation, we use the reference domain $C^-=(-\infty,0]\times S^1$.  

In suitably chosen holomorphic polar coordinates around $\zeta_0$ and local coordinate system in $(X,L)$, there is a local description of the nodal map as $u^+\colon C^+ \to \R^{2n}$ with asymptotic constant $c \in \R^{n}\subset\R^{2n}$, where $\R^n$ corresponds to $L$ in the local coordinates. 

Consider complex conjugations in the source,  
$$
\tau\colon C^+\sqcup C^-\to C^+\sqcup C^-,\qquad \tau(s,t)=(-s,t),
$$ 
and target, 
$$
{ }^\dagger\colon\R^{2n}\to\R^{2n},\qquad (x_1,y_1,\dots,x_n,y_n)^\dagger=(x_1,-y_1,\dots,x_n,-y_n),
$$
respectively. Define the map conjugate to $u^+\colon C^+\to\R^{2n}$ as $u^-\colon C^-\to\R^{2n}$, where
$$
u^-(s,t)={{u^+}(\tau(s,t))}^\dagger={u^+(-s,t)}^\dagger.
$$
Let
\[ 
Z_{0}=C^+,\qquad u^+\colon Z_0\to\R^{2n}.
\]
For comparison with usual gluing we let 
\[
Z_{0}'=C^+\sqcup C^-.
\]

If $0\le a\le 1$, let $Z_a^+=[0,R(a)]\times S^1$, $Z_a^{-}=[-R(a),0]\times S^1$, and let $Z_a'$ be obtained from the usual identification $(s,t)=(s'+R,t)$ and take $Z_a$ to be the left half of $Z_a'$, $Z_a\approx [0,\frac12R(a)]\times S^1\subset Z_a^+$. We define $u^+\oplus_a\colon Z_a\to\R^{2n}$ to be the identity when $a=0$ and when $a> 0$, $u^+\oplus_a\colon Z_a^+$ is given by 
\[
u^+\oplus_a := (u^+\oplus_a u^-)|_{Z_a} 
\]
where
\[
u^+\oplus_a u^-(s,t) = \beta(s)u^{+}(s,t) + (1-\beta(s))u^-(s,t).
\]
Note that since $\beta(0)=\frac12$, with $(s,t)$ coordinates in $Z_0^+$, we have 
\[
[u^+\oplus_a](\tfrac12 R,t)= \tfrac12 u^{+}(\tfrac12 R,t) + \tfrac12 {u^+(\tfrac12,t)}^\dagger \in\R^n\subset\R^{2n},
\]
i.e., $u^+\oplus_a$ takes the boundary component of $Z_a$ which is also a boundary component of the domain to $L$.

Consider next a pair of vector fields $h$ which is a first order variations of the maps $u$. In the local model this simply means that $h$ live in the same functional analytic spaces as $u$. With the definition above we then note that for any $a\in[0,1]$, $h\oplus_{a}$ is naturally a vector field along $u\oplus_{a}$ that is tangent to $\R^{n}$ along the part of the boundary of $Z_{a}$ which is also boundary of the domain. The operation $u\mapsto u\oplus_{a}$ is the \emph{gluing} operation at an elliptic node. 

Then the analogue of Theorem \ref{lem: retract} again holds by passing to a $\Z/2$-invariant subset in proof: use the argument for the pair of maps $(u^+,u^-)$ gluing to $u^+\oplus_a u^-$, consider the subset invariant under $\Z/2$ and resterict to one half of the map.

\begin{remark}
In the gluing process for elliptic nodes the anti-gluing $u^+\ominus_a u^-\colon C_a'\to \R^{2n}$ becomes, when restricted to half its domain in line with the above, a map of a disk with interior puncture with real asymptotic constant and boundary in $\R^n$.    
\end{remark}

\subsection{Polyfold structures}\label{sec:stablemapmoduli}

The results \cite[Theorem 1.6, 1.7]{HWZ-GW} generalize readily to the boundary case: 

\begin{theorem} (Compare \cite[Theorem 1.6]{HWZ-GW})
	For any $\delta \in (0,2\pi)$, the space $\bZ^{(3, \delta)}(X, L, \omega)$ admits a second countable paracompact Hausdorff topology.  In this topology, the loci \linebreak $\bZ_{A, g,h,m,b}(X, L,\omega)$ are open and closed.
\end{theorem}

\begin{theorem} (Compare \cite[Theorem 1.7]{HWZ-GW}) \label{thm: polyfold structure on Z} 
	For any sequence $\delta_0 < \delta_1 < \cdots < 2\pi$, there is a polyfold structure on $\bZ$ with underlying filtration
	$\bZ^{(3, \delta_0)} \subset \bZ^{(4, \delta_1)} \subset \cdots \subset \bZ^{(3 +k, \delta_m)} \subset \cdots$
\end{theorem}
\begin{proof}
	Let us outline the argument of \cite[Section 3.5]{HWZ-GW} and indicate the necessary adaptation to the bordered case.  Local charts centered around smooth maps from parameterized smooth (not nodal) curves are entirely straightforward to construct.  Indeed, if $u$ is such a map, the space of sections of $H^{3,\delta}(S,u^{\ast} TX)$ has a standard sc-structure; this space can be identified with a space of maps to $X$ using any map $TX\to X$ with vertical differential along the zero section equal to the identity. In the the case of bordered Riemann surfaces, we have demanded that the maps  be $J$-holomorphic on the boundary.  This will be ensured by the exponential map $\exp\colon TX\to X$ of a Riemannian metric $g$ with the following two properties: $L$ is totally geodesic in $g$ and $J$ takes Jacobi-fields in $g$ along $L$ to Jacobi-fields, as in \cite[Section 5.2]{ESS}. The first condition is used to ensure that the exponential map preserves the boundary condition, the second to assure that the map is holomorphic along the boundary.

    Coordinate charts defining `neighborhoods' for maps from parameterized nodal curves are obtained in \cite{HWZ-GW} from \cite[Theorem 1.28]{HWZ-models} (recalled above as Theorem \ref{lem: retract}) and as explained above, by doubling, we have the analogous results above for the similar gluings in the bordered case.  

    The forgoing discussion explains how to construct local charts around {\em parameterized} maps.  It remains to explain how to use these to give an orbifold cover of $\mathbf{Z}$, whose points correspond to {\em unparameterized} maps.  Since stable Riemann surfaces have already only finitely many automorphisms, the main point is to explain what to do about maps with (positive symplectic area but) unstable domains.  
    
    In \cite{HWZ-GW}, given such a map, a (locally defined) hypersurface in $X$ transverse to the map is chosen. Points on the domain which meet the hypersurface are marked. Here the hypersurface is chosen so that the resulting marked domain curve is stable.  One considers the space of maps from such nearby marked curves which remain transverse to the hypersurface.\footnote{To ask for transversality to a hypersurface, one must speak of the derivatives of the map at points; to expect that nearby maps should remain transverse, one should have continuity of the derivative.  Note that $H^3$ ensures these properties by Sobolev embedding.} 
    
    This method of constructing `good uniformizing families' is detailed in  \cite[Theorem 3.13, Proposition 3.19, 3.20]{HWZ-GW} and goes through verbatim in the case of curves with boundary and show that there is a countable and locally finite cover of $\bZ$ by such good uniformizing families centered at smooth maps.  

    With charts covering $\bZ$ found, we need to check compatibility, i.e., that there are $\mathrm{sc}$-smooth coordinate changes between charts. The corresponding result in the closed case is \cite[Theorem 3.13]{HWZ-GW}, where it is shown in several steps that maps given by vector fields and gluing parameters in one chart around a nodal curve is expressed by vector fields and gluing parameters around the same map in another chart is $\mathrm{sc}$-smooth. The proof applies also to the charts for curves with boundary. 
    
    The construction of the polyfold is then obtained from a word by word repetition of \cite[Theorems 3.35 -- 37]{HWZ-GW}.
\end{proof}

\begin{remark}\label{r:boundarystabliization}
While not necessary, we may extend  \cite[Definition 3.6]{HWZ-GW} to also allow for stabilization by boundary marked points as follows (numbers in parenthesis refers to the numbering in \cite[Definition 3.6]{HWZ-GW}): In (1) we allow also for additional boundary marked points $\Theta$ (that are analogous to the additional interior marked points $\Sigma$ already considered). In (4) we require every point $w\in \Theta$ to lie in an $(n-1)$-dimensional submanifold $N_{w}\subset L$ and the tangent space to $N_{w}$ is a complement to $du(z)|_{\partial S}$. In (6) we use half-disks at the boundary marked points and in (7) the restriction to the boundary is transverse to $N_{m}$ and (8) is modified accordingly.   	
\end{remark}

As for closed curves,  there are evaluation maps, and a map to Deligne-Mumford space:

\begin{theorem} \label{thm: evaluation and forget} (Compare \cite[Theorem 1.8]{HWZ-GW}) 
	If $4g+2h+2m+b\ge 5$ then the natural evaluation maps $\ev_{j}\colon \bZ_{g,h,m,b}\to X$, at interior points, $\ev_{k}\colon \bZ_{g,h,m,b}\to L$, at boundary points, $\ev_{r}^{(0,1)}\colon \bZ_{g,h,m,b}\to L^{(0,1)}(T_{z_{r}}S,T_{u(z_{r})}X)$ (the complex linear part of the differential) 
	and the forgetful map (forgetting the map and unstable domain components) $\gamma\colon \bZ_{g,h,m,b}\to\overline{\mathcal{M}}_{g,h,m,b} = \bZ_{g,h,m,b}(\mathrm{point})$  
	are sc-smooth.  \qed 
\end{theorem}

\subsection{The bundle of complex anti-linear differentials}\label{sec:formaldiff}  
\cite[Section 1.2]{HWZ-GW} 
describes a bundle $\bW \to \bZ$ with fiber over $(u,S)\in\bZ$ the complex anti-linear formal differentials $TS^{0,1}\to u^{\ast}(TX^{0,1})$ to the map $u$.
Then $u \mapsto \bar \partial_J u$ is a section of $\bW$, with zero locus that corresponds to the $J$-holomorphic maps. 

Here we give the necessary adaptations to the case of maps with boundary conditions.  
Fix an almost complex structure $J$ on $X$ compatible with $\omega$. We write $(S,j,M,D,u)$ for a stable map in $\bZ$. Here $S$ is the Riemann surface which is the domain of the map, $j$ is the complex structure on $S$, $M$ the marked points, $D$ the nodal points, and $u\colon (S,\partial S)\to (X,L)$ the actual map. Elements in $\bW$ are then of the form
\[ 
(S,j,M,D,u,\xi),
\]  
where $\xi$ is a complex anti-linear formal differential that respects the boundary condition. More precisely, for $z\in S$
\[ 
\xi(z)\colon T_{z}S\to T_{u(z)}X 
\] 
is a complex anti-linear map: $J\circ \xi=-\xi\circ j$ and if $z\in\partial S$ then
$\xi(z)=0$.

We say that two tuples $(S,j,M,D,u,\xi)$ and $(S',j',M',D',u',\xi')$ are equivalent if there is a $(j,j')$-holomorphic isomorphism $\phi\colon S\to S'$ that takes $(M,D)$ to $(M',D')$ and such
\[ 
u=u'\circ\phi,\qquad \xi = \xi'\circ d\phi.
\]
We demand the same regularity properties of $\xi$  as in \cite[Section 3.6]{HWZ-GW}: $\xi$ is in $H^{2}_{\rm loc}$ at general and at marked points, and in $H^{2,\delta}$ at punctures.  

We will need the following analogues of the \cite[p.~118]{HWZ-GW} gluing formula for sections of $\bW$.  

In the hyperbolic and elliptic case, consider complex anti-linear maps $\xi^{+}$ and $\xi^{-}$ over $u^{+}$ and $u^{-}$.  We take these to be elements of: 
\begin{equation}\label{eq:perthyp}
H^{2,\delta_{0}}(\R^{\pm}\times [0,1],\mathrm{Hom}^{(0,1)}(T\R^{\pm}\times [0,1],u^{\ast}(T^{\ast}\R^{2n})))
\end{equation}
in the hyperbolic case and of
\begin{equation}\label{eq:pertell1} 
H^{2,\delta_{0}}(\R^{+}\times S^{1},\mathrm{Hom}^{(0,1)}(T\R^{+}\times S^{1},u^{\ast}(T^{\ast}\R^{2n})))
\end{equation}
and 
\begin{equation}\label{eq:pertell2}
H^{2,\delta_{0}}(\R^{-}\times S^{1},\mathrm{Hom}^{(0,1)}(T\R^{-}\times S^{1},u^{\ast}(T^{\ast}\R^{2n})))
\end{equation}
in the elliptic case, where in the elliptic case $\xi^-$ is obtained from $\xi^+$ by conjugation.
We define:
\begin{equation*}
\xi^{+}\widehat{\oplus}_{a} \xi^{-}(s,t)=
\begin{cases}
(\xi^{+}(s,t),\xi^{-}(s,t)) &\text{ for }a=0,\\
\beta_{a}(s)\xi^{+}(s,t)+(1-\beta_{a})\xi^{-}(s,t) &\text{ for }a>0.
\end{cases}\\
\end{equation*}
in the hyperbolic case and
\begin{equation*}
\xi^{+}\widehat{\oplus}_{a} =
\begin{cases}
\xi^{+}(s,t) &\text{ for }a=0,\\
{\beta_{a}(s)\xi^{+}(s,t)+(1-\beta_{a})\xi^{-}(s,t)}|_{Z_a} &\text{ for }a>0,
\end{cases}
\end{equation*}
in the elliptic case. Again, the arguments of \cite[Section 3.3]{HWZ-models} may be repeated to show that the projection $\widehat \oplus_a$ is sc-smooth.

Recall that in the scale calculus of \cite{HWZ}, spaces carry filtrations (in practice, by the regularity of functions), and that for vector bundles, it is natural to filter both base and fiber separately, with the fiber filtration going up to one larger than the base.  

Recall the levels $m = 0, 1, \ldots$ of the filtration on $\bZ$ correspond to maps of regularity $(m+3,\delta_{m})$ where $0 < \delta_0 < \delta_1 < \cdots < 2\pi$ is some sequence fixed once and for all.  For maps $u$ of regularity $(m+3,\delta_{m})$, the notion of the complex anti-linear map $\xi$ along $u$ having regularity $(k+2,\delta_{k})$ makes sense for $0\le k\le m+1$ and one says that an element $(S,u,\xi)$ in $\bW$ has bi-regularity $((m,\delta_{m}),(k,\delta_{k}))$ if $u$ is of class $(m+3,\delta_{m})$ and $\xi$ of class $(k+2,\delta_{k})$ along $u$.

We have the following structures:

\begin{theorem}(Compare \cite[Theorem 1.9]{HWZ-GW})
The space $\bW$ has a natural second countable paracompact Hausdorff topology so that the projection map
$p\colon \bW\to\bZ$, forgetting the $\xi$-component, is continuous.
\end{theorem}

\begin{theorem}(Compare \cite[Theorem 1.10]{HWZ-GW}) \label{thm: polyfold structure on W}
	Consider the polyfold $\bZ^{(3,\delta)}$ of stable maps.
	The bundle $p\colon \bW^{(2,\delta)} \to \bZ^{(3,\delta)}$ has the structure of a strong polyfold bundle in which the $(m,k)$-bi-level $\mathbf{W}_{m,k}$ (for $0 \le k \le  m + 1$) consists of elements of base regularity $(m+3,\delta_{m})$ and fiber regularity $(k + 2, \delta_{k})$.
\end{theorem}

\begin{proof}
    The topology and polyfold structure on $\bW$ is constructed as in \cite[Section 3.6]{HWZ-GW}.   Note that one begins with a set-theoretic map $\bW \to \bZ$, immediately from the definitions.  One proceeds by producing a chart for $\bW$ covering the preimage of each chart used in the construction of $\bZ$ and then proves that above each chart, $\bW$ is a scale retract of a trivial vector bundle, and that this structure is compatible with change of charts, see \cite[Proof of Proposition 3.39 on p.~123]{HWZ-GW}. 
    
    Constructing the charts for $\bW$ is straightforward along maps from smooth domains.  For a neighborhood of the maps from nodal domains, one again shows that $\bW$ is a retract of the corresponding bundle on the nodal locus, using the gluing formula given above.  
         
    The sole refinement of the discussion in \cite{HWZ-GW} which is necessary concerns the trivialization of the bundle $\bW$ on local charts.  In \cite[p.~118]{HWZ-GW}, this is done using parallel transport.  Here we must ensure the parallel transport is compatible with the boundary condition; this is guaranteed by our choice of metric (recall that $L$ was totally geodescic), using the property of Jacobi-fields along $L$ being $J$-complex linear, compare \cite[Sections 3.2 and 3.3]{EESPxR}.
\end{proof}

\subsection{The $\bar\partial_{J}$-section}\label{ssec: dbarsection}
The $\bar\partial_{J}$-operator gives a natural section $\bar\partial_{J}\colon \bZ\to \bW$,
\[ 
(S,j,M,D,u)\mapsto \left(S,j,M,D,u,\tfrac12\left(du+J\circ du\circ j\right)\right).
\]

The properties of this section are analogous to the closed case.  As there, the key result  is suitable smoothness near nodes.

\begin{theorem}\label{thm:basicdbarsection} (Compare \cite[Theorem 1.11]{HWZ-GW})
	The Cauchy-Riemann section $\bar\partial_{J}$ of the strong polyfold bundle $\bW\to \bZ$ is sc-Fredholm in the sense of \cite[Definition 3.8]{HWZ}, and its solutions are the stable holomorphic maps.
	On the component $\bZ_{A,g,m,b}$ of the polyfold $\bZ$, the Fredholm index of $\bar\partial_{J}$ is equal to
	\[ 
	\ind(\bar\partial_{J}) = (n-3)(2-2g-h)+2m+b +  2c_{1}^{\rm rel}(A), 
	\]
	where $\dim X=2n$, $g$ the arithmetic genus of the noded Riemann
	surfaces, $h$ the number of boundary components, $m$ the number of interior marked points, $b$ the number of boundary marked points, and $A\in H_{2}(X,L)$.
\end{theorem}

\begin{proof}
In the case of closed curves the corresponding result is the part of \cite[Theorem 4.6]{HWZ-GW} contained in \cite[Propositions 4.7 and 4.8]{HWZ-GW} which shows that $\bar\partial_J$ is smooth and regularizing in the polyfold sense. Above, we adapted the closed curve constructions to curves-with-boundary, with models $(\R_\pm\times[0,1],\R_\pm\times\{0,1\})\to (\R^{2n},\R^{n})$ at hyperbolic real nodes and $(\R_+\times S^1,\{\infty\}\times S^1)\to (\R^{2n},\mathrm{pt})\subset(\R^{2n},\R^n)$ at elliptic boundary nodes. The (regularity and gluing) estimates of \cite[Section 4.3--5]{HWZ-GW} are then obtained, as in the closed case, from the open counterparts of the results in \cite[Section 5]{HWZ-GW} which are in turn obtained from the same arguments as in the closed case after application of standard tools, e.g., doubling, to transfer to the bordered case. It then remains only to compute the index in the bordered case. This is well known, see e.g.~\cite[Theorem A.1]{CEL}.
\end{proof}

Let us also recall a key feature of the polyfold setting: if a $\mathrm{sc}$-Fredholm section has compact zero locus, then any perturbations of the section by an $\mathrm{sc}^+$ section small in an appropriate `auxiliary norm' also has compact zero locus \cite[Definition 5.5, Theorem 5.1]{HWZ}. Thus, Gromov's theorem guaranteeing compactness of $\bar \partial_J^{-1}(0)$ ensures also the compactness of $\bar \partial_J^{-1}(\lambda)$ for $\lambda$ small in an appropriate auxiliary norm. (In the closed case, this is also included in \cite[Theorem 4.6]{HWZ-GW}, the part contained in \cite[Proposition 4.9]{HWZ-GW}.)

We recall from \cite[Section 5.1]{HWZ} that, 
by definition, an auxiliary norm on $\mathbf{W}$ is a continuous function $\mathbf{W}_{0,1} \to \mathbb{R}$ that restricts to a complete norm on each fiber and which has the property that if $w_\alpha\in\mathbf{W}_{0,1}$ is a sequence such that $p(w_k)\to x$ in $\mathbf{Z}$ and if $N(w_k)\to 0$ then $w_k\to 0\in \mathbf{W}_{0,1}$.
Per \cite[Lemma 5.1, Proposition 5.1]{HWZ}, auxiliary norms exist and can be bounded above and below by a continuous function times the appropriate model norm on the fibers, which for us is $\|w\|_{3, \delta}$ for $w \in \mathbf{W}_{0,1}$.  Moreover, observe (by partition of unity) that given any region of $\mathbf{Z}$ with a locally finite cover by standard charts, we may and will ask that over this region, the auxiliary norm is uniformly comparable to the sum of fiber norms in these charts.  

That is, a sufficient condition for $\bar \partial_J^{-1}(\lambda)$ to be compact is if $\lambda$ is a $\mathrm{sc}^+$ section and in standard charts we have $\|\lambda\|_{3, \delta}$ sufficiently small. 

\subsection{Extendability} \label{sec: extendability} 
Let us write (just for this subsection) $\bZ$ for the configuration space of maps with no marked points or punctures on the domains, $\bZ_{\bullet}$ for the configuration space of curves with possible marked points (either interior or boundary) but no punctures, and $\bZ_{\circ}$ for the configuration space of points with punctures, but no marked points.  Recall that the difference between punctures and marked points is that we require $H^{3}_{\mathrm{loc}}$ regularity at marked points, while we allow $H^{3, \delta}$ regularity at punctures. 

\begin{lemma}\label{l:marked to punctures}
    There is a $\mathrm{sc}$-smooth  map $i\colon \bZ_{\bullet} \to \bZ_{\circ}$ which replaces each marked point by a puncture.
\end{lemma}
\begin{proof}
    We define the map in charts by sending each map `to itself'. In the notation of Section \ref{sec:regularity} this means in a neighborhood of a marked point $\zeta$ and corresponding puncture in the domain that, we send $v_0$ to $(v_\infty,v_0(\zeta))$, where $v_0(\zeta)\in\R^{2n}$ is the value of the asymptotic constant. Since $v_0\in H^3$ and the $H^3$-norm controls the $C^1$-norm the map to $v_0(\zeta)\in\R^6$  is a bounded linear map (in particular smooth) $H^3\to \R^6$, where the target is thought of as the space of asymptotic constants. Lemma \ref{l:H3toH3delta} then shows that $v_0\to v_\infty$ is a linear and bounded map of Sobolev spaces $H^k\to H^k_\delta$, $\mathrm{sc}$-smoothness follows.
\end{proof}

We write $\mathbf{W}_\bullet\to\mathbf{Z}_\bullet$ and $\mathbf{W}_\circ\to\mathbf{Z}_\circ$ for the bundles of formal complex anti-linear differentials over maps with marked points and punctures, respectively. 

\begin{lemma}
    There is a bundle map $i^\#\colon \bW_\bullet \to i^* \bW_{\circ}$ that covers the map $i\colon \bZ_\bullet\to\bZ_{\circ}$ that is linear and bounded on each fiber. 
\end{lemma}
\begin{proof}
    The proof is similar to the proof of Lemma \ref{l:marked to punctures}. 
    Again, in charts we send each formal complex antilinear differential to itself and Lemma \ref{l:H3toH3delta} shows that the maps is linear and bounded on fibers.   
\end{proof}

We write $\bZ^{\mathrm{reg}}$ for the locus of maps in $\mathbf{Z}$ with domains with positive area components  that are smooth. 
We also write $\bZ_{\bullet}^{\mathrm{reg}}$ and $\bZ_{\circ}^{\mathrm{reg}}$ for the locus of maps in $\mathbf{Z}_\bullet$ and $\mathbf{Z}_\circ$ where the components of the domain containing the marked points or punctures are smooth (not nodal) and have positive symplectic area.  The following is immediate from the definitions: 

\begin{lemma}
    There is a  sc-smooth forgetful map $f\colon \bZ_{\bullet}^{\mathrm{reg}} \to \bZ^{\mathrm{reg}}$, and a canonical identification $f^* \bW = \bW_\bullet$ over $\bZ^{\mathrm{reg}}$. 
\end{lemma} 
\begin{proof}
    The regularity conditions imposed at marked points are the same as at general points, and we have restricted attention to loci where forgetting marked points does not affect stability. 
\end{proof}

Let $\mathbf{U}^{\mathrm{reg}} \to \bZ^{\mathrm{reg}}$ be an \'etale chart and let $\lambda$ be an $\mathrm{sc}^+$-section (or multi-section) of $\bW|_{\mathbf{U}^{\mathrm{reg}}} \to \mathbf{U}^{\mathrm{reg}}$. Write $\mathbf{U}^{\mathrm{reg}}_\bullet=f^{-1}(\mathbf{U}^{\mathrm{reg}})$ and $\lambda_\bullet=f^\ast\lambda$. 
\begin{definition} \label{def: extendable}
    We say $\lambda$ is {\em extendable} on $\mathbf{U}^{\mathrm{reg}}$ if there is an \'etale chart $\mathbf{U}_\circ^{\mathrm{reg}} \to \bZ^{\mathrm{reg}}_{\circ}$ and identification $\mathbf{U}^{\mathrm{reg}}_\bullet = \mathbf{U}^{\mathrm{reg}}_\circ \times_{\mathbf{Z}^{\mathrm{reg}}_\circ} \mathbf{Z}^{\mathrm{reg}}_\bullet$ 
 and an $\mathrm{sc}^+$-section (or multi-section) $\lambda_\circ$ of $\bW_{\circ}|_{\mathbf{U}_\circ^{\mathrm{reg}}} \to \mathbf{U}_\circ^{\mathrm{reg}}$ such that $i^* \lambda_\circ = i^\#(\lambda_\bullet)$ on $\mathbf{U}_\bullet^{\mathrm{reg}}$.
\end{definition}

Note that $\mathbf{U}^{\mathrm{reg}}_\circ$ and $\lambda_\circ$ are unique if they exist, since $H^3_{\mathrm{loc}}$ maps are dense in the $H^{3, \delta}$ maps. Extendability can be checked locally around each marked point and corresponding puncture separately, since the regularity condition on the maps and section is imposed locally at the marked point.
The $\bar \partial_J$-section is evidently extendable.

\subsection{Index bundles and orientations of $\mathbf{W}$} \label{ssec: orientation} 
For closed curves, orientations associated to complex linear operators can be used to define orientations.
However, in the bordered case, there is orientation input also from the boundary condition, and coherent orientations exist only provided certain global conditions are met.  

Consider an $\mathrm{sc}$-smooth Fredholm section $F\colon \mathbf{Z}\to\mathbf{W}$; for us, $F=\bar\partial_{J}-\lambda$, where $\lambda$ is an $\mathrm{sc}^{+}$-section $\mathbf{Z}\to\mathbf{W}$. The vertical component (for some local trivialization, see e.g., \cite[Lemma 3.5]{EESPxR}) of the linearization of $F$, $dF\colon T_{(u,S)}\mathbf{Z}\to \mathbf{W}_{(u,S)}$ is a Fredholm operator. Similarly, $F$ gives families of Fredholm operators as follows. If $C\to \mathbf{Z}$ is a continuous map of a compact CW-complex $C$, we pull back the family of operators. By compactness of $C$, we can stabilize the pull-back bundle and extend the family of Fredholm operators so that all operators over $C$ are surjective. We obtain an index bundle over $\mathbf{Z}$, that is defined as a stable bundle when pulled back over any compact CW-complex and which depends only on the homotopy class of the map of the CW-complex. To see the last statement we repeat the construction of the bundle over a map over a homotopy of maps.

Our moduli spaces are zero-sets of $\bar\partial_{J}-\lambda$, so we may orient them given an orientation
of the index bundle over $\bar\partial\colon\mathbf{Z}\to\mathbf{W}$. 
In the case of curves with boundary, orienting the index bundle requires more care than in the closed case.
This has been discussed by many authors, following Fukaya, see e.g., \cite{FOOO}.  Existing discussions
happen in other perturbation setups (not using polyfolds), but as the problem is topological in nature
there is no essential difference.  We briefly recall here the idea of constructing the orientation. 

In cases relevant to our study here, the ambient symplectic manifold $X$ is spin. In this case coherent
(i.e., compatible with gluing)  orientations exist 
provided the Lagrangian boundary condition $L$ is spin and equipped with an orientation and a spin structure. More 
precisely we equip the relevant index bundle with the Fukaya orientation as follows. Consider first closed curves. As 
mentioned above, over a closed curve the index bundle is a complex linear bundle and receives an induced orientation 
from the orientation of the complex numbers. Consider next a disk. Using the spin structure, the bundle and the linearized 
$\bar\partial_{J}$-operator can be trivialized in a homotopically unique way outside a small neighborhood of the origin. 
After further homotopy, the bundle can be expressed as a sum of a standard $\bar\pa$-bundle on a trivialized bundle 
over the disk with constant $\R^{n}$-boundary conditions (kernel constant $\R^{n}$-valued functions, cokernel trivial) and a complex bundle over the sphere attached at
the origin of the disk. This then induces an orientation. 

Over general curves with boundary (higher genus and any number of boundary components) an orientation is defined as follows. Express the curve as a closed curve with marked points at which disks with marked point at the origin are glued in. The orientations of the index bundles of the closed curve (even index) and the disks with an ordering (index $\dim L-1$) then induces an orientation of the glued curve. In this paper $\dim L=3$, the orientation is independent of boundary ordering and we obtain a coherent orientation of the index bundle associated to $\mathbf{W}\to\mathbf{Z}$.    

\subsection{Adding multi-sections} \label{sec: adding multi}

To achieve transversality in the polyfold framework, or in any setting where the moduli spaces are orbifolds, one generally needs multi-valued perturbations, since one can have bundles for which no equivariant section is transverse.  
To inductively build perturbations, one has to add multi-sections to multi-sections.  Let us recall what this means, in a chart.  
If $\sigma$ and $\tau$ are a multi-sections with $m$ and $n$ local branches $s_1,\dots,s_m$ and $t_1,\dots,t_n$ with weights $w_1(\sigma),\dots,w_m(\sigma)$ and $w_1(\tau),\dots,w_m(\tau)$, respectively. Then we subdivide $\sigma$ and $\tau$ to  
$$
\left(\underbrace{s_1,\dots,s_1}_{n},\dots,\underbrace{s_m,\dots,s_m}_{n}\right) \text{with weights}  \left(\underbrace{\tfrac{1}{n}w(\sigma)_1,\dots,\tfrac{1}{n}w(\sigma)_1}_{n},\dots,\underbrace{\tfrac{1}{n}w(\sigma)_m,\dots,\tfrac{1}{n}w(\sigma)_m}_{n}\right)
$$
and
$$\left(\underbrace{t_1,\dots,t_1}_{m},\dots,\underbrace{t_n,\dots,t_n}_{m}\right)
\text{with weights }
\left(\underbrace{\tfrac{1}{m}w(\tau)_1,\dots,\tfrac{1}{m}w(\tau)_1}_{m},\dots,\underbrace{\tfrac{1}{m}w(\tau)_n,\dots,\tfrac{1}{m}w(\tau)_n}_{m}\right)
$$
and take $\rho=\sigma + \tau$ as the multi-section with local branches $\rho_{ij}$ of weight $w_{ij}(\rho)$, $1\le i\le m$, $1\le j\le n$, where
\[
\rho_{ij}= s_i + t_j,\quad w_{ij}(\rho)=\tfrac{1}{mn}w_i(\sigma)w_j(\tau).
\]
Note that when either of the multi-sections is a section, this is simply addition.  
We refer to \cite[Section 13]{HWZ} for more details on multi-sections.

\begin{remark}\label{r : multisections and partitions of unity}
    The standard notion of addition of multi-sections -- add all possible branches in all possible ways -- does not interact well with partitions of unity.  Indeed, suppose on some space $U$ there is some partition of unity $1_U = \phi_1 + \phi_2$, and some multi-section $\sigma$, which is just the union of two sections, $\sigma = \sigma_1 \cup \sigma_2$.  Then 
    $$\phi_1 \sigma + \phi_2 \sigma = 
    \phi_1 \sigma_1 \cup \phi_1 \sigma_2 \cup \phi_2 \sigma_1 \cup \phi_2 \sigma_2 \ne \sigma_1 \cup \sigma_2  = \sigma$$
    We refer to \cite[Section 1.9]{Pandharipande-Solomn-Tessler} for some discussion of this issue.
    This can be remedied at the cost of considering a more structured notion of multi-section, encoding the information about which branches should be added to which.  Such a notion is developed in \cite[Chapter 13 -- 14]{HWZ}.    

    In fact, in the present article we can avoid this problem without explicit resort to such measures:  our multi-section perturbations are obtained by symmetrizing sections supported in local charts, and our extension procedure is performed on the underlying sections, before symmetrizing.  (Presumably, equipping a multi-section with the data of arising by symmetrization of a section is an instance of a `structured multi-section'.) 
\end{remark}

\section{Basic perturbations} \label{sec: basic perturbation}
The map  $u \mapsto \bar \partial_J u$ provides a section of the bundle $\bW\to \bZ$; solutions to the Cauchy-Riemann equation are just the zeros of this section.  To perturb the Cauchy-Riemann equation, one should specify some other (multi-)section $\lambda\colon\mathbf{Z}\to\mathbf{W}$ and consider the equation $\bar \partial_J u = \lambda$.  

The general theory of \cite{HWZ} includes the existence of a rich class of perturbations  which suffice to ensure transversality of solutions, while also allowing elliptic bootstrapping for 
$\bar \partial_J u = \lambda$ and preserving the compactness of the solution space.  However, as discussed in Section \ref{our perturbations}, as we need perturbations with various additional and different properties, we will not appeal to this general construction but instead build or perturbations explicitly.  

In this section we describe the basic building blocks of our perturbations.  The perturbations we construct here will be supported over the locus $\mathbf{B} \subset \mathbf{Z}$ of bare maps, and in fact in a neighborhood of one such map.

Consider a smooth stable map with smooth domain $(u,S)$, and corresponding chart on $\bZ$ parameterized by
\begin{equation}\label{eq: nbhd basic perturbation}
\mathbf{U}(u,S)=\boldsymbol{j}\times H^{3}(S,u^* TX)
\end{equation}
around $(u,S)$, where $\boldsymbol{j}$ parameterizes deformations of the complex structure of the domain.  For an element $v\in H^{3}(S,u^* TX)$, we write $\widetilde{v}$ for the corresponding map $\widetilde{v}\colon S \to X$ given by composition with the exponential map.  For $(j, v) \in \mathbf{U}(u, S)$, it thus makes sense to write $\bar \partial_J \widetilde{v}$. 

\begin{definition} \label{def: basic perturbation} (Basic perturbations)
Fix a stable map with smooth domain $(u, S)$ and chart $\mathbf{U}(u, S)$ as above.
Fix the following collection of data: 
\begin{itemize}
	\item \emph{Domain differential and cut-off:} Let $\alpha$ be a smooth family of $j$-complex anti-linear differentials, $j\in\boldsymbol{j}$, $\alpha\in C^{\infty}(S, {T^{\ast} S}^{0,1})$ supported in the interior of $S$. 
   \item \emph{Target vector field:} Let $V\colon X\to TX$ be a smooth vector field supported inside a coordinate ball which lies at distance $>\delta$ from $L$.  
	\item \emph{Conformal structure cut-off:} Let $\gamma\colon \boldsymbol{j}\to[0,1]$ be a smooth cut-off function of compact support.
	\item \emph{Map cut-off:} Let $\beta_{0}, \beta_1 \colon [0,\infty) \to[0,1]$, be smooth cut-off functions with compact support equal to $1$ in a neighborhood of $0$ and equal to $0$ outside a larger neighborhood. 
    We will use cut-offs in the $H^{3}(S,u^{\ast} TX)$-direction of the form $\beta_{0}(\|v\|)$ and $\beta_{1}(\|\bar\partial_{J} \widetilde v\|_2)$  
\end{itemize}
then we define sections of $\mathbf{W}$ over $\mathbf{U}(u,S)$:
\begin{align}\label{eq: uncut basic}
\widehat{\lambda}(j,v)\ &= \  \gamma(j)\bigl(\widetilde{v}^* V \otimes \alpha\bigr),\\ \label{eq: chi0 basic perturbation in}
\lambda_0(j,v) \ &= \ \beta_0(v)\widehat{\lambda}(j,v),\\ \label{eq: L2 cut off lambda0}
\lambda(j,v) \ &= \ \beta_1(v)\lambda_0(j,v).
\end{align}

By a \emph{basic perturbation} centered at $(u, S)$, we mean a choice of chart $\mathbf{U}(u, S)$ and a section $\lambda$ as in  \eqref{eq: chi0 basic perturbation in}. If $M$ is a smooth manifold then a \emph{parameterized basic perturbation} is a 
($\mathrm{sc}$-)smooth section $\kappa\colon \mathbf{Z}\times M\to\mathbf{W}\times M$, where $\bZ\times M$ is endowed with the natural polyfold structure induced from the structure of $\bZ$ and $M$ viewed as an M-polyfold, see \cite[Theorem 4.2]{HWZ}, in such that for each $m\in M$ the restriction $\kappa_m=\kappa|_{\{m\}}$ is a basic section.
\end{definition}

\begin{lemma}\label{l: chi0-basic sc smooth}
The sections $\widehat{\lambda}$, $\lambda_0$, and $\lambda$ are sc-smooth $\mathrm{sc}^{+}$-sections. The fiberwise $(3,\delta)$-norms of $\widehat{\lambda}$, $\lambda_0$, and $\lambda$, as computed in the defining standard chart, is bounded. 
\end{lemma}
\begin{proof}
    The sections are supported in a small neighborhood of a map $(u,S)$ with smooth domain and $\mathrm{sc}$-smoothness then means that the section is differentiable as a map between Sobolev spaces $H^{3+k}\to H^{2+k}$ in the usual sense.

    To check  that $\lambda$ is $\mathrm{sc}^{+}$ we also need to check that it increases regularity scale by one, which means we must check that $\lambda(v,S)$ has the same Sobolev regularity $H^{3+k}$ as $v$. 
    
    Both these properties are straightforward consequences from the expression in a local chart. We have $\widehat{\lambda}=\gamma(j)(\widetilde{v}^\ast V\otimes \alpha)$ which is the map $v$ composed by a smooth vector field $V$ and multiplied by a smooth function. The fact that $V(v)$ has the same regularity as $v$ and that $V(v)$ is differentiable as a function, $H^3\to H^3$, of $v$ is standard, see e.g., \cite{Moser}, and multiplication by a smooth function keeps the regularity and differentiability properties. 
    It follows that $\widehat{\lambda}$ is $\mathrm{sc}$-smooth and $\mathrm{sc}^+$. 
    
    The section $\lambda_0$ is obtained by cutting $\widehat\lambda$ off by $\beta_0(\|v\|)$, for a smooth function $\beta_0$ constant around $0$, of the $L^2$-norm $\|v\|$. The $L^2$-norm is smooth for $\|v\|>\epsilon$, as the composition of a bounded linear map $H^3\to L^2$ followed by the norm. Hence $\lambda_0$ is $\mathrm{sc}$-smooth and $\mathrm{sc}^+$. The section $\lambda$ is obtained by further cutting off $\lambda_0$ by $\beta_1(\|\bar\partial_J \widetilde{v}\|_2)$, after noting that the 3-norm of $v$ controls the $2$-norm of $\bar\partial_J\widetilde{v}$ the $\mathrm{sc}$-smoothness and $\mathrm{sc}^+$ property of $\lambda$ then follows as for $\lambda_0$.  
    
    Finally, to see uniform boundedness of the sections note that the norms of the local coordinate expressions above are bounded in terms of the $H^3$-norm of $v$. 
\end{proof}

\begin{lemma}\label{lem: extendable}
    The perturbations $\widehat{\lambda}$ and $\lambda_0$ are extendable (in the sense of Definition \ref{def: extendable}). 
\end{lemma}
\begin{proof}
Consider $\widehat{\lambda}$ on $S$, let $D\subset S$ be a disk with marked point at the origin and let $[0,\infty)\times I$ be corresponding cylindrical coordinates. Center $\R^{2n}$-coordinates at $u(0)$. We use notation as in Lemma \ref{l:markedvspuncturenorms}: $v_0$ for the function on $D$ and $v_\infty$ for the corresponding function on $[0,\infty)\times I$. By Lemma \ref{l:H3toH3delta}, since $v\in H^3(D,\R^{2n})$, we have $v_\infty=w+c$, where $w\in H^{3,\delta}([0,\infty)\times I,\R^{2n})$ and $c=v_0(0)\in\R^{2n}$ is the asymptotic constant. The section of $\mathbf{W_\circ}$ obtained by pulling back $\widehat{\lambda}_\bullet$ by the change of coordinates $z=\exp(-2\pi(s+it))$ is then
\[
\gamma(j)V(v(z))\otimes\alpha(z)dz = \gamma(j)V(v(e^{-2\pi(s+it)}))\,\otimes\, -2\pi\alpha(e^{-2\pi(s+it)}) e^{-2\pi(s+it)}(ds+idt),
\]
and taking 
\[
\widehat{\lambda}_\circ = \gamma(j)V(w(s,t)+c)\,\otimes\, -2\pi\alpha(e^{-2\pi(s+it)}) e^{-2\pi(s+it)}(ds+idt)
\]
for $(w,c)\in N\subset H^{3,\delta}([0,\infty)\times I,\R^{2n})\times\R^{2n}$, where $N$ is an open subset containing $v_\infty$ for each $v_0\in \mathbf{U}(S,u)$, extendability follows.

To see that $\lambda_0$ is extendable we consider also the pull-back of the $L^2$-norm $\beta_0(v_0)$ to $v_{\infty}=(w,c)\in H^{3,\delta}([0,\infty)\times I,\R^{2n})\times \R^{2n}$. Lemma \ref{l:markedvspuncturenorms} \eqref{eq:ref pull-back L^2} says that
		\[
		\int_{D}|v_{0}|^{2} \,dxdy= \int_{[0,\infty)\times I}|v_{\infty}|^{2}e^{-4\pi s} \,dxdy.
		\] 
Extendability follows.
\end{proof}

\begin{remark}\label{r : not extendable}
The section $\lambda$ in \eqref{eq: L2 cut off lambda0} is not extendable: the pull-back of the function $\|\bar\partial_J v\|^2$ to $\mathbf{U}_\circ(u,S)$ is given by \eqref{eq:ref pull-back L^2 on derivatives} (take $j=1$) and since the integral norm in the chart on the cylindrical end is weighted, the pre-image of the locus where $\beta_1=1$ is not contained in any finite norm region. The main role of the cut-off $\beta_1(v)$ is to localize the support of the perturbation to the chart so that we can extend the cut-off perturbation to the whole configuration space by zero. 
\end{remark}

\begin{lemma}\label{l: function cut off basic}
Consider a smooth map from smooth domain $(u,S)$. 
Assume $\|\bar\partial_J u\|<\epsilon$. Then there exists $\epsilon_{0}>0$ and $C>0$  such that if $v \in H^3(S,u^\ast TX)$ satisfies $\|v\|_{3}<\epsilon_{0}$ and
$
\|v\|+\|\bar\partial_J\widetilde{v}\|<\delta
$
then $\|\tilde v\|_{3}< r$.
\end{lemma}

\begin{proof}
We pick a finite cover of $S$ by local coordinate disks $D$ such that $\widetilde{v}$ maps each $D$ into some coordinate ball in $X$. (Here we use the $\epsilon_0$-bound.) From the elliptic estimate:
\begin{align*}
\|\widetilde{v}-u\|_3 &\le C(\|\bar\partial_J(\widetilde{v}-u)\|_2 + \|\widetilde{v}-u\|)\\
&\le C(\|\bar\partial_J\widetilde{v}\|_2+\|\bar\partial_J u\|_2 + \|\widetilde{v}-u\|)
\end{align*}
Since $v=F(u,v)=\exp_u^{-1}(\widetilde{v}-u)$ is a smooth map with $G(u,0)=0$ and $\partial_v F(0)=v$ we find by Taylor expansion that there exists a constant $C_0$ (depending on $u$) such that $\|v\|_3\le K\|\widetilde{v}-u\|_3$. Also, $\widetilde{v}=\exp_u(v)$ is smooth and similarly $\|\widetilde{v}-u\|\le C_1\|v\|$. We get
\[
\|v\|_3 \le  C'(\|\bar\partial_J\widetilde{v}\|_2+\|\bar\partial_J u\|_2 + \|v\|).
\]
The lemma follows. 
\end{proof}

\begin{corollary}\label{c : cut off L2 norm and L2norm of dbar}
If the supports of $\beta_{0}$ and $\beta_1$ are sufficiently small, then the closure of the support of $\lambda$ is contained in a bounded subset of $\mathbf{U}(u,S)$. 
\end{corollary}
\begin{proof}
We apply Lemma \ref{l: function cut off basic}
to the $v \in H^3(S,u^\ast TX)$ used, per \eqref{eq: nbhd basic perturbation}, in the parameterization of $\bU(u,S)$. 
\end{proof}

\begin{remark}\label{r : multisection basic}
We recall from \cite[Theorems 13.7 -- 9]{HWZ} that we may obtain multi-sections of polyfold bundles by symmetrizing a section which is supported in the interior of an appropriate sort of open chart, and extending by zero. In \cite{HWZ-GW} it is shown that the $\mathbf{U}(u, S)$ charts are of the correct sort.  Correspondingly, we will write $\widehat{\lambda}(j,v)$ and $\lambda(j,v)$ also to mean the multi-sections over $\mathbf{Z}$ obtained by symmetrizing the corresponding sections in their local charts.  Note however that questions of regularity, transversality, etc.\ are typically local to each branch, hence reduce quickly to something checked for the original section.   
\end{remark}

\section{Transversality and genericity at bare maps} \label{sec: transversality at bare}
In this section we show that basic perturbations suffice to achieve transversality along the locus of bare maps, and moreover to
ensure that solutions are `1-generic', i.e.\ in 0- and 1-dimensional families of solutions, one does not meet
phenomena of expected codimension $\ge 1$.  The arguments proceed along standard lines. 

Below will sometimes consider  forgetting marked points. While maps that forget marked points do not exist globally over $\mathbf{Z}$, they do exist, and behave well, on the locus where the marked points stay off the ghost components, and in particular, marked points do not collide, or limit to nodes. 
We will only be using the forgetful maps on such loci. 

Throughout we fix the homology class represented by maps $(u,S)$, $u\colon S\to X$.  
Recall from Definition \ref{def:bare} that a stable numerically symplectic map is bare if it has no ghost components. We write $\bB\subset \bZ$ for the subset of bare maps and $\bB^{\chi}\subset\bZ^{\chi}$ for bare maps of Euler characteristic $\chi$. 

\subsection{Transversality at one map} 
We will next present a result showing that basic perturbations are sufficient to achieve transversality for the perturbed $\bar\partial_J$-operator, compare \cite[Theorem 5.5 and Definition 5.8]{HWZ} which gives similar results for more general $\mathrm{sc}^+$-sections.

We start with a basic result that shows that a finite collection of basic perturbations along a bare map suffices to span the cokernel of the linearization of $\bar\partial_J$-section. Consider a bare map $u\colon S \to X$, and fix a chart $\mathbf{U}(u, S)$ centered at $(u, S)$.

\begin{lemma}\label{l: basic perturbations span the cokernel}
The $\mathrm{sc}$-Fredholm $\bar\partial_J$-section has the following property at $(u,S)$:
If $L=L_{(u,S)}\bar\partial_J$ is the linearizaion of $\bar\partial_J$ at $(u,S)$ then there is a finite collection of basic perturbations $\lambda_k$, $k=1,2,\dots,m$, supported in $\bU(u,S)$ such that $\lambda_k(u,S)$ span a linear complement of the image of $L$. Moreover, the domain differential cut-offs $\alpha_j$ of $\lambda_j$, see \eqref{eq: uncut basic}, can be chosen to be supported in any open subset $D$  which intersects every irreducible component of $S$ and which has the property that $u|_D$ is an immersion.  
\end{lemma}

\begin{proof} 
This follows from a standard argument, see e.g.\ \cite[Proof of Proposition 3.2.1 on p.\ 49]{McDuffSalamon}, that verifies a specified class of perturbations is sufficiently large to guarantee transversality. We recall the argument and show that it applies to the class of basic perturbations.

Let the irreducible components of $S$ be $S_r$, $r=1,2,\dots,m$, and choose a disk $D_r \subset S_r$ for each $r$ such that $u|_{D_r}$ is an embedding, such disks exists since the area of $u|_{S_r}$ is positive for each $r$ by the assumption that $(u,S)$ is bare. Choose an inner product $\langle\,,\rangle$ on $L^2(\mathrm{Hom}^{0,1}(TS,u^{\ast}TX))$. 

Since the linearized operator $L=L_{(u,S)}\bar\partial_{J}$ is Fredholm, its cokernel is finite dimensional. We show that either $\mathrm{coker}(L)=0$, which means the lemma holds with zero basic perturbations, or there exist a basic perturbation $\lambda$ with conformal cut-off supported in $\bigcup_r D_r$ such that $\lambda(u,S)$ is not contained in the image of $L$. In the second case, $L$ followed by the projection to the $L^2$-complement of $\lambda(u,S)$ is a Fredholm operator of index one higher than $L$, with kernel of the same dimension, and with cokernel of dimension one less than $L$. Repeating the argument until the index equals the kernel dimension the lemma follows in general.

Assume then that there exists $\xi\ne 0$ which is $L^2$-perpendicular to the image of $L$. Then $L^\ast\xi=0$, where $L^\ast$ is the operator $L^2$-dual to $L$. By partial integration, $L^\ast=(L_{(u,S)}(\bar\partial_{J})^\ast$ is elliptic operator (with smooth coefficients by smoothness of $u$). Then, by elliptic bootstrapping for $L^\ast$, $\xi$ is continuous. By continuity of $\xi$ and since there exists a basic perturbation supported in an arbitrarily small open set around any point $\zeta\in D_r$ that takes any value at $\zeta$, $\xi$ must vanish throughout $D_r$. By unique continuation for the elliptic operator $L^\ast$, $\xi$ then vanishes everywhere in $S_r$ for each $r$, which contradicts $\xi\ne 0$. The lemma follows.
\end{proof}

Let $s_\theta$ be an $\mathrm{sc}^+$-section of $\bW$ over $\mathbf{U}(u,S)$. Then $\bar\partial_J-s_{\theta}$ is an $\mathrm{sc}$-Fredholm section, see \cite[Proposition 1.8]{HWZ}.

\begin{lemma}\label{l: basic perturbations span at sc+ section}
If $s_\theta$ is an $\mathrm{sc}^+$-section then Lemma \ref{l: basic perturbations span the cokernel} holds with the section $\bar\partial_J$ replaced by (the perturbed) section $\bar\partial_J-s_\theta$.
\end{lemma}

\begin{proof} 
The proof is directly analogous to the proof of Lemma \ref{l: basic perturbations span the cokernel}. Here the linearized operator $L$ has an new $0$-order term coming from the linearization of $s_\theta$. Since $s_\theta$ is $\mathrm{sc}^+$, and $u$ is smooth this $0$-order term is smooth as well. Then, as before, the dual operator $L^\ast$ is elliptic with smooth coefficients and unique continuation gives the lemma.   
\end{proof}

\begin{corollary} \label{cor: local perturbation at solution}
Given a $\mathrm{sc}^+$ multi-section $\theta\colon \bB^\chi \to \bW^\chi$ and a map $u\colon S \to X$ satisfying $\bar \partial_J u = \theta(u)$,  there exists a finite set of basic perturbations $\lambda_\alpha$, $\alpha\in A$, defined in a neighborhood $\mathbf{U}(u,S)$ of $(u,S)$ such that the solution $(u,0)$ to $\bar\partial_J(u,\epsilon_\alpha)=\theta(u)+\sum_\alpha\epsilon_\alpha\lambda_\alpha$ in $\mathbf{U}\times [0,1]^{|A|}$ is transverse in the sense of \cite[Definition 15.2]{HWZ}.   
\end{corollary}
\begin{proof}
Consider a local section $s_\theta$ of the multi-section $\theta$. Since $s_\theta$ is an $\mathrm{sc}^+$-section, by elliptic regularity for $\bar\partial_J$, the solution $(u,S)$ to $\bar\partial_J-s_\theta=0$ is smooth, compare \cite[Definition 3.8 and Theorem 3.2]{HWZ}. Then Lemma \ref{l: basic perturbations span at sc+ section} for $\bar\partial_J-s_\theta$ at $(u,S)$ gives a finite collection of basic perturbations $\lambda_\alpha$ at $(u,S)$ that together span the cokernel of the linearization of $\bar\partial_J-s_\theta$. Letting $A$ range over the basic perturbations for all local sections of $\theta$, the corollary follows.
\end{proof}

\subsection{Transversality on the bare locus} 
Let $\mathbf{B}^{\chi;0}\subset \mathbf{B}^\chi$ denote the space of bare stable maps from smooth domains.    
Let $\mathbf{B}^{\chi;1} \subset \mathbf{B}$ denote the space of bare stable maps from domains which are either smooth, or have exactly one node, which is a boundary node.
Recall that our `basic perturbations' are always supported in neighborhoods of bare maps with smooth (not nodal) domains.  

Here we observe that if we have been so fortunate as to have already achieved compactness of perturbed solutions on the loci $\mathbf{B}^{\chi; 0}$ and $\mathbf{B}^{\chi; 1}$, then our basic perturbations suffice to achieve transversality while preserving this compactness. 

\begin{lemma}\label{bare transversality} 
$\quad$
\begin{itemize}
\item[$(0)$]
Suppose given an $\mathrm{sc}^+$-multi-section $\theta\colon \bB^{\chi;0}  \to \bW^\chi$, such that the solution space to $\bar \partial_J = \theta$ is compact.  
Then there are finitely many basic perturbations $\lambda_\alpha$, $\alpha\in A$ 
of arbitrarily small auxiliary norm and defined in charts around maps $(u_{\alpha},S_{\alpha})$ such that all solutions to  
\begin{equation}
\label{eq: perturb by basic perturbations}
\bar\partial_{J} u = \theta(u)+\sum \lambda_\alpha(u)     
\end{equation}
are transversely cut out, and the space of such solutions is compact.  Here, the meaning of \eqref{eq: perturb by basic perturbations} is that we first symmetrize the sections $\lambda_\alpha$ on the \'etale charts where they are defined to multi-sections which descend to $\mathbf{Z}$, and then add the resulting multi-sections to each other and to $\theta(u)$ by the usual sum of multi-sections prescription as recalled in Section \ref{sec: adding multi}.

\vspace{2mm}

\item[$(1)$]
Suppose given a parameterized $\mathrm{sc}^+$-multi-section $\theta_t\colon \bB^{\chi;1}\times [0,1]  \to \bW^\chi$ such that the solution space to $\bar\partial_{J}=\theta_{t}$ is compact.  Assume additionally that solutions over the boundary $t=0,1\in\partial[0,1]$ are transversely cut out, and that the solution space meets the locus of maps with one boundary node transversely.  

Then there are finitely many 
parameterized basic sections $\lambda_{\alpha,t}$, $\alpha\in A$ 
of arbitrarily small auxiliary norm in charts centered at $((u_{\alpha},S_{\alpha}),t_\alpha)$ where $t_\alpha$ lies in the interior of $[0,1]$ such that the space of solutions to 
\begin{equation}\label{eq : (1) bare transversality}
\bar\partial_{J} u = \theta_{t}(u)+\sum \lambda_{\alpha;t}(u) 
\end{equation}
is compact and transversely cut out; and such that the solutions
agree with the solutions of the unperturbed equation for $t$ in some neighborhood of $\partial[0,1]$ and in some neighborhood of the locus of maps in $\bB^{\chi;1}\times [0,1]$ with one boundary node. 
\end{itemize}
\end{lemma} 
\begin{proof}
	We first consider $(0)$.
	By hypothesis the moduli space of solutions $\mathcal{M}^{\chi}$ to $\bar\partial_{J}u=\theta(u)$ in $\bB^{\chi}$ is compact. 
	Cover $\mathcal{M}^{\chi}$ by coordinate neighborhoods $\bU(u_\beta,S_\beta)$ inside $\bB^{\chi}$
    for which Corollory \ref{cor: local perturbation at solution} holds.  We then find finitely many basic sections 
	$\kappa_\alpha$, $\alpha\in A$, such that all solutions to
    $$\bar\partial_J(u,\epsilon)=\theta(u)+\sum_\alpha\epsilon_\alpha\kappa_\alpha(u),
    $$
    $(u,\epsilon_\alpha)\in\mathbf{B}^{\chi}\times[0,1]^{|A|}$, with $\epsilon=0$, i.e., $(u,\epsilon)=(u,0)$, are transverse. By the Sard-Smale theorem (compare \cite[Definition 15.2 and Theorem 15.3]{HWZ}), we then have transversality of solutions to the equation $\bar\partial_{J}(u)-\theta^{\chi}(u)-\sum_{\alpha\in A}\epsilon_{\alpha}\kappa_\alpha$ for an open dense set of $\epsilon=(\epsilon_{\alpha})$ in a neighborhood of $0\in[0,1]^{|A|}$. 
    Take $\lambda_{\alpha}=\epsilon_{\alpha}\kappa_\alpha$ for $\epsilon$ in this subset. 
    
	We turn to compactness of the perturbed solution set.  The operator $\bar\partial_J-\theta$ is an $\mathrm{sc}$-Fredholm section $\mathbf{Z}\to\mathbf{W}$ with compact solution set. 
    The section $\sum\lambda_{\alpha}$ is $\mathrm{sc}^{+}$ by Lemma \ref{l: chi0-basic sc smooth} and have auxiliary norm of size $\mathcal{O}(|\epsilon|)$. Hence the solution set $\bar\partial_J^{-1}(\theta+\sum \lambda_\alpha)$ is compact (and transverse) for all sufficiently small $|\epsilon|$ by \cite[Theorem 5.1]{HWZ}.
    
	The argument for $(1)$ is the same after observing that the desired transversality holds already near $\partial[0,1]$ and in a neighborhood of the locus of maps with a boundary node, hence we need introduce no new perturbations here. 
\end{proof}

\begin{remark}
    It is only in Lemma \ref{bare transversality} that the basic perturbations, which by definition are (not multi) sections defined in \'etale charts, are symmetrized to multi-sections and then descended to $\mathbf{Z}$.  We will not invoke this lemma until we turn to the proofs of our main results in Theorem \ref{main theorem 0-parameter}.  
\end{remark}

\subsection{Genericity of solutions}\label{sec: add Fredholm}
Let us also note that
Lemma \ref{bare transversality}  asserts nothing yet about the maps $u$ themselves which are solutions, only that the solution space is transversely cut out.  We will later want to know that we may choose perturbations so that the maps themselves have some particular genericity properties.  This is checked by a standard general position argument: we introduce auxillary marked points and formulate the failure of the desired generic properties as some corresponding Fredholm problems, then argue by Sard-Smale.  We give the details in the remainder of this section. 

We 
write $\bB_n^{\chi;0}$ to indicate the locus of bare maps from smooth domains whose domain has $n$ marked points
(and no punctures). We indicate in context whether the marked points should be interior or boundary. Recall the bundle of complex anti-linear differentials $\bW^{\chi;0}\to\bB^{\chi;0}$. In the presence of marked points we have the analogous bundle $\bW^{\chi;0}_{n}\to\bB^{\chi;0}_{n}$. 

We first list the various Fredholm problems we will consider and their indices, then draw conclusions. 

\subsubsection{The Cauchy-Riemann operator}\label{ssec:basicoperator}
Consider the polyfold bundle $\bW^{\chi;0} \to \bB^{\chi;0}$. 
Our basic Fredholm section is just the Cauchy-Riemann operator, 
$\bar\pa_{J}$. Because we are studying the case of a Calabi-Yau 3-fold with Maslov zero Lagrangian boundary conditions, the index is, per Theorem \ref{thm:basicdbarsection},
\[ 
\ind(\bar\pa_{J})=0.
\]

\subsubsection{The complex linear differential}\label{ssec:clindiff}
Consider the polyfold bundle $\bW^{\chi;0}_{1} \to \bB_{1}^{\chi;0}$. 
We extend the bundle to
\[ 
\bW_{1}^{\chi;0} \oplus \mathrm{Hom}^{0,1}(T_{p}S,T_{u(p)}X)\to \bB_{1}^{\chi;0},
\] 
by adding the space of $(J,j)$-complex linear maps from the tangent space of the surface at the marked point $p$ to the tangent space of the target at its image $u(p)$. In this setting we will study the extended $\bar\partial_{J}$ section $\bar\partial_{J}\oplus \partial_{J}(p)$ given by
\[ 
(\bar\partial_{J}\oplus \partial_{J}(p))(u) = \bar\partial_{J} u\oplus \partial_{J}u(p),
\] 
where $\partial_{J}u(p)$ denotes the complex linear part of the differential at $p$: $\partial_{J}u=\frac12\left(du - J\circ du\circ j\right)$. Since $u\in H^{3}_{\rm loc}$, the differential of $u$ is continuous and the evaluation described is well defined.

In the case of an interior marked point $p$, the Fredholm index of $\bar\partial_{J}\oplus \partial_{J}(p)$ is
\begin{equation}\label{eq:cdiffinterior}
	\ind(\bar\partial_{J}\oplus \partial_{J}(p))= \ind(\bar\partial_{J})+ 2 - 6 = -4,
\end{equation}
and in the case of a boundary marked point $q$, 
\begin{equation}\label{eq:cdiffboundary}
	\ind(\bar\partial_{J}\oplus \partial_{J}(q))= \ind(\bar\partial_{J})+ 1 - 3 = -2.
\end{equation}

\subsubsection{Interior crossings}\label{ssec:interiorinjective}
Consider the polyfold bundle $\bW^{\chi;0}_{2} \to \bB_{2}^{\chi;0}$ where $\bB_{2}^{\chi;0}$ is the space of stable maps with two interior marked points $p_{1}$ and $p_{2}$. 
We extend the bundle to
\[ 
\bB_{2}^{\chi;0} \times X\times X\to \bB_{2}^{\chi;0}.
\] 
In this setting we will study the extended $\bar\partial_{J}$-section $\bar\pa_{J}\times\ev_{12}$ given by
\[ 
\bar\pa_{J}\times\ev_{12}\, (u)= \bar\partial_{J} u\times (u(p_{1}),u(p_{2})).
\] 

We then consider $(\bar\pa_{J}\times\ev_{12})^{-1}(0\times\Delta_{X})$, where $\Delta_{X}$ is the diagonal in $X\times X$. The Fredholm index $\ind(\bar\pa_{J}\times\ev_{12},0\times\Delta_{X})$ of this problem is
\[ 
\ind(\bar\pa_{J}\times\ev_{12},0\times\Delta_{X})= \ind(\bar\partial_{J})+ 4 - 6 = -2.
\]

\subsubsection{Boundary crossings}\label{ssec:injective}
Consider the polyfold bundle $\bW_{2}^{\chi;0} \to \bB_{2}^{\chi;0}$ where $\bB_{2}^{\chi;0}$ is the space of stable maps with two boundary marked points $p_{1}$ and $p_{2}$. 
We extend the bundle to
\[ 
\bB_{2}^{\chi;0} \times L\times L\to \bB_{2}^{\chi;0}.
\] 
In this setting we will study the extended $\bar\partial_{J}$-section $\bar\pa_{J}\times\ev_{12}$ given by
\[ 
\bar\pa_{J}\times\ev_{12}\, (u)= \bar\partial_{J} u\times (u(p_{1}),u(p_{2})).
\] 

We then consider $(\bar\pa_{J}\times\ev_{12})^{-1}(0\times\Delta_{L})$, where $\Delta_{L}$ is the diagonal in $L\times L$. The Fredholm index $\ind(\bar\pa_{J}\times\ev_{12},0\times\Delta_{L})$ of this problem is
\[ 
\ind(\bar\pa_{J}\times\ev_{12},0\times\Delta_{L})= \ind(\bar\partial_{J})+ 2 - 3 = -1.
\]
That is, for generic $1$-parameter families of perturbations, one expects isolated instances when the boundary curve has a self-intersection, and at such instances, the two tangent vectors together with the tangent vector of the deformation at the intersection point span the tangent space of $L$.  

\subsubsection{Crossing $L$}\label{ssec:4chainandL}
Consider the polyfold bundle $\bW^{\chi;0}_{1} \to \bB_{1}^{\chi;0}$ where $\bB_{1}^{\chi;0}$ is the space of stable maps with one interior marked point $p$. 
We extend the bundle to
\[ 
\bW \times X\to \bB_1^{\chi;0}.
\]  
In this setting we will study the extended $\bar\partial_{J}$-section $\bar\pa_{J}\times\ev$ given by
\[ 
(\bar\pa_{J}\times\ev)(u) = \bar\partial_{J} u\times (u(p)).
\] 

We then consider $\bar\pa_{J}\times\ev^{-1}(L)$. The Fredholm index $\ind(\bar\pa_{J}\times\ev,0\times L)$ of this problem is then
\[ 
\ind(\bar\pa_{J}\times\ev,0\times L)= \ind(\bar\partial_{J})+ 2 - 3 = -1.
\]
This means that for generic $1$-parameter families there are isolated instances where a curve meets $L$ at an interior point and at such instances the image of the differential of the map together with the tangent space of $L$ and the tangent vector of the deformation spans the tangent space to $X$.

\subsubsection{Genericity}\label{sssec : genericity}
Using the  Fredholm problems above, we  refine Lemma \ref{bare transversality} as follows.

\begin{lemma}\label{l: generic solutions 0}
    In the situation of Lemma \ref{bare transversality} $(0)$, we may ensure that all solutions to the perturbed equation are embeddings with everywhere non-zero complex linear differential and with interior disjoint from $L$. 
\end{lemma} 
\begin{proof}
	To see that the there are basic perturbations so that solutions are injective in the interior we consider the operator $\bar\partial_{J}\times\ev_{12}$ as in Section \ref{ssec:interiorinjective} and consider the compact inverse image of $0\times\Delta_{X}$.
	Arguing as the proof of Lemma \ref{bare transversality} $(0)$, we find a finite number of basic perturbations that span the cokernel of this operator at every solution and then apply the Sard-Smale theorem to find finitely many basic perturbations so that solutions are transverse and injective in the interior. Repeating the argument with the operators in Sections \ref{ssec:clindiff}, \ref{ssec:injective}, and \ref{ssec:4chainandL} we find similarly that after possibly adding more basic perturbations, solutions have injective complex linear differential, are injective on the boundary, and miss $L$ in the interior. The lemma follows.
\end{proof}

Similarly, Lemma \ref{bare transversality} $(1)$ implies that in the $1$-parameteric case, the space of solutions to $\bar\partial_J-\theta_t$ is a 1-dimensional manifold-with-boundary consisting of the solutions with nodal domains and using the Fredholm problems above we refine it as follows:  

\begin{lemma}\label{l: generic solutions 1}
    In the situation of Lemma \ref{bare transversality} $(1)$, we may ensure the solutions are embeddings except for isolated instances of maps with smooth domains and one of the following two generic degenerations:  
	\begin{itemize}
	\item[$(a)$] A transverse boundary crossing, i.e., the tangent vectors along the intersecting boundaries, together with the limit of secants connecting the two eventually colliding points, span a three dimensional space. 
	\item[$(b)$] A transverse intersection with $L$: the image of the tangent plane to the curve, the tangent to the Lagrangian, and the limiting secant together span a six dimensional space. 
	\end{itemize}
\end{lemma}
\begin{proof}
	The proof is directly analogous to the proof of Lemma \ref{bare transversality} $(1)$: we find an open cover of the solution space in the complement of maps over $\partial[0,1]$ and the locus of nodal maps and basic sections that span the cokernel of the augmented operators associated to $\bar\partial_{J}$ in Sections \ref{ssec:4chainandL}, \ref{ssec:clindiff}, \ref{ssec:interiorinjective}, \ref{ssec:injective}, and \ref{ssec:4chainandL}. Using compactness of the solution space we pass to a finite cover and use the Sard-Smale theorem to get transversality. The Fredholm index calculations in Section \ref{sec: add Fredholm} then implies the lemma.
\end{proof}

\subsubsection{Gluing generic 1-parameter families}\label{ssec : gluing 1-parameter familes}
In this section we show that solutions of $\bar\partial_J$-equation perturbed by basic perturbations for which Lemma \ref{l: generic solutions 1} holds have standard boundary crossings/hyperbolic nodes, see \cite[Definition 3.1]{SOB}, in case $(a)$ and standard Lagrangian crossings/elliptic nodes, see \cite[Definition 3.2]{SOB}, in case $(b)$. (This ensures that wall-crossings in solution spaces for the perturbed Cauchy-Riemann equations are compatible withe the HOMFLYPT skein-relations.)

To this end we first discuss some general properties of sections that vanishes near $L$. 
\begin{definition}\label{def : near L property}
A 1-parameter family of sections $\theta_t\colon \bW\times[0,1]\to\bZ\times[0,1]$ \emph{vanishes near $L$} if 
there is a neighborhood $N(L)\subset X$ of the Lagrangian $L\subset X$ such that for any $(u,S)\in\bZ$, $\theta_t(u,S)$ vanishes on the pre-image of $N(L)$, $\theta_t(u,S)|_{u^{-1}(N(L))}=0$.
\end{definition}

We next consider transversely cut out $1$-parameter families of solutions to the equation $\bar\partial_J-\theta_t=0$. For later reference (Lemma \ref{l : standard crossings/nodes in 1 parameter families} below) we include references to the corresponding results in standard Floer gluing for $\bar\partial_{J_t}$ established in \cite{SOB}.
Consider a $1$-parameter family of sections $\theta_t$ over the space $\bZ_{>\chi_0}\times[0,1]$ of stable maps of Euler characteristic $>\chi_0$ that vanishes near $L$. Let $(u_{t_0},S_{t_0})$ be a bare $1$-generic map of Euler characteristic $\chi_0+1$ with either $(a)$ a single boundary crossing, compare \cite[Section 4.2.2]{SOB} or $(b)$ a single interior intersection with $L$, compare \cite[Section 4.2.4]{SOB}. Then we can mark the boundary intersections in case $(a)$ and consider them as a hyperbolic boundary node, or in case $(b)$, mark the interior intersection point with $L$ and consider it an elliptic boundary node. We denote the new domain so obtained $\hat S_{t_0}$ and the map $u_{t_0}$ on this new domain $\hat u_{t_0}$. Then $(\hat u_{t_0},\hat S_{t_0})$ is a stable map of Euler characteristic $\chi_0$ which is bare and has a single boundary node. 

Consider a neighborhood $V'$ of $(\hat u_{t_0},\hat S_{t_0})$ in the codimension one subset of $\bZ_{>\chi_0-1}\times[0,1]$ of stable maps with a single boundary node. We define a section $\theta_t'$ of $\bW\times[0,1]$ over $V'$ as follows. If $(v,T)\in V'$ then the normalization $(\tilde v,\tilde T)$ is a map in $\bZ_{<\chi_0}$ and we take $\theta_t'(v,T)$ to agree with $\theta_t(\tilde v,\tilde T)$. We then extend $\theta_t'$ to a neighborhood $V$ of $V'$ in $\bZ_{>\chi_0-1}\times[0,1]$ by hat gluing $\widehat\oplus_a$, for small $a>0$, see Section \ref{sec:formaldiff}. We denote the resulting local section of $\bW\times[0,1]$ over $V$ by $\hat\theta_t$ and call $\hat\theta_t$ on $V$ the \emph{local extension of $\theta_t$}. Note that $\hat\theta_t$ also vanishes near $L$. 

Consider now a $1$-parameter family of $\mathrm{sc}^+$-sections $\theta_t$ on $\bZ_{>\chi_0}\times[0,1]$ that vanish near $L$ and assume that $(u_t,S_t)$, $t\in(t_0-\delta,t_0+\delta)$ is a $1$-parameter family of bare solutions to $\bar\partial_J-\theta_t=0$ with $\chi(S_t)=\chi_0$ such that $(u_{t_0},S_{t_0})$ has either $(a)$ a boundary crossing or $(b)$ an interior intersection with $L$ and such that $(u_t,S_t)$ is 0-generic for $t\ne t_0$ and such that Lemma \ref{l: generic solutions 1} holds for $(u_t,S_t)$.

\begin{lemma}\label{l: 1-generic nodal solutions}
    If $\hat\theta_t$ is the local extension of $\theta_t$ then $(\hat u_{t_0},\hat S_{t_0})$ solves $\bar\partial_J-\hat\theta_t=0$ and the solutions $(\hat u_t,\hat S_t)$ to $\bar\partial_J-\hat\theta_t=0$ in $V$ constitute a transversely cut out 1-generic family, in particular $(\hat u_t,\hat S_t)$ is 0-generic for $t\ne t_0$.
\end{lemma}

\begin{proof}
It is immediate that $(\hat u_{t_0},\hat S_{t_0})$ solves $\bar\partial_J-\hat\theta_t=0$. Since $\bar\partial_J -\hat\theta_t$ is $\mathrm{sc}$-Fredholm, it sufficies by \cite[Theorem 3.13]{HWZ} to show that the linearization $\hat D:=L_{(\hat u_{t_0},\hat S_{t_0})}(\bar\partial_J -\hat\theta_t)$ is surjective and that its kernel $\mathrm{ker}(\hat L)$ is in good position, see \cite[Definition 3.9]{HWZ} with respect to the nodal boundary stratum. We show this by relating the tangent spaces of $\bZ$ to the maps $(\hat u_{t_0},\hat S_{t_0})$ and $(u_{t_0},S_{t_0})$ and the linearizations $D:=L_{(u_{t_0},S_{t_0})}(\bar\partial_J -\theta_t)$ and $\hat D$. 

Consider first case $(a)$. Introduce punctures at the double points $p_1$ and $p_2$ on the boundary of $(u_{t_0},S_{t_0})$. For shorter notation we write $q_k:=u_{t_0}(p_k)$, $k=1,2$, and note that $q_1=q_2$. Then using local coordinates as in \eqref{eq: nbhd basic perturbation} around $(u_{t_0},S_{t_0})$, the tangent space to $(u_{t_0},S_{t_0})$ is given by
\begin{equation}\label{eq: tangent space bdry double pt}
T_{(u_{t_0},S_{t_0})}\bZ_{<\chi_0} \ = \ T\boldsymbol{j}_0 \oplus \R\cdot j_1\oplus \R\cdot j_2 \oplus H^{3,\delta}(S_{t_0},u_{t_0}^\ast TX)\oplus T_{q_1}L\oplus T_{q_2}L\oplus T[0,1],
\end{equation}
where the summands in the right hand side are as follows: $T\boldsymbol{j}_0$ is the finite dimensional space of linearized conformal variations of the domain $S_{t_0}$ without punctures, they can be chosen to be supported outside neighborhoods of the punctures. The tangent vectors $j_k$, $k=1,2$, are linearized conformal variations corresponding to moving the boundary punctures, we represent them as $j_k=du_{t_0}(\bar\partial(\beta_k w_k))$ where $\beta_k$ is a cut off function on $S_{t_0}$ equal to $1$ in a neighborhood of the puncture and equal to $0$ on the rest of $S_{t_0}$, where $w_k$ is a holomorphic vector field in the strip neighborhood of the puncture with a simple pole at the puncture, if we use coordinates $\R_+\times[0,\frac12]$ then the vector field is $w_k(s,t)=e^{2\pi(s+it)}$, see \cite[Section 2.1.1.]{E}. The infinite dimensional subspace $H^{3,\delta}(S_{t_0},u_{t_0}^\ast TX)$ correspond to linearized variations of the map $u_{t_0}$, $T_{q_k}L$, $k=1,2$, are the linear variations of the asymptotic constants (valued in the Lagrangian) at the two punctures, and $T[0,1]$ is the tangent space to parameter space $[0,1]$. 

By assumption Theorem \ref{l: generic solutions 1} $(a)$ holds at $(u_{t_0},S_{t_0})$. This implies first that the kernel of $D$ on \eqref{eq: tangent space bdry double pt} is 3-dimensional: $1$ dimension for the parameterized solution space and $2$ dimensions for moving the marked points. Consider the following three tangent vectors: the two vectors $j_k+c_k$, $k=1,2$, where $c_k\in T_{q_k}L$ is the vector corresponding to the shift of asymptotic constant matching the motion of the marked point along the boundary according to vector field of $j_k$, and the vector $\partial_t+b_1+b_2$, where $\partial_t\in T_{t_0}[0,1]$, $b_k\in T_{q_k}L$ is such that $b_1-b_2$ is a linearized variation independent from the tangent vectors of the two boundary branches at the double point $q_1=q_2$. The assumption second implies that the subspace 
\begin{equation}\label{eq : approx kernel hyp}
\kappa_{\mathrm{hyp}} = \mathrm{span}\{j_1+c_1 \ , \ j_2+c_2 \ , \ \partial_t+b_1+b_2\}    
\end{equation}
is an approximate kernel of $L$ in the sense that $L$ in invertible on its ($L^2$-)complement.    

Consider now the tangent space to $(\hat u_{t_0},\hat S_{t_0})$. We have
\begin{equation}\label{eq: tangent space hyp node}
T_{(\hat u_{t_0},\hat S_{t_0})}\bZ_{<\chi_0-1} \ = \ T\boldsymbol{j}_0 \oplus \R\cdot \hat j_1\oplus \R\cdot \hat j_2\oplus \R\cdot \hat j \oplus H^{3,\delta}(S_{t_0},u_{t_0}^\ast TX)\oplus T_qL \oplus T[0,1],
\end{equation}
where $q=q_1=q_2$. Here $T_qL$ is the space of linearized asymptotic constants at the node, in terms of \eqref{eq: tangent space bdry double pt}, $T_qL$ is the diagonal in $T_{q_1}L\oplus T_{q_2}L$. The conformal variations $\hat j_1$ and $\hat j_2$ again move the location of the node (thereby changing the conformal structure of $\hat S_{t_0}$ to first order), in terms of \eqref{eq: tangent space bdry double pt}, $\hat j_k=j_k$, and the conformal variation $\hat j$ corresponds to the gluing parameter in the neck region and is the linear variation in the direction normal to the nodal locus. It is represented by $du_{t_0}(\beta w)$, where $\beta$ is a cut off function equal to $1$ in a neighborhood of the punctures and $0$ outside and where $w$ is holomorphic vector field in the strip region (corresponding to translations) that is directed towards both punctures. 

The complement $\hat\gamma^\perp$ of $\hat\gamma$ in $T_{(\hat u_{t_0},\hat S_{t_0})}\bZ_{>\chi_0-1}$ embeds into $T_{(u_{t_0}, S_{t_0})}\bZ_{>\chi_0}$ by including $T_qL$ as the diagonal and taking $\hat j_k$ to $j_k$ and that the $\hat D$ equals the this inclusion composed with $D$. The subspace $T\boldsymbol{j}_0 \oplus H^{3,\delta}(S_{t_0},u_{t_0}^\ast TX)\subset\hat\gamma^\perp$ is perpendicular to $\kappa$ in \eqref{eq : approx kernel hyp}. Let $\tau\subset \hat\gamma^\perp$ be the subspace
\[
\tau_{\mathrm{hyp}} = \mathrm{span}\{j_1 \ , \ j_2 \ , \ T_qL \ , \ \partial_t\}.
\]
Then
\[
\tau_{\mathrm{hyp}}\cap \kappa_{\mathrm{hyp}}=\{0\}
\]
and hence $D$ is invertible on the image of $\hat \gamma^{\perp}$. It follows that $\hat D$ is invertible on $\hat\gamma^\perp$. This then implies that the kernel of $\hat D$ on $T_{(\hat u_{t_0},\hat S_{t_0})}$ is 1-dimensional. Furthermore, noting that if $v\in T_{(\hat u_{t_0},\hat S_{t_0})}\bZ_{>\chi_0-1}$ and $\hat D(v)=0$ then $\langle v,\hat\gamma\rangle\ne 0$, it follows that the kernel is in good position.

Case $(b)$ is similar. Introduce a puncture $p_1$ at the intersection of $(u_{t_0},S_{t_0})$ and $L$, $u_{t_0}(p)=q\in L$. Then using local coordinates as in \eqref{eq: nbhd basic perturbation} around $(u_{t_0},S_{t_0})$, the tangent space to $(u_{t_0},S_{t_0})$ is given by
\begin{equation}\label{eq: tangent space intersection with L}
T_{(u_{t_0},S_{t_0})}\bZ_{<\chi_0} \ = \ T\boldsymbol{j}_0 \oplus \R \cdot j_1\oplus \R\cdot j_2 \oplus H^{3,\delta}(S_{t_0},u_{t_0}^\ast TX)\oplus T_qX \oplus T[0,1],
\end{equation}
where $j_1$ and $j_2$ are the two directions of linearized motions of the interior puncture and $T_qX$ are the linearized variations of the asymptotic constant at $p$.

By assumption, Theorem \ref{l: generic solutions 1} $(b)$ holds at $(u_{t_0},S_{t_0})$ which means that the kernel of $D$ is 3-dimensional: $1$ dimension for the parameterized solution space and $2$ dimensions for moving the marked point. 
Consider the following three tangent vectors: the two vectors $j_k+c_k$, $k=1,2$, where $c_k\in T_{q} X$ is the vector corresponding to the shift of asymptotic constant matching the motion of the marked point according to the vector field of $j_k$, and the vector $\partial_t+b_1$, where $b_1\in T_{q_1}X$ is such that $b_1$ is a linearized variation independent from $T_{q_1} L$ and the tangent space to the curve at $q_1$, i.e. $T{q_1}L+\R\cdot c_1+\R\cdot c_2$. The assumption second implies that the subspace 
\begin{equation}\label{eq : approx kernel ell}
\kappa_{\mathrm{ell}} = \mathrm{span}\{j_1+c_1 \ , \ j_2+c_2 \ , \ \partial_t+b_1\}    
\end{equation}
is an approximate kernel of $D$ in the sense that $L$ in invertible on its ($L^2$-)complement.

Consider next the tangent space to $(\hat u_{t_0},\hat S_{t_0})$:
\begin{equation}\label{eq: tangent space ell node}
T_{(\hat u_{t_0},\hat S_{t_0})}\bZ_{<\chi_0-1} \ = \ T\boldsymbol{j}_0 \oplus \R\cdot \hat j_1\oplus \R\cdot \hat j_2 \oplus\R\cdot\hat j H^{3,\delta}(S_{t_0},u_{t_0}^\ast TX)\oplus T_q L\oplus T[0,1].
\end{equation}
Here $T_qL$ is the space of linearized asymptotic constants at the elliptic node. The conformal variations $\hat j_k$ moves the boundary node and thereby changes the conformal structure of the domain and the conformal variation $\hat j$ corresponds to the gluing parameter at the elliptic boundary node. It is represented by $du_{t_0}(\beta w)$, where $\beta$ is a cut off function equal to $1$ in a neighborhood of the punctures and $0$ outside and where $w$ is holomorphic vector field in the cylinder (corresponding to translations) that is directed toward the puncture.  

The complement $\hat\gamma^\perp$ of $\hat\gamma$ in $T_{(\hat u_{t_0},\hat S_{t_0})}\bZ_{>\chi_0-1}$ embeds into $T_{(u_{t_0}, S_{t_0})}\bZ_{>\chi_0}$ by including $T_qL\subset T_q X$ and taking $\hat j_k$ to $j_k$, and $\hat D$ equals the this inclusion composed with $D$. The proof is then entirely parallel to the hyperbolic case:
the subspace $T\boldsymbol{j}_0 \oplus H^{3,\delta}(S_{t_0},u_{t_0}^\ast TX)\subset\hat\gamma^\perp$ is perpendicular to $\kappa_{\mathrm{ell}}$ and the finite dimensional space 
\[
\tau_{\mathrm{ell}} = \mathrm{span}\{j_1 \ , \ j_2 \ , \ T_qL \ , \ \partial_t\}.
\]
intersects $\kappa_{_{\mathrm{ell}}}$ trivially. It follows as before that the the kernel of $\hat D$ is 1-dimensional and that it is in good position.  
\end{proof}

Since the perturbation $\theta_t$ vanishes near $L$ we can show that the families in Lemma \ref{l: 1-generic nodal solutions} have standard boundary crossings/hyperbolic nodes or Lagrangian crossings/elliptic nodes. More precisley, since Floer gluing which underlies \cite[Theorem 3.13]{HWZ} is local and the perturbation $\theta_t$ and its local extension $\hat\theta_t$ do not have any support near $L$ and are elsewhere identified, the result on solutions to $\bar\partial_J-\theta_t=0$ and $\bar\partial_J-\hat\theta_t=0$ follows from the gluing result for the homogeneous ($\theta_t=\hat\theta_t=0$) equation which is \cite[Theorem 4.2]{SOB}. Likewise the more detailed result \cite[Proposition 4.19]{SOB} on generic boundary behavior for the moduli spaces holds as well. We have the following.

\begin{lemma}\label{l : standard crossings/nodes in 1 parameter families}
  The 1-parameter families of solutions to $\bar\partial_J=\theta_t$ near $(u_{t_0},S_{t_0})$ and $\bar\partial_J=\hat\theta_t$ near $(\hat u_{t_0},\hat S_{t_0})$ have standard boundary crossings/hyperbolic nodes in case $(a)$, see \cite[Definition 3.1]{SOB} and the standard Lagrangian crossing/elliptic nodes in case $(b)$, see \cite[Definition 3.2]{SOB}. 
  
  Furthermore, if $a\in [0,\epsilon)$ is the gluing parameter of the domain $\hat S_{t_0}$ and if $\tau\in (-\delta,\delta)$ parameterizes the section $\theta_t$ around $t=t_0$, then there exists non-zero constants $c_0,c_1\in\R$ such that the following holds:
  \begin{itemize}
  \item The germ at $\tau=0$ of the projection of the crossing family $(u_{\tau+t_0},S_{\tau+t_0})\mapsto \tau\in (-\delta,\delta)$ agrees up to first order with the curve $\tau\mapsto c_0\tau$.
  \item The germ at $(a,\tau)=(0,0)$ of the projection of the nodal family $(\hat u_{\tau+t_0},\hat S_{\tau+t_0})\mapsto (a,\tau)\in [0,\epsilon)\times (-\delta,\delta)$ agrees up to second order with the curve $a\mapsto (a,c_1a^2)$.
  \end{itemize}
\end{lemma}

\begin{proof}
The first part is \cite[Theorem 4.2]{SOB} proved in \cite[Sections 4.3.1--5]{SOB}. The second part is \cite[Proposition 4.19]{SOB}.
\end{proof}

\subsection{Stability of transverse solutions}\label{ssec : stab of transverse}
Here we record basic consequences of the stability of transversality, which, in the `polyfold' context, are justified by the `Fredholm package' of \cite[Chap. 5]{HWZ}

\begin{lemma} \label{estimate openness of genericity} 
Fix a bare and $0$-generic map $(u, S)$. Suppose given locally in some chart $\mathrm{sc}^+$ sections
$\theta, \eta$ of $\mathbf{W}\to\mathbf{Z}$.  Assume $(u, S)$ is a transverse solution to  $\bar\partial_J(u,S)-\theta(u,S)=\eta(u,S)$, and $\eta$ has sufficiently small  auxillary norm.

Then there is a neighborhood of $(u, S)$ and $\delta=\delta((u,S),\eta) > 0$ such that for any $\mathrm{sc}^+$-section $\eta'$ with $|\eta' - \eta| < \delta$, where $|\cdot|$ denotes the auxiliary norm, the equation
$\bar \partial_J (u',S') -\theta(u',S')= \eta'(u',S')$ has a unique solution $(u',S')$ in the neighborhood, which is moreover
bare, transversely cut out, and 0-generic.
\end{lemma} 

\begin{proof}
Since $\bar \partial_J$ is sc-Fredholm and $\theta, \eta$ are $\mathrm{sc}^+$, the various appearing linear combinations remain sc-Fredholm. 
By assumption, $\bar\partial -\theta$ is in general position with respect to
$\eta$. As remarked in the proof of \cite[Theorem 5.8]{HWZ}, general position persists under small perturbation, therefore if the sections $\eta$ and $\eta'$ are sufficiently close, $\bar\partial -\theta$ is in general position with $\eta'$ as well and there is a natural 1-1 correspondence between intersections of $\bar\partial -\theta$ and $\eta$ and  intersections of $\bar\partial -\theta$ and $\eta'$. 

It remains to check that $(u',S')$ is 0-generic. In the local chart, the auxiliary norm controls the $(3,\delta)$-distance, $\|\eta-\eta'\|_{3,\delta}$, between $\eta$ and $\eta'$. By elliptic regularity, $\|\eta-\eta'\|_{3,\delta}$ controls the $C^1$-distance between solutions $u$ and $u'$. Since $0$-generic is an open condition in the $C^1$-norm and $u$ is $0$-generic, so is $u'$ provided $\|\eta-\eta'\|_{3,\delta}$ is sufficiently small. The lemma follows.  	
\end{proof}

We also have the $1$-parametric counterpart. 

\begin{lemma} \label{estimate openness of genericity 1} 
Fix a bare $1$-generic map $(u_{t_0},S_{t_0})$. Supposed given locally in some chart around $(u_{t_0},S_{t_0})\times\{t_0\}$ 1-parametric $\mathrm{sc}^+$- sections $\theta_t,\eta_t$ of $\bW\times[0,1]\to\bZ\times[0,1]$. Assume that $(u_{t_0},S_{t_0})$ belongs to a transversely cut out 1-parameter family of solutions $(u_t,S_t)$ to $\bar\partial_J-\theta_t=\eta_t$ and that $\eta_t$ has sufficiently small auxiliary norm.	Then there is a neighborhood of $(u_{t_0},S_{t_0})$ and 
	$\delta=\delta(u, \eta_t) > 0$ such that for any smooth 1-parameter family of $\mathrm{sc}^+$-sections $\eta'_t$ with $|\eta'_{t} - \eta_{t}| < \delta$, where $|\cdot|$ denotes the auxiliary norm, the equation
	$\bar \partial_{J} - \theta_t =\eta'_{t}$ has a $1$-parameter family of solutions $(u_{t}',S_t')$ in the chart around $(u_{t_0},S_{t_0})$, which is conjugated by $C^1$-diffeomorphisms to the family $(u_{t},S_{t})$ in some neighborhood of $(u_{t_{0}},S_{t_0})$. In particular, the family $(u_{t}',S_t')$ is 1-generic, and transversely cut out.
\end{lemma} 

\begin{proof}
The proof is analogous to the proof of Lemma \ref{estimate openness of genericity}:
	the fact that the $1$-parameter family $(u'_t,S'_t)$ is transversely cut out is a consequence of stability of general position between $\bar\partial_J-\theta$ and $\eta$. To see that the 1-parameter families of solutions are conjugated by diffeomorphisms, note that the norm distance between $\eta_t$ and $\eta'_t$ controls the $C^1$-distance between $u_t$ and $u_t'$, again by elliptic regularity.  
\end{proof}

\section{Ghost bubble censorship} 
We turn to the problem of how to produce perturbations of the holomorphic curve equation with the property that the locus of bare solutions is compact.   
The main tool is the ghost bubble censorship principle \cite{ghost}, which we now review and adapt.
Let $X$ be a symplectic manifold, $L\subset X$ a Lagrangian, and 
$J$ an almost complex structure such that $L$ is 
locally the fixed locus of an anti-holomorphic involution. 

\subsection{Control near ghosts and ghost bubble formation}
Consider a sequence of bare maps $(u_\alpha,S_\alpha)$,
$u_\alpha\colon (S_\alpha,\partial S_{\alpha}) \to (X,L)$ which Gromov converges to a stable map $u\colon (S,\partial S) \to (X,L)$
with ghost components. 
Let $u_+\colon S_+ \to X$ be the restriction of $u$ to the components of positive symplectic area, 
and let $S^- = \overline{S \setminus S_+}$.  
Let $S^-_\alpha \subset S$ be the corresponding subset of $S_\alpha$, and assume $\bar \partial_J u_\alpha|_{S_\alpha^-} = 0$.  
\begin{definition} \label{sufficiently controlled}
The sequence $(u_\alpha,S_\alpha)$ is {\em controlled} if, around each node separating $S^-$ from $S \setminus S^-$, 
there exist neck coordinates $[-\rho_\alpha, \rho_\alpha] \times I \subset S_\alpha$, 
constants $\xi_{\alpha,\mathrm{p}},\xi_{\alpha,\mathrm{n}} \in \R^{2n}$ and exponential weights $\delta' > 0$ and $\delta > \pi$, such that if $u_{\alpha,\mathrm{p}}(s,t)=u_{\alpha}(s+\rho_\alpha,t)$ and $u_{\alpha,\mathrm{n}}(s,t)=u_{\alpha}(s-\rho_\alpha,t)$ then
		\begin{align}
			\label{limit compatible 2} 
			&\|(\bar\partial_{J}u_{\alpha,\mathrm{p}} \ - \ \xi_{\alpha,\mathrm{p}}\otimes d\bar z)|_{[- \rho_\alpha+1, 0] \times I}\|_{2,2\pi + \delta'} \\\notag
			&\qquad\quad+ 
			\|(\bar\partial_{J} u_{\alpha,\mathrm{n}} \ - \ \xi_{\alpha,\mathrm{n}}\otimes d\bar z) |_{[0, \rho_\alpha-1] \times I}\|_{2,2\pi + \delta'} = \mathcal{O}(1).
		\end{align}  
		\begin{equation}
			\label{cut-off decay}
			\|\bar\partial_{J} u_{\alpha}|_{[-1,1]\times I}\|_{0} = {\mbox{\tiny$\mathcal{O}$}} (e^{-2\pi\rho_{\alpha}}).
		\end{equation}
		\begin{equation}
			\label{cut-off decay2}
			\|\bar\partial_{J} u_{\alpha}|_{[-1,1]\times I}\|_{2} = \mathcal{O}(e^{-\delta \rho_{\alpha}}).
		\end{equation}
\end{definition} 

\begin{definition}
A map $u\colon S \to X$ is a {\em $\partial_J$-generic immersion} if 
$\partial_J u$ never vanishes, the two values of $\partial_J u$ are linearly independent at double points, 
and there are no triple points.  
\end{definition}

\begin{theorem} \cite[Theorems 1.5, 1.8, 1.10]{ghost} \label{ghost bubble censorship} 
Given a controlled ghost bubble formation as in Definition \ref{sufficiently controlled}, 
the bare part of the map $u_{+}\colon S_{+} \to X$ is \emph{not} a $\partial_J$-generic immersion. 
\end{theorem}

Recall that a standard chart $\bU \to \bZ$ centered around a map $u_0$ is determined in terms of the data
of fixed coordinate
charts $\R^{2n} \to X$ around each node of $u_0$, and neck parameterizations 
$[-\rho, \rho] \times I \to S$ for nearby curves, and in particular neck length functions $\rho\colon \bU \to (0, \infty]$.

\begin{definition}
In a standard chart $\bU$
centered around a map $(w,T)$ with constant ghosts, 
a section $\lambda\colon\bU \to \bW|_{\bU}$ is {\em controlled near ghosts}  if, for  $(u,S)\in\bU$,
$\lambda(u)$ vanishes on all ghost components of $S$, and
for each neck separating the ghost locus of $(w,T)$ 
from the bare locus,  there are functions 
$\xi_{\mathrm{p}}, \xi_{\mathrm{n}}\colon \bU \to \R^{2n}$ and exponential weights $\delta' > 0$ and $\delta > \pi$, such that, if $\lambda(u)_{\mathrm{p}}(s,t)=\lambda(u)(s+\rho_\alpha,t)$ and 
$\lambda(u)_{\mathrm{n}}(s,t)=\lambda(u)(s-\rho_\alpha,t)$
\begin{align}
	\label{limit compatible 2 section} 
	&\|(\lambda(u)_{\mathrm{p}}  \ - \ \xi_{\mathrm{p}}(u) \otimes d\bar z)|_{[- \rho+1, 0] \times I}\|_{2,2\pi + \delta'} \\\notag
	&\qquad\quad+ 
	\|( \lambda(u)_{\mathrm{n}} \ - \ \xi_{\mathrm{n}}(u)\otimes d\bar z) |_{[0, \rho-1] \times I}\|_{2,2\pi + \delta'} = \mathcal{O}(1).
\end{align}  
\begin{equation}
	\label{cut-off decay section}
	\|\lambda(u)|_{[-1,1]\times I}\|_{0} = {\mbox{\tiny$\mathcal{O}$}} (e^{-2\pi\rho}).
\end{equation}
\begin{equation}
	\label{cut-off decay2 section}
	\|\lambda(u)|_{[-1,1]\times I}\|_{2} = \mathcal{O}(e^{-\delta \rho}).
\end{equation}
The implicit constants in the estimates  depend on the section $\lambda$, but of course not on $\rho$.  
\end{definition}

Let $\bU$ and $\bV$ be standard charts around $(w,T)$ with constant ghosts. Consider $\lambda|_{\bU}\colon \bU\to\bW|_{\bU}$. Then $\lambda|_{\bU}$ induces a section $\tilde\lambda|_{\bV\cap\bU}\to \bW|_{\bV\cap\bU}$, where we think of $\bU\cap\bV\subset \bV$, so that $\tilde \lambda$ is a section in the local coordinates determined by $\bV$, defined over a subset.

\begin{lemma}\label{l:some=any}
	  If $\lambda|_{\bU}$ is controlled near ghosts then $\tilde \lambda|_{\bU\cap\bV}$ is controlled near ghosts.
\end{lemma}

\begin{proof}
The change of coordinates determines a conformal equivalence between ends $\psi\colon [0,\infty)\times I\to [\rho_{0},\infty)\times I$ which is a translation and a rotation if $I=S^{1}$, up to exponentially small error, $(s,t)\mapsto (s+c,t+\theta) + \mathcal{O}(e^{-2\pi s})$, see Remark \ref{r : delta weight well def}. Let $\Phi\colon\R^{2n}\to\R^{2n}$ be the change of coordinates around the node corresponding to $\bU$ and $\bV$. Then
\[ 
d\tilde u = d\Phi(u)\circ du\circ d\psi 
\] 
and the section induced from $\xi_{\bullet}\otimes d\bar z$, $\bullet=\mathrm{p}$ or $\bullet=\mathrm{n}$ is 
\[ 
d\Phi(u)(\xi_{\bullet})\otimes d\psi d\bar z.
\]
It follows by boundedness of $d\Phi$ and its derivatives that \eqref{cut-off decay section} and \eqref{cut-off decay2 section} hold and also that
\begin{align} 
	&\|(\tilde \lambda(u)_{\mathrm{p}}  \ - \ d\Phi(u)(\xi_{\mathrm{p}})\otimes d\bar z)|_{[- \rho+1, 0] \times I}\|_{2,2\pi + \delta'} \\\notag
	&\qquad\quad+ 
	\|( \tilde\lambda(u)_{\mathrm{n}} \ - \ d\Phi(u)(\xi_{{\mathrm{n}}})\otimes d\bar z) |_{[0, \rho-1] \times I}\|_{2,2\pi + \delta'} = \mathcal{O}(1).
\end{align}
Then, if $\tilde\xi_{\bullet}=d\Phi(u(\infty))\xi_{\bullet}$ we have,
\[
\|(d\Phi(u)(\xi_{\bullet})\otimes d\bar z-\tilde\xi_{\bullet}\otimes d\bar z)|_{[- \rho+1, 0] \times I}\|_{2,2\pi + \delta'}=\mathcal{O}(1),
\]
by Taylor expanding $d\Phi$. The lemma follows.
\end{proof}

\begin{definition}\label{control}
We say that a section $\lambda\colon \bZ \to \bW$ is {\em controlled near ghosts} if $\lambda$ vanishes on all ghost components, 
and is controlled near ghosts
when restricted to some (equivalently any by Lemma \ref{l:some=any}) chart centered around each map with constant ghosts. 
\end{definition}

It is obvious from the definition that if $\lambda$ is  controlled near ghosts, then 
any sequence of solutions to $\bar \partial = \lambda$ will satisfy the hypothesis of 
Theorem \ref{ghost bubble censorship} and hence also the conclusion. 

We will need the following slight generalization of Theorem \ref{ghost bubble censorship}.  
\begin{proposition}\label{ghost bubble censorship linear combo}
Let $\theta\colon \bZ\to\bW$ be a section which is a finite linear combination $\theta=\sum_{k=1}^{m}\lambda_{k}$, where all $\lambda_{k}$ are controlled near ghosts. Let $(u_{\alpha},S_{\alpha})$ be a sequence of bare solutions to $\bar\partial_{J}(u_{\alpha},S _{\alpha})=\theta$ that Gromov converges to a nodal solution $\bar\partial_{J}(u,S)=\theta$.
Then the bare part $(u_{+},S_{+})$ of $(u,S)$ is not a $\partial_J$-generic immersion.	
\end{proposition}

\begin{proof}
 We explain the modifications needed to the proof in \cite{ghost} of Theorem \ref{ghost bubble censorship}.
Consider a neck region $[-\rho_{\alpha},\rho_{\alpha}]\times I$, $\rho_{\alpha}\to\infty$ as $\alpha\to\infty$ near the node in the limit. In these coordinates there are sharp cut-off regions $[\kappa_{\alpha;j},\kappa_{\alpha;j}+1]$ where the perturbation $\lambda_{j}$ is cut off, where $\kappa_{\alpha;j}$ are uniformly bounded as $\alpha\to\infty$. After a bounded translation of coordinates we assume that the largest $\kappa_{\alpha;j}$ satisfies $\kappa_{\alpha;j}=0$ and that the smallest is uniformly bounded by some fixed $\kappa<0$.  

The solutions to $\bar\partial_{J}(u_{\alpha},S_{\alpha})=\theta$ satisfy the main hypothesis of \cite[Proposition 9.1]{ghost}, i.e., $\|u_{\alpha}\|_{C^{0}}\le \epsilon$ and $\|\bar\partial_{J} u_{\alpha}\|_{1,\delta}\le \epsilon$ in the neck-region. The proposition is then a consequence of the arguments in \cite{ghost} once we show that solutions $\bar\partial_{J}(u_{\alpha},S_{\alpha})=\theta$ have $+$-reasonable necks, see \cite[Definition 1.3]{ghost}. The only difference when considering a finite linear combination, as compared to only one perturbation, appears in \cite[Lemma 4.2]{ghost}. There, the $\bar\partial$-equation is explicitly solved in the neck region and it is shown that the cut-off decay condition (i.e., the sharp cut-off) leads to only a small change in the Fourier coefficient corresponding to the complex derivative. For a linear combination of perturbations, we have $m$ such regions in the uniformly bounded interval $[\kappa,1]$. The total change in (rescaled) Fourier coefficient is then the sum of these (rescaled) changes and we conclude that the complex linear derivative Fourier coefficient in the region $\rho>1$,  where the map is unperturbed holomorphic, is similarly close enough to the corresponding coefficient in the region $\rho<\kappa$, where the map converges to the perturbed solution on the compact disk. The proposition follows.  
\end{proof}

To see the effect of bubble censorship on the solution spaces, let $\bG\subset\bZ$ denote the locus of stable maps with at least one ghost component and for which all ghost components are constant. Let $\theta$ be an $\mathrm{sc}^+$-section such that $\bar\partial_J^{-1}(\theta)$ is compact (e.g., if $\theta$ has sufficiently small auxiliary norm then $\bar\partial_J^{-1}(\theta)$ is compact by \cite[Theorem 5.1]{HWZ} as in the proof of Lemma \ref{bare transversality}). 
\begin{corollary}\label{c: no bare solutions near ghosts} 
If $\theta$ is a finite sum of $\mathrm{sc}^+$-sections that are controlled near ghosts then all solutions $\bar \partial_J = \theta$ have constant ghosts. If furthermore any non-bare solution to $\bar \partial_J=\theta$ is a $\partial_J$-generic immersion when restricted to the union of its positive area components, then there is a punctured neighborhood of the locus $\bG$   
which contains no solutions to $\bar \partial_J = \theta$.  In particular, the locus of bare solutions to $\bar\partial_J=\theta$ in $\bZ$ is compact.   
\end{corollary} 
\begin{proof}
The first statement follows since $\theta$ vanishes on any zero area component: any ghost component of a solution to $\bar\partial_J=\theta$ is the $J$-holomorphic of area zero and hence constant.
  
To see the second statement, if there is no punctured neighborhood of $\bG$ without solutions then there is a sequence of solutions $(u_\alpha,S_\alpha)$ to $\bar\partial_J=\theta$ that by compactness converges to a solution $(u,S)\in \bG$. Proposition \ref{ghost bubble censorship linear combo} then implies that the restriction $u_+\colon S_+\to X$ of $u$ to the union of its positive area components is not a $\partial_J$-generic immersion, contradicting the hypothesis. 
\end{proof} 

The same result holds with identical proof for parameterized sections. 

\begin{corollary}\label{c: no bare solutions near ghosts parameterized} 
Let $M$ be a smooth compact manifold and 
assume that $\theta_m\colon\bZ\times M\to\bW\times M$ is a parameterized $\mathrm{sc}^+$-section such that $\bar\partial_J^{-1}(\theta_m)$ is compact.

If $\theta_m$ is a finite sum of parameterized $\mathrm{sc}^+$-sections that are controlled near ghosts then all solutions $\bar \partial_J = \theta_m$ have constant ghosts. If furthermore any non-bare solution to $\bar \partial_J=\theta_m$ is a $\partial_J$-generic immersions when restricted to the union of its positive area components, then there is a punctured neighborhood of the locus $\bG\times M$   
which contains no solutions to $\bar \partial_J = \theta_m$.  In particular, the locus of bare solutions to $\bar\partial_J=\theta_m$ in $\bZ\times M$ is compact.   \qed
\end{corollary}

\subsection{Constructing sections controlled near ghosts}\label{ssec:sharpgluing}
We discuss how to produce 
sections which are controlled near ghosts.  Recall that the gluing
operation $\oplus$ serves to provide coordinates on collar neighborhoods of the nodal strata of configuration spaces $\mathbf{Z}$.  The corresponding hat-gluing operations for sections $\widehat{\oplus}$ can be used to extend sections from the nodal locus to the collar neighborhood. 

We require a variant
of the hat-gluing operation from \cite{HWZ-GW},
which differs from the original solely in the 
choice of cut-off, where we pick a cut-off so that the perturbations we construct are controlled near ghosts (the main point being to arrange that \eqref{cut-off decay} holds).

More precisely, recall that, as in \cite[Section 1.2]{HWZ-GW}, in Section \ref{sec:formaldiff} there is a choice of cut-off function 
$\beta\colon\R\to[0,1]$, a smooth function with derivative supported in $(-1,1)$, that is used to interpolate between formal complex anti-linear differentials over the neck region for domains in a neighborhood of a nodal domain. We will use instead a family of cut-off functions $\beta^{R}$ that goes from $1$ to $0$ over the shorter interval $\left(-\frac{1}{R},\frac{1}{R}\right)$. We define the \emph{sharp gluing}  $\widehat{\oplus}_{a}^{R}$ by
replacing $\beta$ with $\beta^{R}$ for $R=\varphi(a)=e^{\frac{1}{a}}-e$.  

We next show that sharp gluing respects regularity conditions.
Consider gluing at a node as in Sections \ref{ssec:gluing} and \ref{sec:formaldiff} where we join semi-infinite cylinders $C_\pm=\R_\pm\times I$. The cylinders $C_\pm$ come equipped with weight functions $e^{\delta|s|}$ used to define the integral norms in $H^{3,\delta}$.
The gluing region $$Z_{a} := \{0\le s\le R\} = \{-R\ge s'\ge 0\}$$
then has a natural induced weight functions $w_{a;\delta}\colon Z_{a}\to [0,\infty)$:   
\[
w_{a;\delta}(s,t)= e^{\delta|s-\frac12 R|}, \;\;s\in[0,R],
\]
which is the restriction of the weight functions $e^{\delta|s|}$ on the end $C_{+}$ for $s\in[0,\frac12R]$ and on the end $C_-$ for $s\in[\frac12 R,R]$. We write $H^{3,\delta}(Z_{a},\R^{2n})$ for the corresponding Sobolev spaces weighted by $w_{\delta}$, and we denote the weighted Sobolev norm $\|\cdot\|_{k;\delta}$. We point out that as (sc) Banach spaces $H^{3,\delta}(Z_{a},\R^{2n}) \approx H^{3}(Z_{a},\R^{2n})$, but the isomorphism is not uniform in $a$. We call the norms in $H^{3,\delta}(Z_{a})$ \emph{restricted weight norms}. (When we extend basic perturbations below, restricted weighted norms will be used to  separate out asymptotic constants, see Section \ref{ssec:modifiedpolyfoldcharts}.)  

\begin{lemma}\label{l:gluepertfast}
	If $\xi^{+}$ and $\xi^{-}$ lie in $H^{k,\delta}(\R^{\pm}\times I)$ and $\delta>\delta'$ then $\xi^{+}\widehat{\oplus}_{a}^{R}\xi^{-}$ lies in $H^{k,\delta'}(Z_{a})$  uniformly as $R\to\infty$.
\end{lemma}

\begin{proof}
	To see this we note that the effect of replacing $\beta$ by $\beta^{R}$ on the $k^{\rm th}$ derivative of the hat-glued map is multiplication by a linear combination of terms of the form  
	\begin{equation}\label{eq: sharp gluing derivatives} 
	\varphi^{(j)}(R(a))\frac{d^{k_{1}}R}{d a^{k^{1}}}\dots \frac{d^{k_{j}}R}{d a^{k^{j}}}, 
	\end{equation}
	where $\sum k_{j}=k$. Such factors can be estimated as 
	\begin{equation}\label{eq: sharp gluing derivatives estimated}
	(e^{\frac{1}{a}})^{2k}P_{k}(\tfrac{1}{a}),
	\end{equation} 
	where $P_{k}$ is a polynomial. As $e^{(\delta-\delta')(e^{\frac{1}{a}})}$, $\delta'<\delta$  grows faster than such terms, the statement follows.  
\end{proof}

Lemma \ref{l:gluepertfast} shows in particular that sharp gluing has the properties of usual gluing for the bundle $\bW$. Let $\xi_\pm\in H^{2,\delta}(C_\pm,\mathrm{Hom}^{0,1}(TC_\pm,T\R^{2n}))$.  

\begin{corollary}\label{cor : sharp hat gluing scp local}
The projection $(\xi_+,\xi_-)\mapsto \xi_+\widehat\oplus_a^R \xi_-$ induced by sharp gluing is $\mathrm{sc}$-smooth and if $\xi_\pm$ are $\mathrm{sc}^+$-sections then so is $\xi^+\widehat \oplus_a^R\xi^{-}$. 
\end{corollary}
\begin{proof}
After Lemma \ref{l:gluepertfast}, we may repeat the argument given in \cite[Section 3.6]{HWZ} for their  $\widehat{\oplus}_a$.
\end{proof}

Let $\bU$ be a local chart centered around some particular $(u, S)$.  Fix some subset $N$ of the nodes of $S$, 
and let $(\widetilde u, \widetilde S)$ be the corresponding map from the partial normalization of $S$ along $N$ (leaving punctures where the nodes were), 
and let $\bU_{\widetilde{N}}$ be a corresponding local chart around $(\widetilde u, \widetilde S)$. 

Let $\bU_N \subset \bU$ 
be the locus along which the nodes $N$ remain nodal, and let $n\colon \bU_N \hookrightarrow \bU_{\widetilde{N}}$ 
be the natural embedding.  There is a canonical identification $n^*\bW = \bW$. Let $\mathrm{sc}^+(\bU, \bW)$ denote the set of $\mathrm{sc}^+$-sections over $\bU$. As with the usual hat-gluing, we have the following result which is a direct consequence of Corollary \ref{cor : sharp hat gluing scp local}:

\begin{corollary} \label{sharp hat gluing scp} 
The sharp hat gluing defines a map 
\begin{eqnarray*}
\mathrm{sc}^+(\bU_{\widetilde{N}}, \bW) & \to & \mathrm{sc}^+(\bU, \bW) \\
\lambda & \mapsto & \widehat \oplus^R n^*  \lambda
\end{eqnarray*}
\end{corollary} 

In \cite{ghost}, we showed (translated into our present terminology): 

\begin{proposition} \label{sharp hat gluing is controlled} \cite[Proposition 10.2]{ghost}
    If, in the setting above, the set of nodes $N$ consists of those separating the positive area part and ghost parts of $(u, S)$, and the section $\lambda$ is a basic perturbation (see Definition \ref{def: basic perturbation}) on the positive area part and zero on the ghost part, then $\widehat \oplus^R n^*  \lambda$ is controlled near ghosts. 
\end{proposition}

\subsection{Error estimate}
Consider a basic perturbation $\lambda$ supported in $\bU(u,S)$, we think of domains in $\bU(u,S)$ as the surface $S$ with a complex structure varying in $\boldsymbol{j}$. Consider a marked point $\zeta\in S$ with a disk neighborhood $\zeta\in D_0\subset S$ and attach a constant rational curve $\C \P^1_{m+1}$ with $m+1>2$ distinct marked points $\zeta_r$, $r=0,1,\dots,m$ to $S$ by joining $\zeta$ and $\zeta_0$ to a node. Consider a cylindrical neighborhood $S^1\times [0,\infty)$ of $\zeta$. Let $a_0$ be a gluing parameter at $\zeta$ and $R_0$ the corresponding gluing length in the cylindrical neighborhood. Similarly, fix disk neighborhoods $D_r$ around $\zeta_r$, gluing parameters $a_k$, and corresponding gluing lengths $R_r$ $r=1,\dots,m$, at the marked points of $\C \P^1$.    

If $a_0\ne 0$ and $a_r=0$, $r>0$, we get the domain $S_{a_0}$ with marked points $\zeta_r$, $r=1,\dots,m$ that we view as punctures. On $S_{a_0}$ we have two perturbations. First, the original basic perturbation $\lambda$.  
Second, the perturbation $\lambda^\ast$ obtained by sharp gluing of $\lambda$ over $S$ and $0$ over $\C \P^1_{m+1}$. We write $\|\cdot\|_{3,\delta}$ for the weighted Sobolev norm on $S_{a_0}$ with weight function given by the restricted weight function as in Section \ref{ssec:sharpgluing}.

\begin{lemma} \label{controlled incompatibility}
    For any $\eta>0$, we have the following estimate for perturbations on $S_{a_0}$ 
    \[
    \|\lambda - \lambda^*\|_{3,\delta} =\mathcal{O}(e^{-(2\pi-\delta-\eta)R_0}),
    \]
    where $R_0$ is the gluing length corresponding to $a_0$. In particular, $\|\lambda - \lambda^*\|_{3,\delta}\to 0$ on $S_{a_0}$ as $a_0\to 0$. 
\end{lemma}
\begin{proof}
    Recall the form of the basic perturbation $\lambda(\tilde v)=V(\tilde v)\alpha(z)$, where $V$ is a smooth vector field and $\alpha=A(z)d\bar z$ a smooth $(0,1)$-form on $S$. We have $\lambda=\lambda^\ast$ in $S\setminus D_0$. Consider the perturbations inside $D_0$ in the cylindrical coordinates $S^1\times[0,\infty)$. Here the perturbations agree on $S^1\times[0,R_0)$ and it remains to estimate the difference in $[R_0,\infty)\times S^1$. Let $s+it\in[R_0,\infty)\times S^1$ be coordinates then $d z=e^{-2\pi(s+it)}(ds + idt)$ and we have:  
    \[
    \lambda(s,t)= V(s,t)A(s,t) e^{-2\pi(s-it)}(ds-idt)
    \]
    and 
    \[
    \lambda^\ast(s,t) = \beta(s)V(s,t)A(s,t) e^{-2\pi(s-it)}(ds-idt),
    \]
    where $V(s,t)=V(\tilde v(s,t))$, $A(s,t)=A(e^{-2\pi(s+it)})$, and $\beta$ is the sharp cut off function equal to $1$ on $\{R_0\}\times S^1$ and equal to $0$ on $[R_0+R_0^{-1},\infty)\times S^1$. 

    Standard arguments give estimates on $V(\tilde v(s,t))A(s,t)$: we use the chain rule, boundedness of derivatives of the smooth vector field $V$, the exponential map, and the smooth function $A$, and Soboloev estimates for products involving derivatives of $v$, see \cite{Moser}. These estimates, together with the sharp cut-off estimates \eqref{eq: sharp gluing derivatives} and \eqref{eq: sharp gluing derivatives estimated} imply
    \[
\|\lambda-\lambda^\ast\|_{3,\delta}=\mathcal{O}(e^{-(2\pi-\eta-\delta)R_0})
    \]  
    for any $\eta>0$.
\end{proof}

\section{Extensions of basic perturbations}\label{sec: extension of basic}
In this section we give a construction which `extends' a basic perturbation centered at some $(u, S)$ to a neighborhood of the locus of maps whose positive area part lies in the defining chart $\bU(u,S)$, in such a way that (as we will see later) will be suitable for an inductive construction of perturbations suitable for the ghost bubble censorship argument. The main idea is that we will stratify according to the number of points where ghosts are attached to $S$, as opposed, e.g., to the total number of nodes.  

\subsection{Cutoff adapted to a stratification} \label{cutoffs and extensions}
Let $M$ be a smooth compact manifold (possibly with boundary and corners), equipped with some closed locus $M_*$ carrying a decomposition into smooth locally closed submanifolds $M_* = \bigsqcup_{i\ge 0} M_i$.   We assume that the stratification satisfies the following weak frontier property: 
$$\overline{M}_n \setminus M_n \subset \bigsqcup_{m < n} M_m$$
Compatibly with this we write $M_\infty:= M \setminus M_*$.  

We write $\nu_i$ for the normal bundle to $M_i$, and assume that a tubular neighborhood parameterization near the zero section $\exp_i\colon \nu_i \to M$ and a (fiberwise) metric on the $\nu_i$ are given.  

Let $\sigma\colon [0,1] \to [0,1]$ be some fixed cutoff function, taking the value $1$ in some neighborhood of $0$, and $0$ in a neighborhood of $1$.  We write $\sigma_\epsilon$ for the rescaling of the domain to $[0, \epsilon]$, i.e. $\sigma_\epsilon(x) := \sigma(x/\epsilon):[0,\epsilon] \to [0,1]$.  

By hypothesis, $M_0$ is closed. So for some sufficiently small $\epsilon_0$, the map $\exp_0$ is defined on the neighborhood of radius $\epsilon_0$.  While $M_1$ is not closed, the locus 
\begin{equation} \label{away from collar}
    M_1^\circ:= M_1 \setminus \exp_0(\mathrm{Interior}((\sigma_{\epsilon_0} \circ \|\cdot \|)^{-1}(1)))
\end{equation} is closed, so on this locus we may choose some $\epsilon_1$ so that $\exp_1$ is defined on the neighborhood of radius $\epsilon_1$ in $\nu_1$, etc.  Fix such a sequence of $\epsilon_i$.

We define $\kappa_0 = \sigma_{\epsilon_0} \circ \|\cdot \| \circ \exp_0^{-1}$ on the image of $\exp_0$, and $0$ elsewhere.  By construction  it is a smooth function on $M$ supported in a neighborhood of $M_0$ and equal to $1$ on $M_0$. 

We define inductively
\begin{equation} \label{cutoff formula}
    \kappa_i = \left (1 - \sum_{j < i} \kappa_j \right) \cdot (\sigma_{\epsilon_i} \circ \|\cdot \| \circ \exp_i^{-1})  
\end{equation}
on the image of $\exp_i$, and $\kappa_i=0$ elsewhere.

\begin{lemma}\label{l: kappa part of unity}
    Each $\kappa_i$ is smooth. The sum
    $\sum_{j \le i} \kappa_j$ restricts to a partition of unity on a neighborhood of $\bigsqcup_{j \le i} M_j$ and is supported on the complement of $\bigsqcup_{i < k} M_k^\circ$.
\end{lemma}

Another important property (directly visible from the definition) is that the functions $\kappa_j$ depend only on the $\epsilon_i$ parameters, the $\exp_i$ maps, and the norm on the normal bundle.  We will use this later to justify some compatibility assertions.  

We may define cutoff functions by the formula \eqref{cutoff formula} also when $M$ is infinite dimensional as long as all strata have finite codimension and normal bundles satisfy uniformity conditions over strata.  

\begin{remark}
Let $f\colon M_* \to \R$ and $\widetilde{f}_i\colon \nu_i \to M$ be smooth functions such that $\widetilde{f}_i|_{M_i} = f|_{M_i}$.  (E.g., one can take $\widetilde{f}_i := f \circ \mathrm{exp}_i$.)
Let
$$
\widetilde{f}(m) := \sum \kappa_i(m)  \widetilde{f}(\mathrm{exp}_i^{-1}(m)).
$$
Then $\widetilde{f}$ is smooth since all summands are smooth.
    However, in general, $\widetilde{f}|_{M_*} \ne f$: the functions $f$ and $\widetilde{f}$ do not necessarily agree near where the strata meet.  When we use this construction later, we will use Lemma \ref{controlled incompatibility} to estimate the distance between $f$ and $\widetilde{f}$. 
\end{remark}

\subsection{Coordinate corner structures and versal families of curves}\label{sec: cosheaf for families of curves}
In this section we give a pedantically precise  meaning to the following assertion: ``the gluing parameters at nodes give compatible parameterizations of the normal directions along the nodal locus''.

Let $M = \bigsqcup M_\alpha$ be a stratified space.
A {\em coordinate corner structure} is the  data of: 
\begin{enumerate}
    \item A cosheaf of finite sets $\mathfrak{n}$ over $M$, constructible with respect to the stratification.  We will always discuss this  in terms of the exit path presentation: a locally constant family of finite sets over each stratum, along with maps for specialization of strata, compatible for iterated specializations.  We  require  these maps to be injective.  On a given stratum $M_\alpha$, the $\mathfrak{n}_\alpha := \mathfrak{n}|_{M_\alpha}$ will be later identified with  coordinates on the normal bundle; we fix a splitting $\mathfrak{n} = \mathfrak{n}_\C \sqcup \mathfrak{n}_\R$ for real and complex normal directions.  

    Over a stratum $M_\alpha$, we write $\pi_\alpha\colon \Delta_\alpha \to M_\alpha$ for the bundle with fiber $ \R_{\ge 0}^{\mathfrak{n}_{\alpha, \R}} \oplus \C^{\mathfrak{n}_{\alpha,\C}}$.  
    \item     
    Over each stratum 
     of $M_{\alpha}$, a neighborhood of the zero section $T_{\alpha} \subset \Delta_\alpha$ and an embedding
     $\exp_\alpha\colon T_{\alpha} \to M$ which is `the  identity' on $M_\alpha$. 
\end{enumerate}

We require this data to satisfy the following compatibility conditions.  
\begin{enumerate}
\setcounter{enumi}{2}
    \item $\Delta_\alpha$ itself carries a  stratification by coordinate hyperplanes, and a corresponding cosheaf of sets $\mathfrak{n}^{\mathrm{can}}_{\alpha} = \mathfrak{n}^{\mathrm{can}}_{\alpha, \C} \sqcup \mathfrak{n}^{\mathrm{can}}_{\alpha, \R}$ given by the normal coordinates.  Note that $\mathfrak{n}^{\mathrm{can}}_{\alpha}$ is canonically a sub-sheaf of $\pi_\alpha^* \mathfrak{n}_\alpha$.  
    So too is $\exp_\alpha^*(\mathfrak{n})$.  We require that $\exp_\alpha^*(\mathfrak{n}) = \mathfrak{n}^{\mathrm{can}}_{\alpha}$. 

    \item \label{coordinate compatibility} Over $\Delta_\alpha$, there is a natural exponential map $\exp_\alpha^{\mathrm{can}}$ for $\mathfrak{n}^{\mathrm{can}}_\alpha$.  Indeed, $\mathfrak{n}^{\mathrm{can}}_\alpha$ along any given stratum is just the coordinates which vanish along that stratum; we may everywhere identify $\Delta_{\alpha}^{\mathrm{can}}$ with $\pi_\alpha^* \Delta_\alpha$ and the  projection $\pi_\alpha^{\mathrm{can}}$ with forgetting the appropriate coordinates.  We choose the exponential by sending fiber to base coordinates. 
    
    Now suppose given some stratum $M_{\beta}$ in the closure of $M_{\alpha}$.  Then over
    $M_\beta \cap \exp(T_\alpha)$, we require that 
    $\exp_\alpha$ identifies $\exp_\beta$ with $\exp_{\alpha}^{\mathrm{can}}$.  
\end{enumerate}

Consider now a family of nodal curves over a smooth base $M$.  Then $M$ may be given a stratification according as to the number and type of nodes of the curve in question.  We assume the stratification is locally diffeomorphic to the coordinate stratification on some $\R_{\ge 0}^m\times \C^n$.  Over each connected component $M^{i; \alpha}$ of the stratum $M^i$, there is a 
smooth bundle of gluing parameters $\Delta^{i; \alpha} \to M^{i; \alpha}$.  We may also, over $M^{i; \alpha}$, choose a (smoothly varying) choice of disk coordinates around the nodes of the corresponding curves.  Having fixed a gluing profile, this data determines a family of curves over the total space of $\Delta^{i; \alpha}$.  
By a {\em versality witness}, we mean a neighborhood $T^{i; \alpha}$ of the zero section in $\Delta^{i; \alpha}$, and a diffeomorphism $\exp\colon T^{i; \alpha} \to M$, covered by a fiberwise holomorphic diffeomorphism of the corresponding families of curves.  

Note that the family of glued curves over the total space of $\Delta_{i, \alpha}$ comes with a tautological versality witnesses, not just at the zero section, but at all strata.  Thus if $M^{j; \beta} \subset \overline{M^{i; \alpha}}$, then on the intersection $M^{i; \alpha} 
\cap \exp(T^{j; \beta})$, we have two versality witnesses: the one associated to $M^{i; \alpha}$ and the one transported from $\Delta^{j; \beta}$ via the versality witness for $M^{j; \beta}$.  We say our versality witnesses form a compatible system if these two possibilities agree for every such overlap.  

The standard charts on moduli of curves or configuration spaces of maps come with such compatible systems of versality witnesses. (In the typical differential-geometric approaches of constructing moduli, such as the polyfold approach we follow here, this is true by definition/construction of the moduli spaces.  In algebro-geometric settings, it is instead a theorem about the versality of the standard deformation of the node.)

It should be clear from the above discussion that a compatible system of versality witnesses determines a coordinate corner structure: the cosheaf $\mathfrak{n}$ is the cosheaf of nodes, and the compatible system of versality witnesses is the data of the exponential maps.

\subsection{Marked curves}

Our families of curves carry an additional important piece of data:  which components have positive symplectic area and which are ghost components.  

By a {\em marked curve}, we mean a (possibly nodal) curve, together with a subset of its dual graph.  We understand this subgraph as indicating the `positive area' components of the curve.  We impose the usual notion of stability on the unmarked components.  We refer to the unmarked graph edges which are adjacent to a marked vertex as  \emph{Ploutonion edges}, and the corresponding nodes as \emph{Ploutonion nodes}.

Recall that a family of nodal curves gives a stratification of the base by loci along which the dual graph is locally constant, such that the generization maps off strata give edge contractions.  That is, it defines a morphism from the exit path category of the stratification to the category whose objects are graphs and morphisms are edge contractions.

We say such a family is a family of marked curves if each curve is marked, and the edge contractions respect the marking, in that if $\pi\colon \Gamma \to \Gamma'$ is an edge contraction, and the markings are given by the subsets $G \subset \Gamma$ and $G' \subset \Gamma'$, then $G' = \pi(G)$.  This encodes the fact that if one joins various components, the resulting curve has positive symplectic area if any of the components do.  We write $\mathfrak{p} \subset \mathfrak{n}$ for the Ploutonian nodes.   Note that $\mathfrak{p}$ 
is not preserved by specialization maps. 

As in Section \ref{sec: cosheaf for families of curves}, consider a family of marked nodal curves over a smooth base $M$. We will be interested in the locus $M^{\chi_+}\subset M$ of curves whose marked part has Euler characteristic $\chi_+$.  This is a locally closed subset.

\begin{definition} \label{def: ploutonion stratification}
       Let $M^{\chi_+}_{2i+b} \subset M^{\chi_+}$ be the locus where the marked part has $i$ un-marked half-edges corresponding to interior nodes, and $b$ un-marked half-edges corresponding to boundary nodes. 
\end{definition}

\begin{example}\label{ex : two tori on a sphere}
    Begin with a curve $S$ of Euler characteristic $\chi=\chi_+$, attach a sphere to $S$, and attach two tori to the sphere.  Mark only $S$. Take $M$ to be the usual coordinate chart parameterizing smoothings of the resulting curve. Then $M$ is a neighborhood of the origin in $\C^3$ where the coordinates $(x,y,z)$ of a point in $\C^3$ corresponds to gluing parameters $x$ and $y$ at the nodes where the tori are attached and $z$ the gluing parameter at the node where the sphere is attached to $S$. We will refer to the nodes by the coordinate of the correspoding gluing parameter.  
    
    There is a unique way to extend the marking on the central fiber to a marking of this family: the component corresponding to $S$ is marked, any other component is unmarked, and all nodes are unmarked ($\mathfrak{n}_0 = \mathfrak{n}$).  Then $\mathfrak{p}$ is the $z$-node on the $xy$-plane, $\{z=0\}$. Off the $xy$-plane $\mathfrak{p}$ consists of the $x$-node along $\{x=0,y\ne 0,z\ne 0\}$, the $y$-node along $\{x\ne 0,y=0,z\ne 0\}$, and both the $x$- and the $y$-nodes along $\{x=y=0,z\ne 0\}$.  
    
    Here $M^{\chi_+}$ is the union of the $xy$-plane with the $z$-axis, 
    $$
    M^{\chi_+} \ = \ \{z=0\}\cup\{x=y=0\}.
    $$
    This decomposes further $M^{\chi_+}=M^{\chi_+}_2\cup M^{\chi_+}_4$, where 
    $$
    M^{\chi_+}_2 \ = \ \{x=y=0\},\qquad 
    M^{0}_4 \ = \ \{x=y=0,z\ne 0\}.
    $$ 
    Note in particular that the closure of $M^{0}_4$ is not a union of strata. Remaining strata of $M$ are 
    \begin{align*}
    M^{\chi_+-2}=M^{\chi_+-2}_2 \ &= \ \{x=0,y\ne 0,z\ne 0\}\cup\{x\ne 0,y=0,z\ne 0\},\\ 
    M^{\chi-4} = M^{\chi-4}_0 \ &= \ \{x\ne 0,y\ne 0,z\ne 0\}.
    \end{align*} 
\end{example}

\begin{example}\label{ex : degenerate annulus with torus}
    Consider a curve $S$ with boundary $\partial S$ and Euler characteristic $\chi=\chi_+$.  We consider a disk $D$ and attach $D$ to $S$ at two boundary marked points $x$ and $y$.  We attach a genus one surface with boundary at a third boundary point $z$ of $D$.  We mark only $S$.  Let $M$ be the gluing parameter space for deforming the resulting curve as in Example \ref{ex : two tori on a sphere}, here  it is an open subset of $\R^3_{\ge 0}$. Then $M^{\chi_+}$ is the union of  the coordinate axes with strata: 
    $$ 
    M^{\chi_+}_1 \ = \ \emptyset, \quad M^{\chi_+}_2 \ =\ \{x=y=0\},\quad M^{\chi_+}_3 \ = \ \{x\ne 0,y=z=0\}\cup \{x=z=0,y\ne 0\}.
    $$
\end{example}
We return to a general family of nodal curves parameterized by a smooth manifold $M$.
\begin{lemma} \label{ploutonion stratification}
    The strata of the stratification $M = \bigsqcup_n M_n^{\chi_+}$ are characterized as the maximal loci on which the specialization maps of $\mathfrak{n}$ induce bijections on the Ploutonion nodes $\mathfrak{p}$.  
\end{lemma}
\begin{proof}
    It is obvious that smoothing nodes which do not meet a marked component does not change the Euler characteristic of the marked locus nor the number or type of such nodes, hence will not move between different strata among the $M_m^{\chi_+}$.  Conversely, smoothing a node which connects marked components necessarily decreases the Euler characteristic $\chi=\chi_+$ of the marked components, and smoothing a node which connects a marked to an unmarked component decreases $\chi=\chi_+$ unless the unmarked component is a sphere or disk, in which case, by stability, it increases the number or unmarked half edges $m$. 
\end{proof}

\begin{corollary} \label{weak frontier}
    The following weak frontier property holds: 
    $$\overline{M^{\chi_+}_n} \setminus M^{\chi_+}_n \subset \bigsqcup_{m < n} M^{\chi_+}_m.$$ 
\end{corollary}

\begin{proof}
    Follows from the proof of Lemma \ref{ploutonion stratification}.
\end{proof}

Suppose that $M$ parametrizes a family of marked nodal curves and a compatible system of versality witnesses. Then we have the following: 
\begin{corollary} \label{smooth family of disks}  
The strata $M^{\chi_+}_n$ are smooth manifolds (with corners where $M$ had corners) and there exists smooth families of disk neighborhoods of the Ploutonion nodes on the marked components of the curves in $M^{\chi_+}_n$.\footnote{We are not asserting that one can globally choose parameterizations of these neighborhoods (indeed the Chern class of the line bundle that such a global parameterization would trivialize is the ubiquitous $\psi$ class of enumerative geometry).  
The parameterization is however fixed up to rotation, which is all that is relevant for the cutoffs etc. involved in gluing maps.}
\end{corollary}

\begin{proof}
    Clear from Lemma \ref{ploutonion stratification}.
\end{proof}

\begin{remark}\label{r:gluingparematertubes}
Corollary \ref{smooth family of disks} implies that
tubular neighborhoods of $M_n^{\chi_+}\subset M$ are parameterized by coordinates corresponding to the gluing parameters at Ploutonion nodes (defined by choosing a family of disk neighborhoods). The tubular neighborhood parameterizations are compatible on overlaps as in Section \ref{sec: cosheaf for families of curves}, \eqref{coordinate compatibility} for coordinate corner structures.  Note that we assert no compatibility of the choices of disk neighborhoods between the different strata.
\end{remark}
\begin{remark}
    The codimension of $M_n^{\chi_+}\subset M^{\chi_+}$ is determined by counting the unmarked edges meeting the marked locus, whereas the subscript is given by counting the corresponding half-edges.  Thus in general the codimension of $M_n^{\chi_+}\subset M^{\chi_+}$ can be less than $n$.  
\end{remark}

\subsection{Extending basic perturbations, I}\label{ssec : ext basic I} 
Fix a map with smooth domain $u\colon S \to X$ and positive symplectic area. After adding marked points in local hypersurfaces we make the domain stable without automorphisms, as explained in the proof of Theorem \ref{thm: polyfold structure on Z}.
Let $\bU(u,S)$ be a chart around $(u,S)$ as in \eqref{def: basic perturbation} and let $\boldsymbol{j}$ denote the open ball of complex structures of domains of maps in $\bU(u,S)$, in view of the marked points added we think of the domains $S\in\boldsymbol{j}$ as having a unique holomorphic parameterization. 

Consider a basic perturbation  $\widehat{\lambda}$ in $\bU(u,S)$ as given by \eqref{eq: uncut basic}, i.e., before cutting off in the function direction. Recall from Lemma \ref{lem: extendable} that $\widehat{\lambda}$ is extendable in the sense of Definition \ref{def: extendable}, i.e.\ if we take charts with additional marked points ($\bullet$) and punctures ($\circ$), and consider the forgetting/marked-point-to-puncture-diagram 
$$\bU(u, S) \xleftarrow{f} \bU(u, S)_{n;\bullet}^{\mathrm{reg}} \xrightarrow{i} \bU(u, S)_{n;\circ}^{\mathrm{reg}}$$
where, as indicated by the superscripts, all domain curves are smooth and the $n$ marked points or $n$ punctures do not collide, 
then there is a section $\widehat{\lambda}_{\circ}$  of $\bW_\circ \to \bU(u,S)^{\mathrm{reg}}_\circ$ over $\bU(u,S)^{\mathrm{reg}}_{n;\circ}$ such that
$i^{\#} f^* \widehat{\lambda} = i^* \widehat{\lambda}_\circ$. (The map $i^{\#}: \mathbf{W}_\bullet \to i^* \mathbf{W}_\circ$ was explained before Definition \ref{def: extendable}.)

We will now construct an \'etale `chart'  $\bN(\bU(u,S))$ (chart  in quotes because this space itself has points with orbifold structure) for $\bZ$ such that:  
\begin{itemize}
    \item the image of $\bN(\bU(u,S))$ contains the set of all stable maps $(v,T)$ whose positive area part  $(v_+,T_+)$ lies in the image of $\bU(u,S)$ and whose zero area part is close to constant.
    \item the element in $\bN(\bU(u,S))$ mapping to such a $(v, T)$ comes with a lift of the corresponding $(v_+, T_+)$ to $\bU(u, S)_{n;\circ}^{\mathrm{reg}}$.
\end{itemize}  

\begin{remark}
    The sole purpose of the chart $\bN(\bU(u,S))$ is that  allows us to define, on the locus  inside $\bN(\bU(u,S))$ with stratum-wise forgetful maps to $\bU(u, S)_{n;\circ}^{\mathrm{reg}}$,  a stratum-wise pullback (of an extension of) the (not multi) section $\widehat{\lambda}$  which we then extend using the partition of unity \eqref{cutoff formula}, before descending to a multisection on $\bZ$.  That is, this chart lets us avoid the problem that multisections do not interact well with partitions of unity.   
\end{remark}

Domains in $\boldsymbol{j}$ are smooth complex curves (with markings trivializing automorphisms). For fixed such domain $S$, consider the polyfold $\bP_S$ of stable maps into $S$. Let $\M_S\subset \bP_S$ be the sub-polyfold of $\bP_S$ that consists of holomorphic degree $1$ stable maps $v\colon T\to S$. Then elements $(v,T)\in \M_S$ are nodal curves $T$ obtained from the smooth complex curve $S\in\boldsymbol{j}$ by attaching a stable curve $S_0$ at points in $S$. The restriction $v|_{T_+}$ of the map to the positive area component gives a unique holomorphic diffeomorphism $T_+\to S$ and if $S_{0;\mathrm{c}}$ is a connected component of the zero area part $S_0$ then the restriction $v|_{S_{0;\mathrm{c}}}$ is a constant map to the point $\zeta$, where $\zeta\in S$ is the point where $S_{0;\mathrm{c}}$ is attached. 

Let $\bD$ denote the polyfold which is the fibration over $\boldsymbol{j}$ with fiber over $S\in\boldsymbol{j}$ equal to $\M_{S}$. We denote elements $(v,T)\in \bD$ simply $T$ since the map $v: T \to S$ is uniquely determined by $T$. Let $\boldsymbol{j}_{\bullet;n}$ denote the space of domains in $\boldsymbol{j}$ with $n$ distinct marked points. If $T\in\bD$ then forgetting area zero components of $T$, the positive area $T_+$ lies in $\boldsymbol{j}_n$ for some $n$. Letting $\bD_n$ denote the subset of $\bD$ with positive area part in $\boldsymbol{j}_{\bullet;n}$, the stratification of Definition \ref{def: ploutonion stratification} induces a stratification:
\begin{equation}\label{eq : stratification of bD}
\bD = \bigsqcup_n \bD_n,
\end{equation}
and forgetting ghosts gives 
\begin{equation}\label{eq: ploutoninan nodes to punctures of bD}
 \pi\colon \bD_n \to \boldsymbol{j}_{n;\bullet},
\end{equation}

Consider the polyfold $N(\bD)$ that consists of all domains in a neighborhood of domains in $\bD$. Here the polyfold structure is obtained in the standard way: if a domain $T\in N(\bD)$ is obtained by attaching a zero area component to some domain in $\bD_n$ and also to some domain in $\bD_m$ then as in \cite[Theorem 2.2, Definition 2.3]{HWZ-GW} there exists an $\mathrm{sc}$-smooth change of coordinates of a neighborhood of $T$ in the chart of $\bD_n$ to a neighborhood of $T$ in the chart of $\bD_m$.

So far we considered only domains. We next consider also maps. Let $\Omega(u,S)$ be the polyfold of maps $(v,T)$, $v\colon T\to X$, where $T\in\bD$ has regularity $H^{3}$, where the positive area part $(v_+,T_+)$ has domain the positive area part $T_+\in\boldsymbol{j}_{\bullet;n}$ of $T$ and where the zero area part $(v_0,T_0)$ has domain the zero area part $T_0$ of $T\in\bD$ and $v_0$ has $C^0$-norm bounded by some fixed $\epsilon>0$. This means that the   
positive area part $(v_+,T_+)\in \bU(u, S)_{n;\bullet}^{\mathrm{reg}}$ for some $n$. Then 
the stratification of Definition \ref{def: ploutonion stratification} induces a stratification:
\begin{equation}\label{eq : stratification of Omega}
\Omega(u, S) = \bigsqcup_n \Omega(u, S)_n,
\end{equation}
where $\Omega(u,S)_n$ is the subset of $\Omega(u,S)$ with positive area part of the domain in $\boldsymbol{j}_{\bullet;n}$ and where forgetting ghosts gives  
\begin{equation}\label{eq: ploutoninan nodes to punctures}
\pi\colon \Omega(u, S)_n \to \bU(u,S)^{\mathrm{reg}}_{n;\circ}.
\end{equation}

Fix disk neighborhoods of the Ploutonion nodes smoothly over each stratum $\boldsymbol{j}_{\bullet;n}$ underlying $\Omega(u,S)^{\chi(S)}_n$ as in Corollary \ref{smooth family of disks}. This choice of disk neighborhoods determines gluing parameters for maps in $\Omega(u,S)_n$ and gives polyfold charts around $(v,T)\in\Omega(u,S)_n$ of the 
\[
H^{3,\delta}(T, v^\ast TX)\times C \times G,
\]
where $C$ is the finite dimensional space of asymptotic constants and $G$ is the finite dimensional space of gluing parameters, at nodes, as in Section \ref{sec:stablemapmoduli}. If a map $(v,T)$ appears in two such charts then again as in Section \ref{sec:stablemapmoduli} (see \cite[Theorem 1.6, 1.7]{HWZ-GW}) there are $\mathrm{sc}$-smooth coordinate changes between the neighborhoods of $(v,T)$ in the two charts. We then define $\bN(\bU(u,S))$ as the set of stable maps that lies in charts as above centered at elements in $\bigcup_{n}\Omega(u,S)_n$ with the polyfold structure induced by the coordinate changes between neighborhoods in charts discussed.

As in Section \ref{sec:formaldiff} we have the bundle of formal differentials $\bW\to \bN(\bU(u,S))$.
We next extend the basic parturbation $\widehat{\lambda}$ which gives a section $\Omega(u,S)\to\bW$ to a section over $\bN(\Omega(u,S))$. The choice of disk neighborhoods over the strata in $\Omega(u,S)$ determines coordinates on the normal directions to $\Omega(u,S)_n$ in $\bN(\bU(u,S))$, and fixes the choices in the sharp gluing $\widehat{ \oplus}^R$, which we use to extend sections of $\bW$ defined over ${\Omega(u,S)_n}$ to sections of $\bW$ over a neighborhood of ${\Omega(u,S)_n}$ in $\bN(\Omega(u,S))$.  We define: 
    \begin{equation}\label{eq: def widehatlambda star}
    \widehat{\lambda}_{\star,k} \ = \ \widehat{\lambda} \ \widehat{\oplus}_a^R \ 0.  
    \end{equation}    
and, using the partition of unity \eqref{cutoff formula}, 
    \begin{equation}\label{eq: def widehatlambda[epsilon]}
        \widehat{\lambda}[\epsilon] := \sum \kappa_i \widehat{\lambda}_{\star, i}.
    \end{equation}
Then for sufficiently small $\epsilon<0$, $\widehat{\lambda}[\epsilon]$ is a section of $\bW$ defined over $\bN(\Omega(u,S))$.  

The following result is immediate from the definitions, we include it for future reference.
\begin{lemma}\label{l : small Euler char in support of extended}
Let $\widehat{\lambda}$ be a basic perturbation supported in a chart $\bU(u,S)$ with $\chi(S)=\chi_0$. If $(v,T)$ is a stable map of Euler characteristic $\chi\le \chi_0$ that lies in the support of $\widehat{\lambda}_{\star,k}$ and if $T_+$ is the positive area part of $(v,T)$ then also $\chi(T_+)\le \chi_0$.  
\end{lemma}

\begin{proof}
The support of $\widehat{\lambda}_{\star,k}$ is a subset which is obtained by attaching area zero components to $S$ at $k$ marked points.
\end{proof}

The section $\widehat{\lambda}[\epsilon]$ is a $\mathrm{sc}^+$-section of $\bW$ over $\bN(\Omega(u, S))$.
In Section \ref{ssec: extension basic II}, we will consider how to cut off this section in the functional-analytic directions, so as to obtain a corresponding $\lambda[\epsilon]$ which is moreover supported in a neighborhood of the subset of $\Omega(u,S)$ in $\bN(\Omega(u,S))$ that is obtained by attaching constant ghost components to maps in $\bU(u,S)$.

\subsection{Asymptotic constant charts} \label{ssec:modifiedpolyfoldcharts}
It will be convenient to express our cutoff after a certain change of variables in the chart, designed to preserve the `asymptotic constant' term over the whole chart, not just at the nodal moments.   
Consider the gluing region
$$Z_{a} := \{0\le s\le R\} = \{-R\ge s'\ge 0\}$$
For a function $h$ on $Z_a$, we will write 
$$
\langle h \rangle_0 = 
\begin{cases}
\int  h\left(\frac12R,t\right) dt, \quad\text{(interior)},\\
\int \pi_{\R^{n}}\left(h\left(\frac12R,t\right)\right) dt, \quad\text{(boundary)}.
\end{cases}
$$

Note these integrals are well defined continuous functionals on $H^3_{\mathrm{loc}}$.  
We write $H^3(Z_a)_0 \subset H^3(Z_a)$ for the locus of $h$ with $\langle h \rangle_0 = 0$.

Let $\gamma_a \colon Z_a \to \R$ equal to $1$ in $[1,R-1]$ and equal to $0$ outside $[\frac12,R-\frac12]$.
Obviously the following is an (sc) isomorphism: 
\begin{eqnarray*}
	H^{3,\delta}(Z_{a}, \R^{2n})_{0}\oplus\R^{2n} & \to & H^{3}(Z_{a}, \R^{2n}) \\
	(h, v) & \mapsto & w_{a,\delta}(s,t)\cdot h + \gamma_a v.
\end{eqnarray*}
with inverse 
\begin{equation}\label{eq : proj asympt const}
h \ \mapsto \ (w_{a,\delta}^{-1}(h - \gamma_a \langle h \rangle_0), \langle h \rangle_0).    
\end{equation}

By such isomorphisms, we use modified building blocks for standard polyfold charts around nodes, e.g. in the case of interior nodes 
we replace the standard
\[ 
G^{a}=H^{3}(Z_{a},\R^{2n}) \ \oplus \ H^{3,\delta_{0}}\left(C_{a},\R^{2n}\right) 
\]
with 
\begin{equation}\label{eq : charts with asymptotic constants}
G^{a}=H^{3,\delta}(Z_{a},\R^{2n})_{0}\oplus\R^{2n} \ \oplus \ H^{3,\delta_{0}}\left(C_{a},\R^{2n}\right),
\end{equation}
where the first factor in the right hand side is the Sobolev space with the restricted weight norm, see Section \ref{ssec:sharpgluing}, and similarly for elliptic and hyperbolic boundary nodes, where instead the asymptotic constant lies in $\R^{n}\subset\R^{2n}$. We call charts as above (with the asymptotic constants singled out and with the restricted weight norm in gluing regions) \emph{asymptotic constant charts}.

\subsection{Extending basic perturbations, II: cutoffs}\label{ssec: extension basic II}
Recall that the basic perturbation $\lambda$ is obtained from $\widehat{\lambda}$ by cutting off in the map direction in two steps:
\[
\lambda_0(v) \ = \ \beta_0(\|v\|)\widehat{\lambda}(v),\quad \lambda(v) \ = \ \beta_1(\|\bar\partial_J\widetilde{v}\|_2)\lambda_0(v),
\]
where $\beta_k$, $k=0,1$ are functions equal to $1$ in some neighborhood of $0$ and equal to $0$ outside a larger neighborhood, see \eqref{eq: chi0 basic perturbation in}. 

We point out the different roles of the cut off factors $\beta_0(\|v\|)$ and $\beta_1(\|\bar\partial_J\widetilde{v}\|_2)$. If $\|\bar\partial_J\widetilde{v}\|_2$ increases, we move away from the original moduli space of solutions (i.e., away from the locus $\bar\partial_J=0$) and if the auxiliary norm of $\widehat\lambda$ is sufficiently small compared to the support of $\beta_1$ then  all solutions to $\bar\partial_J=\widehat\lambda$ lie in the region where $\beta_1(\|\bar\partial_J\widetilde{v}\|_2)=1$. This means that the exact size of the support of $\beta_1$ will not affect solutions as long as it is sufficiently large compared to $\widehat{\lambda}$. In contrast to this, the $L^2$-norm $\|v\|$ may well increase as we move along the original moduli space (solutions to $\bar\partial_J=0$) and it will be important that our extensions of  $\lambda_0=\beta_0(\|v\|)\widehat{\lambda}$  to neighborhoods of maps with constant ghosts attached changes compatibly. This compatibility is the main subject of this section. We will use the asymptotic constant polyfold charts, see Section \ref{ssec:modifiedpolyfoldcharts}.

We consider a basic perturbation $\lambda$ supported in $\bU(u,S)$ and the extension of its underlying not cut-off perturbation $\widehat{\lambda}[\epsilon]$ to $\bN(\Omega(u,S))$ as in Section \ref{ssec : ext basic I}. 

Let $(w,T)\in \Omega(u,S)$ be a map with constant ghosts with underlying positive part map $(w_+,T_+)\in \bU(u,S)$. Then the domain $T_+$ is the curve $S$ equipped with a complex structure $j\in\boldsymbol{j}$ and $w_+=\exp_u(v_+)$, $v_+\in H^{3}(S,u^\ast TX)$, see \eqref{eq: nbhd basic perturbation}. 
Recall the stratification $\Omega(u,S)=\bigsqcup_n\Omega(u,S)_n$, see \eqref{eq : stratification of Omega}.

Let $\nu_r$ denote the gluing parameter neighborhood of $\Omega(u,S)_n$, see Remark \ref{r:gluingparematertubes}. Denote the $n$ punctures $\boldsymbol{\zeta}=(\zeta_1,\dots,\zeta_n)$ and $S_{\boldsymbol{\zeta}}$ the corresponding punctured domain. Let $T_0$ be the ghost part of $T$. If $(w',T')\in \nu_r$, we let $a_k(T')$, $k=1,\dots,r$ denote the gluing parameters at the nodes determined by $T'$.

Recall  $w_+=\exp_u(v_+)$ for some vector field along $u$. Shifting the origin of the exponential map along $v_+$, maps with domain $T'$ are parameterized by the image under $\oplus$ of 
\begin{equation}\label{eq : local coord Omega(u,S)}
\boldsymbol{v}=(v,v_{k},c_{l})\in H^{3,\delta}(S_{\boldsymbol{\zeta}},u^{\ast} TX)\times H^{3,\delta}(T_0,\R^{2n})\times C,
\end{equation}
where $C$ denote the space of asymptotic constants. 
Here $H^{3,\delta}(T_0,\R^{2n})$ denotes $\R^{2n}$-valued maps, using some identification of the neighborhoods of the images the $\zeta_i$. 
Below we will use norms of such maps where the domains vary over all possible ghosts. We define such norms by covering the compact space of domains by finitely many standard charts used to define norms in $H^{3,\delta}(T_0,\R^{2n})$ and then add them weighted by a partition of unity. To simplify notation, if $v_k$ is a vector field along $T_0$, we will simply write $\|v_k\|_{m,\delta}$ for its norm (even though it may be a linear combination of norms in a finite number of charts).  
We write $G(\boldsymbol{v})$ for the map corresponding to $\boldsymbol{v}$. 

We will work with sections of $\bW$ over $\nu_r$. Using coordinate as in \eqref{eq : local coord Omega(u,S)}, such sections are represented by the image under $\oplus$ of
\[
\boldsymbol{\xi}=(\xi,\xi_k)\in H^{2,\delta}(S_{\boldsymbol{\zeta}},\mathrm{Hom}^{0,1}(TS,v^\ast TX))
\times H^{2,\delta}(T_0,\mathrm{Hom}^{0,1}(TT_0,v_k^\ast T\R^{2n})).
\]
Our extended basic sections are of level $(0,1)$ and just like for basic sections we use the $(3,\delta)$-fiber norms in local coordinates as auxiliary norms, see Section \ref{sec:formaldiff}. Over the ghosts that means just like for the chart itself that the norm is a linear combination of norms in a finite number of charts. Again, as for basic perturbations the auxiliary norms are defined locally but have all the properties needed for auxiliary norms in a neighborhood of the spaces of solutions.   

We now turn to cut-offs. As with the original basic perturbation, we will cut off the extensions $\widehat{\lambda}_{\star,n}$ defined in $\nu_n$ by a product of cut-off functions, one standard that depend on the size of the 2-norm of $\bar\partial_J$ and another that depends on an $L^2$-norm. In order for the cut-off to be compatible with the original basic perturbation we use an $L^2$-norm which is a sum of the standard weighted $L^2$-norm on the ghost part and the original unweighted norm on the parts of the domain corresponding to the original bare curve. For this latter part we must then consider the effect on the norm of replacing a marked point by a puncture and show two things: smoothness of the norm function and the fact that the total cut off is actually supported inside $\nu_n$.

We first introduce the $L^2$-norm along the bare part of the surface.
\begin{definition}
    Let $S_a$, $a=(a_1,\dots,a_r)$ denote the domain $S$ with disks and half disks of radius determined by the gluing parameter $a_k$ around $\zeta_k$ removed and let $S_{+,a}=S_a\cap S_+$, where $S_+$ is the positive area part of $S$. Then $S_{+,a}$ is a compact subset of $S$. Define
\begin{equation}\label{eq:N0+}
	N_{0+}(\boldsymbol{v}) := \left\|(v+\textstyle{\sum_{l}}c_{l})|_{S_{+,a}}\right\|_{+}^{2},
\end{equation}
where $v+\sum_{l}c_{l}$ denotes the sum of the vector field $v\in H^{3,\delta}(S_{\boldsymbol{\zeta}};u^{\ast}TX)$ and cut-off asymptotic constants at punctures, see Section \ref{ssec:modifiedpolyfoldcharts}, and where the norm $\|\cdot \|_+$ refers to the usual $L^{2}$-norm on $S$. 
\end{definition}

We check that the function $N_{0+}$ is smooth as a function on the asymptotic constant polyfold chart. Consider $\boldsymbol{v}=(v,v_{k},c_{l})$ as above and a map from a domain $T'$. Here $N_{0+}(\boldsymbol{v})$ depends only on $v\in H^{3,\delta}(S_{\boldsymbol{\zeta}},u^{\ast}TX)$ and $c\in C$. More precisely, $N_{0+}(\boldsymbol{v})$ is obtained by restricting $v+c$ to the part of $S_{\boldsymbol{\zeta}}$ that lies in $T'$ and then taking its $L^{2}$-norm. Thus to show that $N_{0+}$ is smooth we must compare the weighted $L^{2}$-norm on the punctured domain with cylindrical ends with the usual `compact domain' $L^{2}$-norm.   

\begin{lemma}\label{l:L2normsmoothonghostspace}
	The function $N_{0+}$ is smooth.   
\end{lemma}
	
	\begin{proof}
		Consider a disk or half disk neighborhood $D$ of nodal point $\zeta_k\in S$. The corresponding neighborhood of $\zeta_k$ considered as a puncture is $[0,\infty)\times I$ in the notation of Lemma \ref{l:markedvspuncturenorms}. Furthermore, in this notation the vector field $v+c$, where $c$ is the asymptotic constant at $\zeta$, is $v_{\infty}\in H^{3,\delta}([0,\infty)\times I,\R^{2n})$ and the contribution to the function $N_{0+}$ from $v$ in $D$ is
		\[
		\int_{D}|v_{0}|^{2} \,dxdy.
		\]
		Lemma \ref{l:markedvspuncturenorms} \eqref{eq:ref pull-back L^2} says that
		\[
		\int_{D}|v_{0}|^{2} \,dxdy= \int_{[0,\infty)\times I}|v_{\infty}|^{2}e^{-4\pi s} \,dxdy.
		\] 
		The lemma follows.
	\end{proof}

Second, we introduce the standard asymptotic constant $L^2$-norm of the ghost part: 
\begin{definition}
    Let $T$ be the domain obtained by gluing $S$ and $T_0$ according to the gluing parameters $a=(a_1,\dots,a_r)$ and let $T_{0,a}$ denote the part of $T$ that is a subset of $T_0$.  Define
\begin{equation}\label{eq:N00}
	N_{00}(\boldsymbol{v})=\sum_{k}\|v_{k}|_{T_{0,a}}\|_{\delta}^{2}+\sum_{l}|c_{l}|^{2},
\end{equation}
where $\|\cdot\|_{\delta}$ denotes the restricted weight $L^{2}$-norm. 
\end{definition}

Third, we use the restricted weight $(2,\delta)$-norm of the complex anti-linear derivaitve of $G(\boldsymbol{v})$:

\begin{definition}
With $T$ and $\boldsymbol{v}=(v,c_k,c_l)$ as above define
\begin{equation}\label{eq:N1}
	N_{\bar 1}(\boldsymbol{v}) = \|\bar\partial_{J} G(\boldsymbol{v})\|_{2,\delta}^{2}, 
\end{equation}
where the norm is the restricted weight norm on $T$. 
\end{definition}

The functions $N_{00}$ and $N_{\bar 1}$ are smooth by inspection, they are (parts of) function norms used in the definition of the polyfold charts in Section \ref{ssec:modifiedpolyfoldcharts}.

We next show that the three functions $N_{0+}$, $N_{00}$, and $N_{\bar 1}$ together give rough control on the overall norm in the asymptotic constant chart.
Let $\boldsymbol{v}=(v,v_{k},c_{l})$ be as above in a coordinate chart parameterizing a neighborhood of $\Omega^{\chi(S)}(u,S)_n$ and write $\tilde v=\tilde v(\boldsymbol{v})$ for the gluing of $v$ and $v_{k}$ shifted by the asymptotic constants $c_{l}$. We write $\|\tilde v(\boldsymbol{v})\|_{3,\delta}^{\wedge}$ for the natural norm in the chart, near each puncture we use $H^{3,\delta}(Z_a;\R^{2n})_{0}\oplus \R^{2n}$, see \eqref{eq : charts with asymptotic constants}, so locally in a gluing neck region:
\[
\|\tilde v(\boldsymbol{v})|_{Z_a}\|_{3,\delta}^{\wedge}=\|\tilde v(\boldsymbol{v})_0\|_{3,\delta} + |c(\boldsymbol{v})|,
\]
where $\tilde v(\boldsymbol{v})_0$ is the projection to the space of maps with vanishing asymptotic constant and $c(\boldsymbol{v})$ is the asymptotic constant of $\tilde v(\boldsymbol{v})$, see \eqref{eq : proj asympt const}.

As in Lemma \ref{l: function cut off basic}, all our perturbations will be supported near the original moduli space so we assume that the weighted $L^2$-norm $\|\bar\partial_J u\|_{\delta}<\epsilon$.

\begin{lemma}\label{l: Ns control 3norm}
Assume $\|\bar\partial_J u\|_{\delta}<\epsilon$. Then there exists $\epsilon_{0}>0$ and $C>0$ (independent of $\epsilon>0$) such that the following holds.

If $\boldsymbol{v}$ satisfies $\|v\|_{3,\delta}+\sum_{k}\|v_{k}\|_{3,\delta}+\sum_{l}|c_{l}|<\epsilon_{0}$ and if $r>C\epsilon$ then there exist $\delta>0$ such that
\[
N_{0+}(\boldsymbol{v})+ N_{00}(\boldsymbol{v})+N_{\bar 1}(\boldsymbol{v})<\delta
\]
implies 
\[
\|\tilde v(\boldsymbol{v})\|_{3,\delta}^{\wedge}< r.
\]
\end{lemma}

\begin{proof}
We pick a finite cover of $S$ by local coordinate disks $D$ of two types, those that contain punctures $\zeta_k$ and those that do not, such that $G(\boldsymbol{v})$ maps each $D$ into some coordinate ball in $X$. (Here we use the $\epsilon_0$-bound.) 

For domains without punctures we repeat the argument from Lemma \ref{l: function cut off basic} and conclude:
\[
\|\tilde v\|_3 \le  C'(\|\bar\partial_JG(\boldsymbol{v})\|_2+\|\bar\partial_J u\|_2 + \|\tilde v\|).
\]
Along the ghost components the argument is identical using instead Sobolev norms and asymptotic constants of the asymtotic constant charts. 

For disks that contain punctures we need an estimate using the norm $N_{0+}$, which is not the standard norm for elliptic estimates. Consider a cylindrical end $[0,\infty)\times I$ where the vector field is $v$ and the asymptotic constant $c$, represented as $\gamma\cdot c$, for some cut-off function with derivative supported in $[1,2]\times I$. Here we have from the weighted elliptic estimate on $\R\times I$ using a cut-off function supported in $[0,1]\times I$,
\[
\|v\|_{3,\delta}\le C(\|\bar\partial_Jv\|_{2,\delta}+\|v|_{[0,1]\times I}\|).
\]
The contribution to $N_{0+}(\boldsymbol{v})$ is
\[
\int_D |v_0 + \gamma c|^2 d\sigma = \int_A |v_0|^2 d\sigma + \int_{D\setminus A}|v_0+\gamma c|^2 d\sigma,   
\]
where $A$ corresponds to $[0,1]\times I$.
Here the first term in the right hand side controls $\|v|_{[0,1]\times I}\|$ and we get an estimate on $\|v\|_{3,\delta}$. Then
\begin{align*}
\area(D\setminus A)|c|^2 &\le \int_{D\setminus A}|v_0+\gamma c|^2 d\sigma +\int_D|v_0|^2 d\sigma\\
&=\int_{D\setminus A}|v_0+\gamma c|^2 d\sigma +\int_{[0,\infty)\times I}|v|^2e^{-4\pi s} dsdt\\
&\le\int_{D\setminus A}|v_0+\gamma c|^2 d\sigma +\|v\|_{\delta}^2
\end{align*}
and hence also the size of the asymptotic constant $|c|$ is controlled and we get the desired estimate on $\|\tilde v\|_{3,\delta}^\wedge=\|v\|_{3,\delta}+|c|$. The lemma follows.
\end{proof}

We then define the extension of the basic perturbation $\widehat{\lambda}$ to $\nu_n$ as the previously defined extension $\widehat{\lambda}_{\star,n}$ cut off: 
\[
\lambda_{\star,n}=\beta_0(\|N_{0+}(\boldsymbol{v})+N_{00}(\boldsymbol{v})\|)\beta_{n;1}(N_{\bar 1}(\boldsymbol{v}))\widehat{\lambda}_{\star,k},
\]
where $\beta_{n,1}$ is a cut-off function with sufficiently large support compared to the size of $\widehat{\lambda}_{\star,n}$ as discussed above (i.e., in order that all solutions to $\bar\partial_J=\widehat\lambda$ lies in the region where $\beta_{n,1}=1$) and where $\beta_0$ is the cut-off function used in the definition of the original basic perturbation, see \eqref{eq: L2 cut off lambda0}. 

\begin{lemma}
The section $\lambda_{\star,n}$ is supported inside the neighborhood $\nu_n$.
\end{lemma}

\begin{proof}
    Follows from Lemmas \ref{l: kappa part of unity} and \ref{l: Ns control 3norm}. 
\end{proof}

Consider $(w,T)\in\Omega(u,S)_n$ with constant ghosts, i.e., the restriction of $w$ to the zero area part of $T$ is constant. Consider the restriction $w|_{T_+}$ to the positive area part of $T$. Write $w^+_\bullet$ and $w^+_\circ$ for this map considered as a map from $T_+$ with marked points and with punctures, respectively.

\begin{lemma}
If $\beta_{1,n}(\|\bar\partial_J w^+_\circ\|_{2,\delta})=\beta_1(\|\bar\partial_J w^+_\bullet\|_2)=1$ then
\[
\lambda_{\star,n}(w_\circ,T)|_{T_+}=\lambda(w_\circ^+,T^+),
\]
where the right hand side is the extension of $\lambda=\lambda_0$ as in Lemma \ref{lem: extendable}. 
\end{lemma}
\begin{proof}
Before cutting off the sections agree by definition.
Let $w^+=\exp_u(v)$ then
since $\beta_{n,1}(N_{\bar 1}(v))=1$ and $\beta_1(\|\bar\partial_J\widetilde v\|_{2})=1$ we have 
\begin{align*}
\beta_0(\|N_{0+}(\boldsymbol{v})+N_{00}(\boldsymbol{v})\|)\beta_{k,1}(N_{\bar 1}(\boldsymbol{v})) &=
\beta_0(\|N_{0+}(\boldsymbol{v})+N_{00}(\boldsymbol{v})\|)\\
&=\beta_0(\|v\|+0)=\beta_0(\|v\|)\beta_1(\|\bar\partial_J\widetilde v\|_{2}).
\end{align*}
Thus, also the cut-off functions agree. The lemma follows.
\end{proof}

We next define the full extension of $\lambda$ as an interpolation of the extensions along the gluing parameter neighborhoods $\nu_n$. For small enough $\epsilon = (\epsilon_1, \epsilon_2, \ldots)$, we define 
    \begin{equation}\label{eq: def total extension basic pert}
        \lambda[\epsilon] := \sum \kappa_i \lambda_{\star, i}
    \end{equation}
where the $\kappa_i$ are defined as in \eqref{cutoff formula}. 

Let us note the chart independence of the $\lambda[\epsilon]$.
Indeed, suppose now given some $(w, T)$ and $(w', T')$, both of whose normalization is $(w_+, T_+)$ plus some constant maps.  We claim that if the standard neighborhoods of $(w, T)$ and $(w', T')$ overlap, then on this overlap, the respective $\lambda[\epsilon]$ defined using $(w, T)$ or using $(w', T')$ agree.  This is because the extension and cut off used depends {\em only} on the gluing parameters of the Ploutonion nodes, and the fixed disk neighborhoods on $S$ which were fixed at the outset.   (It is possible that the range of possible $\epsilon$ for which the construction makes sense depend on the chart, but in any fixed Euler characteristic, the space of such possible $(w, T)$ is compact, and we will later fix a bound on the Euler characteristic.)  Thus, for sufficiently small $\epsilon$, there is a well defined section, henceforth denoted $\lambda[\epsilon]$, supported in a neighborhood of the locus of maps whose underlying bare part is in $\bU(u, S)$.  

Finally, we establish an approximation property that is central to our main theorem: we show that by choosing $\epsilon$ sufficiently small, the full extension $\lambda[\epsilon]$ is arbitrarily close to $\lambda_{\star,n}$ in auxiliary norm.

\begin{lemma} \label{small hole estimate}
  For any $\eta > 0$, for all sufficiently small $\epsilon>0$ we have $\|\lambda[\epsilon] - \lambda_{\star, i}\|_{3,\delta} < \eta$ on $\nu_n$. 
\end{lemma} 
\begin{proof}
We use the notation $M^{\circ}_n$ from \eqref{away from collar} in the case $M=\Omega(u,S)$.
The restriction of the extension $\lambda_\circ[\epsilon]$ to $\Omega(u,S)^{\circ}_n$ agrees by construction with 
$\lambda_{\star, n}$.  
The complement $\Omega(u,S)_n \setminus \Omega(u,S)^{\circ}_n$ is a finite union
of regions in moduli where two or more of the points where ghosts meet $S$ collide and move onto a rational bubble.  In such regions, we see from 
Lemma \ref{controlled incompatibility} that the difference $\lambda[\epsilon] - \lambda_{\star, n}$ is controlled by the gluing parameters at Ploutonian nodes. The result follows because $\kappa_n$ is defined by cutting off in precisely these parameters.
\end{proof}

\section{Main results}\label{sec: main res}
In this section we prove our main results.

\subsection{Perturbations for bare curves}  Here we prove 
Theorem \ref{main theorem intro}.  We give first the 0-parametric part as Theorem \ref{main theorem 0-parameter}, and then the 1-parametric part as Theorem \ref{main theorem 1-parameter}.

Fix a symplectic Calabi-Yau threefold $(X, \omega, J)$, and a Maslov zero Lagrangian $L \subset X$.  We consider either $(X, L)$ compact or with appropriate assumptions to ensure Gromov compactness of holomorphic maps.
We study the  configuration space $\mathbf{Z}$ of maps of some fixed positive symplectic area $d$, possibly from disconnected domains.  Note that such maps have Euler characteristic bounded above by some $\chi_{\max} = \chi_{\max}(d)$.   We write $\mathbf{Z}_\chi$ for the (open and closed) locus of maps with domain of Euler characteristic $\chi$.  We write $\mathbf{Z}_{> \chi} = \bigsqcup_{\chi' > \chi} \mathbf{Z}_{\chi'}$, etc.
\begin{theorem} \label{main theorem 0-parameter}
	Fix a compatible almost complex structure $J$ on $X$ and an integer $\chi_{\min}$. There are finitely many basic perturbations $\lambda_\alpha$, 
    and sequences of real numbers $\epsilon^\alpha = (\epsilon^\alpha_1, \epsilon^\alpha_2, \ldots)$, with corresponding extensions $\lambda_\alpha[\epsilon^\alpha]$ as defined by  formula \eqref{eq: def total extension basic pert},
    such that for each $\chi_0 \ge \chi_{\min}$, the section $\theta_{> \chi_0}$ of $\bW_{>\chi_{\min}} \to \bZ_{>\chi_{\min}}$,   
    \begin{equation}\label{eq: def pert chi_min}
    \theta_{> \chi_0} := \sum_{\chi(S^\alpha) > \chi_{0}} \lambda_\alpha[ \epsilon^\alpha],
    \end{equation}
     has the following properties:
\begin{enumerate}
\item \label{main thm 0-parameter bare compactness}
Suppose given a sequence of bare solutions to $\bar\partial_J=\theta_{>\chi_0}$ converging to a not-bare $(u, S)$, i.e. the positive area part $(u_+, S_+)$ has $\chi(S) <  \chi(S_+)$, then $\chi(S_+) \le \chi_0$.  
\item \label{main thm 0-parameter genericity}
In $\mathbf{Z}_{> \chi_0}$, the bare solutions to 
$\bar \partial_J = \theta_{> \chi_0}$  are transversely cut out and 0-generic. 
\end{enumerate} 
Note that \eqref{main thm 0-parameter bare compactness} implies that
the locus in $\mathbf{Z}_{\ge \chi_0}$ of bare solutions to 
$\bar \partial_J = \theta_{> \chi_0}$  is compact. 

Note also that, by \eqref{eq: def pert chi_min}, there is a neighborhood $V$ of $L$ such that for any  map $(u, S) \in \mathbf{Z}_{>\chi_{\min}}$, the support of $\theta_{> \chi_0}((u,S))$ is disjoint from $u^{-1}(V)$.
\end{theorem} 
\begin{proof}
As $\mathbf{Z}_{> \chi_{\max}}$ is empty, \eqref{main thm 0-parameter bare compactness} and \eqref{main thm 0-parameter genericity} hold vacuously for $\theta_{> \chi_{\max}}$ for any collection of $\lambda^{\alpha}, \epsilon^\alpha$. 

We now induct on $\chi_0$. Suppose given $\lambda_\alpha$ center around maps $(u_\alpha, S_\alpha)$ with $\chi(S_\alpha)>\chi_0$ and $\epsilon^\alpha$ such that the assertion holds for $\theta_{> \chi_0}$.  We will choose new basic perturbations supported in $\bU(u_\beta,S_\beta)$ with $\chi(S_\beta)=\chi_0$, and then show that the assertion holds for the resulting $\theta_{> \chi_0-1}$. 

The hypothesis \eqref{main thm 0-parameter bare compactness} together with Lemma \ref{bare transversality} and its refinement Lemma \ref{l: generic solutions 0}, imply the existence of additional basic perturbations $\lambda_\beta$ centered around maps whose domains have Euler characteristics $\chi_0$, so that solutions in $\mathbf{Z}_{\chi_0}$ to $\bar \partial_J = \theta_{> \chi_0}|_{\mathbf{Z}_{\chi_0}} + \sum \lambda_\beta$ are transversely cut out and $0$-generic.  

Then for any $\epsilon^\beta$, 
\begin{equation} \label{extending down one step epsilons}
    \theta_{> \chi_0 - 1} := \theta_{> \chi_0} + \sum \lambda_\beta[\epsilon^\beta]
\end{equation}
will satisfy \eqref{main thm 0-parameter genericity}.  
As in Lemma \ref{bare transversality}, the $\theta$ are multi-sections, while for the $\lambda$ and $\lambda[\epsilon^\beta]$ the meaning of \eqref{eq: perturb by basic perturbations} is that we first symmetrize the sections $\lambda_\alpha$ on the \'etale charts where they are defined to multi-sections which descend to $\mathbf{Z}$, and then add the resulting multi-sections to each other and to $\theta(u)$ by the usual sum of multi-sections prescription as recalled in Section \ref{sec: adding multi}.

Indeed, $\theta_{> \chi_0 - 1}$ and $\theta_{> \chi_0}$ agree on $\mathbf{Z}_{> \chi_0}$, and $\theta_{> \chi_0 - 1}|_{\mathbf{Z}_{\chi_0}} = \theta_{> \chi_0}|_{\mathbf{Z}_{\chi_0}} + \sum \lambda_\beta$, by construction (where again we first symmetrize and then use the standard addition of multi-sections). 

We turn attention to \eqref{main thm 0-parameter bare compactness}.
For any $\epsilon^\beta$, the section $\theta_{>\chi_0 - 1}$ is controlled near ghosts by Proposition \ref{sharp hat gluing is controlled}; we will seek to apply ghost bubble censorship.  Suppose then given some sequence of bare solutions $(u_i, S_i)$ to $\bar \partial_J - \theta_{\chi_0 -1}$, which converges to a non-bare $(u, S)$, with positive part $(u_+, S_+)$ where $\chi(S_+) > \chi_0 - 1$. Note that if in fact $\chi(S_+) > \chi_0$, then the $(u_i, S_i)$ are eventually outside the support of the new $\lambda_\beta[\epsilon^\beta]$, hence solve the equation $\bar \partial_J - \theta_{>\chi_0}$, and we obtain a sequence which contradicts our inductive hypothesis \eqref{main thm 0-parameter bare compactness}. Thus we may assume $\chi(S_+) = \chi_0$. 

Suppose (to eventually derive a contradiction) that for all choices of the $\epsilon^\beta$, there are  sequences of solutions as in \eqref{main thm 0-parameter bare compactness} for which the assertion of that condition fails. Then consider a sequence $\epsilon^{\beta;j} \to 0$ and corresponding sequences
of solutions $(u_i^j, S_i^j)$ of solutions to $\bar \partial_J - \theta_{> \chi_0} - \sum \lambda_\beta[\epsilon^{\beta; j}]$ which converge, $(u_i^j, S_i^j) \to (u^j, S^j)$ as $i\to\infty$.  The positive parts $(u_+^j, S_+^j)$ of the limits $(u^j,S^j)$ all have Euler characteristic $\chi(S^j_+)=\chi_0$ by the previous paragraph.  

Lemma \ref{small hole estimate} implies that $(u_+^j, S_+^j)$ solves $\bar \partial_J - \theta_{> \chi_0} - \sum \lambda_\beta=0$ to within $\mathcal{O}(\epsilon^{\beta; j})$. It follows that a subsequence of $(u_+^j, S_+^j)$ converges to a solution $(v, T)$ of $\bar \partial_J - \theta_{> \chi_0} - \sum \lambda_\beta$ as $j\to\infty$.  

We first show that $(v, T)$ must be a bare solution. Suppose $(v,T)$ is not bare and denote its positive part $(v_+, T_+)$, which satisfies $\chi(T_+) > \chi_0$. Then we may take a `diagonal' subsequence 
$(u_{i}^{j(i)}, S_{i}^{j(i)})$ of bare solutions to $\bar\partial_J-\theta_{>\chi_0}- \sum \lambda_\beta[\epsilon^{\beta; j(i)}]=0$ that converges to some $(v', T')$  
with bare part $(v_+, T_+)$. However, we then again have that for sufficiently large $i$, $(u_{i}^{j(i)}, S_{i}^{j(i)})$ are outside the support of all the $\lambda_\beta[\epsilon^{\beta;j(i)}]$ and hence solve the previous equation $\bar \partial_J - \theta_{> \chi_0}=0$. Then $\chi(T_+) > \chi_0$ contradicts the inductive hypothesis \eqref{main thm 0-parameter bare compactness}. 

We conclude that $(v, T)$ is a bare solution to $\bar \partial_J - \theta_{> \chi_0} - \sum \lambda_\beta$.  Thus by \eqref{main thm 0-parameter genericity} for $\theta_{>\chi_0 - 1}$, which we have already established, $(v, T)$ is transversely cut out and 0-generic.  We now seek to show that the same holds for $(u_+^j, S_+^j)$ for sufficiently large $j$. 

We claim that in fact $\lambda_\beta[\epsilon^{\beta;j}]|_{(u_+^j,S_+^j)}$ can be extended to a basic perturbation $\lambda^j_\beta$ in $\bU(u^\beta,S^\beta)$ that differs from $\lambda_\beta$ by $\mathcal{O}(\epsilon^{\beta;j})$.  Indeed, 
by construction \eqref{eq: def total extension basic pert}, $\lambda_\beta[\epsilon^{\beta;j}]|_{(u_+^j,S_+^j)}$ differs from $\lambda_\beta|_{(u_+^j,S_+^j)}$ only by the form of the domain differential ($\alpha$ in \eqref{eq: uncut basic}). Let $\alpha$ denote the domain differential of $\lambda_\beta$ in the chart $\bU(u^\beta,S^\beta)$ where it is supported. It is straightforward to check that the domain differential of $\lambda_\beta[\epsilon^{\beta;j}]|_{(u_+^j,S_+^j)}$ extends to a domain differential $\alpha^j$ over $\bU(u^\beta,S^\beta)$ such that the $C^2$-distance between $\alpha$ and $\alpha^j$ is $\mathcal{O}(\epsilon^{\beta;j})$. Define the extension $\lambda^j_\beta$ of $\lambda_\beta[\epsilon^{\beta;j}]|_{(u_+^j,S_+^j)}$ by replacing $\alpha$ in the definition of $\lambda_\beta$ by $\alpha^j$. Let $\eta_\beta^j$ denote the basic perturbation which is the difference of $\lambda_\beta$ and $\lambda_\beta^j$, $\eta_\beta^j=\lambda_\beta-\lambda_\beta^j$. 

Now, for sufficiently small $\epsilon^{\beta; j}$, we may apply Lemma \ref{estimate openness of genericity} to $\eta^j=\sum_\beta\eta_{\beta}^j$, where $\eta_\beta^j$ are the basic perturbations above, to deduce that $(u_+^j, S_+^j)$ is $0$-generic.  But this contradicts ghost bubble censorship (Proposition \ref{ghost bubble censorship linear combo}).   
\end{proof}

The 1-parametric part of Theorem \ref{main theorem intro} is a consequence of the following result.  
\begin{theorem}\label{main theorem 1-parameter} 
Fix $\chi_{\min}$ and let $J^0$ and $\theta^0_{>\chi}$ and $J^1$ and $\theta^1_{>\chi}$ for $\chi\ge\chi_{\min}$ be almost complex structures and perturbations as in Theorem \ref{main theorem 0-parameter}. Then there is a $1$-parameter family of almost complex structures $J^t$ and finitely many 1-parameter families of basic perturbations $\lambda_\alpha^t$ and sequences of real numbers $(\epsilon^\alpha_1,\epsilon^\alpha_2, \ldots)$ such that for each $\chi_0\ge \chi_{\min}$, the 1-parameter family of complex structure $J^t$ and sections $\theta^t_{>\chi_0}$ of $\bW_{>\chi_{\min}} \to \bZ_{>\chi_{\min}}\times [0,1]$, 
\begin{equation}\label{eq: def pert chi_min 1parameter}
     \theta_{> \chi_0}^t := \sum_{\chi(S^\alpha) > \chi_{0}} \lambda_\alpha^t[ \epsilon^\alpha],
\end{equation}
interpolates between $J^0$ and $\theta_{>\chi_0}^0$ and $J^1$ and $\theta_{\chi_0}^1$ and has the following properties:

\begin{enumerate}
\item \label{main thm 1-parameter bare compactness}
If $(u, S)$ is a limit of bare solutions to $\bar\partial_{J^t}=\theta_{>\chi_0}^t$ which is not bare, 
i.e.\ $(u,S)$ has positive area part $(u_+, S_+)$ with $\chi(S) <  \chi(S_+)$, then $\chi(S_+) \le \chi_0$.  
\item \label{main thm 1-parameter genericity}
In $\bZ_{>\chi_0}\times[0,1]$ the bare solutions to $\bar\partial_{J^t}-\theta_{>\chi_0}^t=0$ are transversely cut out 1-parameter families of 1-generic maps. 
Moreover, a map from a nodal domain is a solution if and only if the corresponding map from the normalization is a solution, and these solutions come with the same signs and weights. 
\end{enumerate} 
Note that \eqref{main thm 1-parameter bare compactness} implies that
the locus in $\mathbf{Z}_{\ge \chi_0}$ of bare solutions to 
$\bar \partial_J = \theta_{> \chi_0}$  is compact.  
\end{theorem}  

\begin{proof}
The proof is similar to the proof of Theorem \ref{main theorem 0-parameter}.
Again the inductive hypothesis holds vacuously at $\chi = \chi_{\max} + 1$. The only new ingredient in the inductive step concerns the assertion that a map is a solution if and only if its normalization is a solution. 

Since basic perturbations vanishes near $L$, see Definition \ref{def : near L property} so does the perturbation $\theta_{>\chi_0}^t$. Let $\phi_t$ denote the restriction of $\theta_{>\chi_0}$ to $\bZ_{<\chi_0}$. If $(u_t,S_t)$ is a bare map with $\chi(S_t)=\chi_0$ with a single hyperbolic or elliptic boundary node then its local extension $\hat\phi_t$ on a neighborhood $V$ of $(\hat u_t,\hat S_t)$ in $\bZ_{>\chi_0-1}\times[0,1]$, see Section \ref{ssec : gluing 1-parameter familes}, satisfies $\theta_{>\chi_0}|_V=\hat\phi_t$.  

Lemmas \ref{l: 1-generic nodal solutions} and \ref{l : standard crossings/nodes in 1 parameter families} then shows that transverse crossing families of solutions in Euler characteristic $\chi=\chi_0+1$ together with the corresponding nodal family in $\chi=\chi_0$
are conjugated to the standard boundary crossing (or crossing with $L$) and hyperbolic (or elliptic) families.
This implies that in the inductive step when we choose $\lambda_\beta$ supported in $\bU(u_\beta,S_\beta)\times [0,1]$ with $\chi(S_\beta)=\chi_0$, we already have the desired transversality in a neighborhood of any bare solution with a boundary node. We then take $\lambda_\beta$ supported in $\bU(u_\beta,S_\beta)$ where $S_\beta$ is smooth and where $\bU(u_\beta,S_\beta)$ is disjoint from some neighborhood of nodal maps. The proof of the inductive step then follows as in the proof of Theorem \ref{main theorem 0-parameter}. 
\end{proof}

\subsection{Reduced Gromov Witten invariants}
In this section we recover Gromov-Witten counts from bare curve counts: we prove Theorem \ref{t : bare to GW intro}, which asserts: 
\begin{equation} \label{bare vs GW} Z_{\mathrm{bare}}(\beta,e^{g_s/2}-e^{-g_s/2}) = Z_{\mathrm{GW}}(\beta,g_s).
\end{equation}

Let us first explain what is being asserted and what needs to be shown.  Let $\theta_{\mathrm{bare}}$ be a perturbation constructed by Theorem \ref{main theorem 0-parameter}.  The left hand side of \eqref{bare vs GW} is defined by counting bare solutions to $\bar \partial_J u = \theta_{\mathrm{bare}}(u)$, per \eqref{bare curve count defined}.  
As for the right hand side, 
our perturbations do not achieve transversality on the non-bare locus, but, by the general results of \cite{HWZ-GW}, could be further perturbed to $\theta_{\mathrm{bare}} + \delta\theta_{\mathrm{GW}}$ to ensure this, i.e., in the language of \cite{HWZ-GW}, $\theta_{\mathrm{GW}}$ is an $\mathrm{sc}^+$ multi-section, $\delta\in\R$ is an arbitrarily small number, and $\theta_{\mathrm{bare}} + \delta\theta_{\mathrm{GW}}$ is transverse to the $\bar\partial_J$ Fredholm section. Counting all solutions to 
$\bar \partial_J u = (\theta_{\mathrm{bare}} + \delta\theta_{\mathrm{GW}})(u)$, each by (the multi-section intersection number contribution times) $g_s^{-\chi(S)}$ gives the Gromov-Witten count and defines the right hand side of \eqref{bare vs GW}.  

Note that both sides of \eqref{bare vs GW} are invariant under deformation: the left hand side under deformations within the allowed class of $\theta_{\mathrm{bare}}$ by Theorem \ref{main theorem 1-parameter} and the right hand side by \cite{HWZ-GW}.  Thus it will suffice to establish the equality for a single $\theta_{\mathrm{bare}}$.

Recall that in  Theorem \ref{main theorem 0-parameter}, the positive part of each non-bare solution $(v, T)$ to $\bar\partial_J-\theta_{\mathrm{bare}}$ was near some bare solution $(v_\ast,T_\ast)$. 

\begin{lemma}\label{l : uniqueness of underlying bare}
There exists $\theta_{\mathrm{bare}}$ for which Theorem \ref{main theorem 0-parameter} holds and such that for if $(v,T)$ is any solution to $\bar\partial_J-\theta_{\mathrm{bare}}$ then the nearby bare solution $(v_\ast,T_\ast)$ is unique.    
\end{lemma}

\begin{proof}
Since any solution to $\bar \partial_J = \theta_{\mathrm{bare}}$ is associated to some transversely cut out bare curve $(u,S)$ which is an embedding, it will suffice to show that there are no other bare solutions in a neighborhood of $u(S)$. This follows from ghost bubble censorship, Theorem \ref{ghost bubble censorship}: if there is no such neighborhood then there would be a sequence of bare solution curves, other than $(u, S)$, with limit which is a map to $u(S)$. Since the degree of the maps is fixed, the map in the limit can only be the curve $S$ with ghosts attached, which contradicts Theorem \ref{ghost bubble censorship}.
\end{proof}

Using Lemma \ref{l : uniqueness of underlying bare} it will be sufficient to prove the identity locally by 
showing that each bare curve gives rise to exactly $e^{g_s/2}-e^{-g_s/2}$ Gromov-Witten solutions.  

Take $\theta_{\mathrm{bare}}$ so that Lemma \ref{l : uniqueness of underlying bare} holds and define a bare solution $(u,S)$, 
$$
\mathrm{Cont}_{\mathrm{GW}}(u, S) \ := \ \sum_{(v_*, T_*) = (u, S)} \frac{[\bar\partial_J \cap (\theta_{\mathrm{bare}} + \delta \theta_{\mathrm{GW}})]_{(v, T)}}{[\bar\partial_J \cap \theta_{\mathrm{bare}}]_{(u, S)}}  \cdot g_s^{-\chi(S)},$$
where the sum ranges over all spaces of non-bare solutions $(v,T)$ to $\bar\partial_J-\theta_{\mathrm{bare}}$ with $(v_\ast,T_\ast)=(u,S)$, one solution space for each Euler characterstic $\chi=\chi(T)<\chi(S)$.    
Theorem \ref{t : bare to GW intro} now follows from the following local result: 
\begin{theorem}\label{c : bare to GW}
There exist $\theta_{\mathrm{bare}}$ such that
for any bare solution $(u,S)$ to $\bar \partial_J u = \theta_{\mathrm{bare}}(u)$, we have:
\begin{equation}\label{eq : bare to GW local formula}
\mathrm{Cont}_{\mathrm{GW}}(u, S)=
(e^{g_s/2}-e^{-g_s/2})^{-\chi(S)}.
\end{equation}
\end{theorem}
\begin{proof}
In the model case $\theta_{\mathrm{bare}}(u) = 0$, we may deform $\theta_{\mathrm{bare}}+\delta\theta_{\mathrm{GW}}$ to $\delta = 0$ to obtain a Bott-degenerate space of solutions.  The obstruction bundle calculation that yields the right hand side of \eqref{eq : bare to GW local formula} is a result of Pandharipande \cite{Pandharipande-degenerate}. It is carried out using the fact that the space of ghost contributions can be described in terms of products and symmetric products of $S$. For non-vanishing $\theta_{\mathrm{bare}}$, we do not know any such description of the space of solutions $(v, T)$ with $(v_*, T_*) = (u, S)$ and we will derive the result instead by deforming to a setting where we can apply \cite{Pandharipande-degenerate}.
It follows from Lemma \ref{l : uniqueness of underlying bare} that $\mathrm{Cont}_{\mathrm{GW}}(u, S)$ is the count of all (further perturbed) solutions in a neighborhood of the image $u(S)$.  

We will need slightly more control in order to argue inductively.  In our inductive construction of perturbations, $(u, S)$ is a (transverse bare) solution to the equation $\bar\partial_J -\theta_{>\chi(S)-1}$.  
The $[\epsilon^\beta]$ appearing on the right hand side of \eqref{extending down one step epsilons} only concern the perturbation on components of configuration space for lower Euler characteristic ($<\chi(S)$) domains, i.e.:
\begin{equation} \label{extending one step down no epsilons} \theta_{> \chi(S) - 1} |_{\mathbf{Z}^{\chi(S)}} \ := \ \theta_{> \chi(S)}|_{\mathbf{Z}^{\chi(S)}} + \sum \lambda_\beta.
\end{equation}

Since the ghost bubble censorship argument applies already to bare solutions to $\bar\partial_J -\theta_{>\chi(S)}$,  we can fix a neighborhood $U$ of $u(S)$ which contains no such bare solutions. Then, as long as we choose future $\epsilon^\beta$ sufficiently small, we do not ruin the property that $(u, S)$ is the only bare solution contained in the neighborhood $U$.   
This means that the neighborhood $U$ can be chosen already at step $\chi(S)$ in the inductive construction, provided we impose some further constraint on all later choices of $\epsilon^\beta$.  (These constraints on $\epsilon^\beta$, and others appearing in similar arguments below, are why we did not assert that \eqref{eq : bare to GW local formula} holds for all $\theta_{\mathrm{bare}}$ as constructed in Theorem \ref{main theorem 0-parameter}.) 
Moreover, as long as we choose further $\epsilon^\beta$ sufficiently small, all solutions $(v, T)$ to $\bar \partial_J - \theta_{\mathrm{bare}}$ with $(v_*, T_*) = (u, S)$ will have image contained in $U$.   
We now seek to deform the complex structure on $U$ to one in which $u(S)$ is $J$-holomorphic.  To simplify notation,  identify $S$ and its (embedded) image $u(S)$. We model a neighborhood of $S$ on the orthogonal complement of $TS\subset TX|_S$ with respect to the metric $\omega(J\cdot,\cdot)$. Let $N\to S$ denote the corresponding bundle. If $j$ denotes the complex structure along $S$, $N_s\subset T_s X$ the vectors tangent to the fiber of $N$ at $s$ and $\tau$ a vector in $T_s S$ then we have,
\[
J(N_s)\subset N_s,\quad J\tau = j\tau + E\tau,
\] 
where $\|E\|_{C^k}=\|\theta_{>\chi(S)-1}(u)\|_{C^k}\ll 1$, $k\ge 2$ (elliptic regularity shows that $u$ is smooth and $\theta_{>\chi(S)-1}$ is a sum of extended basic perturbations which are constructed using smooth functions, see \eqref{eq: uncut basic}). Consider the complex structures $J(t)$, $0\le t\le 1$ along $S$ given instead by
\[
J(0)|_{N_s}=J|_{N_s},\quad J(t)\tau = j\tau + tE\tau,
\]
then $J=J(1)$ and $J(t)$ admits extensions to $N$ such that $J(t)=J$ outside $N'\subset N$ and such that $\|J-J(t)\|=\mathcal{O}(\|\theta\|)$. Note that $S$ is $J(0)$-holomorphic and transversely cut out. Also, if $S(t)=(J+J(t))^{-1}(J(t)-J)$ then $JS(t)+S(t) J=0$, $J(t)=J(1+S(t))(1-S(t))^{-1}$, and if $A$ is a $(J,j)$-complex (anti-)linear section of $u^\ast TX$ for some $u\colon T\to N$ then $(1-S(t))A$ is $(J(t),j)$-complex (anti-)linear. 

By ghost bubble censorship, there is some neighborhood $U' \subset U$ such that the only bare solution to $\bar \partial_{J(0)} $ in $U'$ is $(u, S)$.  
Now both $(J=J(1),\theta_{\mathrm{bare}})$ and $(J(0), 0)$ are good perturbation data for counting bare curves in $U'$.  By Theorem \ref{main theorem 1-parameter} we may now connect the perturbation data by some 1-parameter family $\theta_{\mathrm{bare}}(t)$ with $\theta_{\mathrm{bare}}(1) = \theta_{\mathrm{bare}}$ and $\theta_{\mathrm{bare}}(0) = 0$.  By inspection of the construction in the proof, we may take this interpolation such that $(u, S)$ is a (transversely cut out bare) solution for every $t\in [0,1]$. Then, by ghost bubble censorship, there is some possibly smaller neighborhood $U''\subset U'$ of $S$ such that there are no bare solutions other than $(u,S)$ in $U''$ throughout the 1-parameter family.  

We next argue that it is possible to arrange that the size of $U''$ depends only on choices made at step $\chi(S)$ or before of the construction.  Note that the 1-parameter family, when restricted to $\mathbf{Z}_{\ge \chi(S)}$ depends only on $\theta$ restricted to the same configuration space.  We rearrange the steps in the inductive procedure slightly: after we have written \eqref{extending one step down no epsilons}, but before having chosen the corresponding $\epsilon^\beta$, we construct the interpolating 1-parameter family on $\mathbf{Z}_{\ge \chi(S)}$, and determine the corresponding $U''$. Then we pick later $\epsilon^\beta$ sufficiently small that, throughout the 1-parameter family, all solutions $(v, T)$ to $\bar \partial_{J(t)} - \theta_{\mathrm{bare}}(t)$ with $(v_*, T_*) = (u, S)$ remain in $U''$.  

It follows that our desired count accounts for all solutions to $\bar \partial_J - \theta_{\mathrm{bare}} - \delta\theta_{\mathrm{GW}}$ in $U''$ and also that the same holds parametrically in the family $(J(t),\theta_{\mathrm{bare}}(t))$, where at $t=0$, we are in a situation where $\theta_{\mathrm{bare}}(t)=0$. Then, taking $\delta \to 0$ and applying \cite{Pandharipande-degenerate}, \eqref{eq : bare to GW local formula} follows.
\end{proof}

\section{Skein valued curve counts}
Let $X$ be a symplectic Calabi-Yau 3-fold and $L\subset X$ a Maslov zero Lagrangian submanifold. In \cite{SOB}, we showed that, when holomorphic multiple covers could be excluded a priori (e.g. by topological considerations), holomorphic curves in $X$ with boundary on $L$ can be invariantly counted by the class of their boundary in the skein module of $L$. In that article, 
we structured the central invariance argument as follows: first, we gave in \cite[Theorem 6.1]{SOB} a list of properties which generic 0- and 1- parameter families of holomorphic curves satisfy; we used only these properties in formulating the definition of an candidate invariant in \cite[Definiton 6.2]{SOB}, and used only these properties in proving invariance in \cite[Theorem 6.3]{SOB}. 

In the present article, we have by now constructed a class of perturbations sufficient to ensure that solutions to the perturbed holomorphic curve equation are never multiply covered.  We will now show that these perturbations for skein-valued curve counting, thus showing that the skein-valued curve counting of \cite{SOB} can be used without any a priori exclusion of multiply covered holomorphic curves.
Our main task is to show that solutions to our perturbed holomorphic curve retain the properties of holomorphic curves enumerated in \cite[Theorem 6.1]{SOB}.  (The only differences will be that `manifolds' of solutions, or pre-images of the zero section, will be replaced by `weighted branched manifolds', or pre-images of the multi-sections of the perturbation, and similarly $\Z$ counts will be $\Q$ counts.) 

As in \cite[Section 2]{SOB}, we fix a non-vanishing vector field $\xi$ along $L$ and a $4$-chain $C$ such that $\partial C=2L$ and such that the normal to $C$ along the boundary is $\pm J\xi$, we also assume that the intersection $\mathrm{int}(C)\cap L$ is transverse and denote it $\gamma$.

\begin{corollary}\label{cor : count curves} (compare \cite[Theorem 6.1]{SOB})
Let $X$ be a symplectic Calabi-Yau 3-fold, let $L\subset X$ a Maslov zero Lagrangian submanifold.
Let $\theta_0,\theta_1\colon\bZ\to\bW$ be $\mathrm{sc}^+$ multi-sections as in 
Theorem \ref{main theorem 0-parameter} 
and let
and $\M(\theta_{t_0})$, $t_0=0,1$ be the moduli space of bare solutions to $\bar\partial_J=\theta_{t_0}$. Let $\Theta\colon\bZ\times[0,1]\to \bW\times[0,1]$ be a 1-parameter family of multi-sections as in Theorem \ref{main theorem 1-parameter} 
connecting $\theta_0$ to $\theta_1$, $\Theta(u,S;t)=\theta_t(u,S)$, and let $\mathcal{M}(\Theta)$ be the moduli space of bare solutions $\bar\partial_J^{-1}(\Theta)$, $\mathcal{M}(\Theta)=\bigcup_{t\in [0,1]}\mathcal{M}(\theta_t)$. Then the following hold. 
\begin{enumerate}
\item 
\label{prop:coherence} 
	{\bf Coherence.} For any $t\in[0,1]$, we have $(u,S) \in \M(\theta_t)$ if and only if the normalization $(\tilde{u},\tilde S) \in \M(\theta_t)$. 
\item 
\label{prop:codimzero}
{\bf Codimension zero transversality.}
Let $t_0=0,1$ we have
	\begin{enumerate}
	\item \label{prop:codimzerocompact} 
		$\M(\theta_{t_0})$ is compact. 
	\item \label{prop:isolated} 
		$\M(\theta_{t_0})$ is an oriented weighted zero-orbifold.
	\item \label{prop:embedded} 
	        Any map $(u,S)$ corresponding to a point in $\M(\theta_{t_0})$ is an embedding of 
	        a smooth curve with interior disjoint from $L$ and with the tangent to the boundary nowhere tangent to $\xi$. 
	\item \label{prop:trivialcobordism}
		The map $\M(\Theta) \to [0,1]$ gives a weighted branched orbifold cobordism from $\M(\theta_0)$ to $\M(\theta_1)$, and can be oriented so that the orientations of $\M(\theta_0), \M(\theta_1)$ are recovered from the boundary orientation.  
	\end{enumerate}
\item \label{prop:codimone} {\bf Codimension one transversality.}
	We also have 
	\begin{enumerate}
	\item \label{prop:codimonecompact} 
		$\M(\Theta)$ is compact.  		
	\item \label{prop:boundarycobordism} 
        $\M(\Theta)$ is an oriented weighted branched 1-orbifold 
        with boundary, to which  
        the orientation 
        on $\M(\Theta)|_{t=0,1}$  from \eqref{prop:trivialcobordism}
        extends.  Boundary  points over the interior $(0,1)$ are precisely the maps in $\mathcal{M}(\Theta)$ with nodal domain.
	\item \label{prop:onecrossingatatime}
        The locus of $t \in [0,1]$ so where \eqref{prop:trivialcobordism} or \eqref{prop:embedded} fails is discrete, and the failure over some $t_1$ occurs at only one point of $\mathcal{M}(\theta_{t_1})$. If \eqref{prop:trivialcobordism} fails at $(u_{t_1}, S_{t_1})$, we term it a \emph{critical point}; if \eqref{prop:embedded} fails, 
		  we term it 
		  a \emph{crossing}. 
	\item \label{prop:crossingtypes} 
		The universal map over a neighborhood of a crossing $(u_{t_1}, S_{t_1}) \in U \subset
		\Mm(\Theta)$ takes one of the following forms: 
		\begin{enumerate}
		\item \emph{Hyperbolic crossing}.  
			The map $(u_{t_1}, S_{t_1})$ is an immersion everywhere and an 
			embedding save at two points $s_1\ne s_2$ in the boundary such that $u_{t_1}(s_1)=u_{t_1}(s_2)$. The images of the two boundary 
			tangent vectors at $s_1$ and $s_2$ are linearly independent from each other and from $\xi$, and 
			that they together with the first order variation of the $1$-parameter family at the double point, $\partial_t u_{t_1}(s_1)-\partial_t u_{t_1}(s_2)$, span the tangent space of $L$.  
		\item \emph{Elliptic crossing}.  
			The map $(u_{t_1}, S_{t_1})$ is an embedding, but some interior point $s$ of the domain is 
			mapped to $L$.  The map at this point is transverse to the 4-chain $C$. 
			At the crossing moment the tangent space of the curve and the tangent 
			space of $L$ together with the first order variation of $1$-parameter 
			family $\partial_t u_{t_1}(s)$ span the tangent space of $X$. 
		\item \emph{Framing change}: 
			$(u_{t_1}, S_{t_1})$ is an embedding, but 
			$\partial u_{t_1}$ becomes tangent to $\xi$ or intersects $\gamma$ generically. 
		\end{enumerate}
	\end{enumerate}
\item \label{prop:standard} {\bf Gluing.} 
	Let $(u_{t_1}, S_{t_1}) \in \M(\theta_{t_1})$ be an elliptic or hyperbolic 
	crossing as in $\mathrm{(3d)}$ above.  Let $v_{t_1}$ be the map with the same image which is an embedding of a nodal curve $T_{t_1}$.  Let $\M_u$ and $\M_v$ be small neighborhoods of $(u_{t_1}, S_{t_1})$ and $(v_{t_1}, T_{t_1})$ in $\M(\Theta)$,  
	and let $u_t,v_t$ be the corresponding families of maps for $t$ neat $t_1$. Then there are neighborhoods $D, D'$ of the crossing point and node such that $(u_t(D), v_t(D'))$ is a standard
	hyperbolic or elliptic degeneration, in the sense of 
	\cite[Definition 3.1]{SOB} or \cite[Definition 3.3]{SOB}, respectively.  
    
    Moreover, the orientations of $\M_u$ and $\M_v$ are related as follows: the moduli space $\M_v$ is non-empty over a half-interval $(t_0-\epsilon,t_0]$ or $[t_0,t_0+\epsilon)$, over which we may compare its orientation with $\M_u$ using the projections of both to $(t_0-\epsilon, t_0+\epsilon)$. 
    We require that there is a global sign $\sigma=\pm 1$ such that the orientations of $\mathcal{M}_u$ and $\mathcal{M}_v$ differ by the product $\sigma\cdot\nu(u_t)$ where $\nu(u_t)$ is the local crossing sign of $u_t$, 
    i.e., $\nu(u_t)$ is the crossing sign between boundary arcs in the hyperbolic case and between $L$ and the curve
    in the elliptic case.  
\end{enumerate}
\end{corollary}

\begin{proof} 
Properties $(1)$, $(2)$, and $(3)$, except for the last statement of $(2\mathrm{c})$ and $(3\mathrm{d(iii)})$, follow from Theorems \ref{main theorem 0-parameter} and \ref{main theorem 1-parameter} where compactness and transversality is established and \cite[Theorem 4.14]{HWZ-perturb} where it is shown that this leads to solution spaces that are oriented weighted branched orbifolds with good boundary behavior. 

The last statement of $(2\mathrm{c})$ and Property $(3\mathrm{d(iii)})$ say that the boundaries of solutions generically have the general position properties with respect to the given vector field $\xi$ on $L$ that follow from dimension counting. To prove this we use an augmented Fredholm problem where we add a boundary marked points and evaluate the derivative at that point. The argument is directly analogous to those used for Properties $(3\mathrm{d(i)})$ and $(3\mathrm{d(ii)})$ (cf.\ Sections \ref{ssec:clindiff} -- \ref{sssec : genericity}), see \cite[Lemma A.6]{SOB}.  

In order to see that $(4)$ holds we note that the basic perturbations in Theorem \ref{main theorem 1-parameter} are supported away from the Lagrangian $L$, this means that solutions are $J$-holomorphic near the boundary and the analysis from \cite[Theorem 4.2]{SOB} carries over directly.
\end{proof}

Corollary \ref{cor : count curves} allows us to define the skein valued curve counting invariant exactly as in \cite[Definiton 6.2]{SOB}. Write $\Sk(L)$ for the $\Q[a^\pm,z^\pm]$-skein module of $L$. Take $\theta$ so that Corollary \ref{cor : count curves} holds and define 
\begin{equation}\label{eq : def skein valued partition function}
Z_{(X,L)}=\sum_{(u,S)\in\mathcal{M}(\theta)} 
w(u,S) \cdot z^{-\chi(S)}\cdot a^{u(S)\cdot C} \langle u(\partial S)\rangle \ \in \ \Sk(L).
\end{equation}
Here $w(u)\in\Q$ is the rational weight of the solution $(u,S)$, $\chi(S)$ the Euler characteristic of $S$, $u(S)\cdot C$ counts intersections of the interior of $u(S)$ and $L$, and $\langle u(\partial S)\rangle$ is the element represented by the boundary framed by $\xi$ in the HOMFLYPT skein module of $L$.

\begin{corollary}\label{invariance}
The skein valued curve count $Z_{(X,L)}$ in \eqref{eq : def skein valued partition function} is independent of $\theta$ and invariant under deformations of $L$ and $J$.     
\end{corollary}

\begin{proof}
    We repeat verbatim the proof of \cite[Theorem 6.3]{SOB} with Corollary \ref{cor : count curves} playing the role of \cite[Theorem 6.1]{SOB}. 
\end{proof}

\bibliographystyle{hplain}
\bibliography{skeinrefs}

@article{Ekholm-Shende-Longhi,
  title={The skein valued mirror of the topological vertex},
  author={Ekholm, Tobias and Longhi, Pietro and Shende, Vivek},
  journal={arXiv:2412.15454}
}

@article {Moser,
    AUTHOR = {Moser, J\"urgen},
     TITLE = {A rapidly convergent iteration method and non-linear partial
              differential equations. {I}},
   JOURNAL = {Ann. Scuola Norm. Sup. Pisa Cl. Sci. (3)},
  FJOURNAL = {Annali della Scuola Normale Superiore di Pisa. Classe di
              Scienze. Serie III},
    VOLUME = {20},
      YEAR = {1966},
     PAGES = {265--315},
      ISSN = {0391-173X},
   MRCLASS = {35.06 (35.35)},
  MRNUMBER = {199523},
MRREVIEWER = {S.\ Bur\'y\v sek},
}

@book {McDuffSalamon,
    AUTHOR = {McDuff, Dusa and Salamon, Dietmar},
     TITLE = {{$J$}-holomorphic curves and symplectic topology},
    SERIES = {American Mathematical Society Colloquium Publications},
    VOLUME = {52},
   EDITION = {Second},
 PUBLISHER = {American Mathematical Society, Providence, RI},
      YEAR = {2012},
     PAGES = {xiv+726},
      ISBN = {978-0-8218-8746-2},
   MRCLASS = {53D45 (32Q65 53D35)},
  MRNUMBER = {2954391},
MRREVIEWER = {Mark\ Alan\ Branson},
}

@book {Taylor,
    AUTHOR = {Taylor, Michael E.},
     TITLE = {Partial differential equations {I}. {B}asic theory},
    SERIES = {Applied Mathematical Sciences},
    VOLUME = {115},
   EDITION = {Second},
 PUBLISHER = {Springer, New York},
      YEAR = {2011},
     PAGES = {xxii+654},
      ISBN = {978-1-4419-7054-1},
   MRCLASS = {35-01 (46N20 47F05 47N20)},
  MRNUMBER = {2744150},
       DOI = {10.1007/978-1-4419-7055-8},
       URL = {https://doi.org/10.1007/978-1-4419-7055-8},
}

@article{RMMP,
  title={Desingularizations of sheaves and reduced invariants},
  author={Rabano, Alberto Cobos and Mann, Etienne and Manolache, Cristina and Picciotto, Renata},
  journal={arXiv:2310.06727}
}

@article{Hu-Li-reduced,
  title={Derived resolution property for stacks, {E}uler classes, and applications},
  author={Hu, Yi and Li, Jun},
  journal={arXiv:1009.5109}
}

@article{Zinger-Li-reduced,
  title={On the genus-one {G}romov-{W}itten invariants of complete intersections},
  author={Li, Jun and Zinger, Aleksey},
  journal={Journal of Differential Geometry},
  volume={82},
  number={3},
  pages={641--690},
  year={2009},
  publisher={Lehigh University}
}

@article{Zinger-reduced-CY,
  title={The reduced genus 1 {G}romov-{W}itten invariants of {C}alabi-{Y}au hypersurfaces},
  author={Zinger, Aleksey},
  journal={Journal of the American Mathematical Society},
  volume={22},
  number={3},
  pages={691--737},
  year={2009}
}

@article{Zinger-reduced,
  title={Reduced genus-one {G}romov-{W}itten invariants},
  author={Zinger, Aleksey},
  journal={Journal of differential geometry},
  volume={83},
  number={2},
  pages={407--460},
  year={2009},
  publisher={Lehigh University}
}

@article{scharitzer-shende,
  title={Quantum mirrors of cubic planar graph Legendrians},
  author={Scharitzer, Matthias and Shende, Vivek},
  journal={arXiv:2304.01872}
}

@article{scharitzer-shende-2,
  title={Skein valued cluster transformation in enumerative geometry of {L}egendrian mutation},
  author={Scharitzer, Matthias and Shende, Vivek},
  journal={arXiv:2312.10625}
}

@article{hu-schrader-zaslow,
  title={Skeins, clusters and wavefunctions},
  author={Hu, Mingyuan and Schrader, Gus and Zaslow, Eric},
  journal={arXiv:2312.10186}
}

@phdthesis{niu2016refined,
  title={Refined convergence for genus-two pseudo-holomorphic maps},
  author={Niu, Jingchen},
  year={2016},
  school={State University of New York at Stony Brook}
}

@article{Doan-Ionel-Walpulski,
  title={The {G}opakumar-{V}afa finiteness conjecture},
  author={Doan, Aleksander and Ionel, Eleny-Nicoleta and Walpuski, Thomas},
  journal={arXiv:2103.08221}
}

@article{MNOP,
  title={Gromov--{W}itten theory and {D}onaldson--{T}homas theory, {I}},
  author={Maulik, Davesh and Nekrasov, Nikita and Okounkov, Andrei and Pandharipande, Rahul},
  journal={Compositio Mathematica},
  volume={142},
  number={5},
  pages={1263--1285},
  year={2006},
  publisher={London Mathematical Society}
}

@article{Pardon-MNOP,
  title={Universally counting curves in {C}alabi-{Y}au threefolds},
  author={Pardon, John},
  journal={arXiv:2308.02948}
}

@article{bai-swaminathan,
  title={Bifurcations of embedded curves and toward an extension of Taubes’s Gromov invariant to Calabi--Yau 3-folds},
  author={Bai, Shaoyun and Swaminathan, Mohan},
  journal={Duke Mathematical Journal},
  volume={173},
  number={15},
  pages={2947--3057},
  year={2024},
  publisher={Duke University Press}
}

@article{doan-walpulski-embedded,
  title={Counting embedded curves in symplectic 6-manifolds},
  author={Doan, Aleksander and Walpuski, Thomas},
  journal={Commentarii Mathematici Helvetici},
  volume={98},
  number={4},
  pages={693--769},
  year={2023},
  publisher={EUROPEAN MATHEMATICAL SOC-EMS}
}

@article{zinger-sharp,
  title={A sharp compactness theorem for genus-one pseudo-holomorphic maps},
  author={Zinger, Aleksey},
  journal={Geometry \& Topology},
  volume={13},
  number={5},
  pages={2427--2522},
  year={2009},
  publisher={Mathematical Sciences Publishers}
}

@article{ionel-genus1,
  title={Genus 1 enumerative invariants in $\mathbb{P}^n$ with fixed j invariant},
  author={Ionel, Eleny-Nicoleta},
  journal={Duke Mathematical Journal},
  volume={94},
  number={2},
  pages={279--324},
  year={1998},
  publisher={Duke University Press}
}

@phdthesis{jemison,
  title={Polyfolds of Lagrangian Floer theory in all genera},
  author={Jemison Jr, Michael Louis},
  year={2020},
  school={Princeton University}
}

@article{ghost,
	title={Ghost bubble censorship},
	author={Ekholm, Tobias and Shende, Vivek},
	journal={arXiv:2212.05835.  To appear in Comm. Anal. Geom.}
}

@article{SOB,
  title={Skeins on Branes},
  author={Ekholm, Tobias and Shende, Vivek},
  journal={arXiv:1901.08027}
}

@article{ekholm-longhi-nakamura-hopf,
  title={The worldsheet skein {D}-module and basic curves on {L}agrangian fillings of the {H}opf link conormal},
  author={Ekholm, Tobias and Longhi, Pietro and Nakamura, Lukas},
  journal={arXiv:2407.09836}
}

@article{ekholm-shende-unknot,
  title={Skein recursion for holomorphic curves and invariants of the unknot},
  author={Ekholm, Tobias and Shende, Vivek},
  journal={arXiv:2012.15366.  To appear in Comptes Rendus Math\'ematiques.}
}

@article{ekholm-shende-colored,
  title={Colored {HOMFLYPT} counts holomorphic curves},
  author={Ekholm, Tobias and Shende, Vivek},
  journal={Proc. Nat. Acad. Sci. USA},
  volume={122},
  number={50},
  year = {2025},
  pages={e2510118122}
}

@article {E,
	AUTHOR = {Ekholm, Tobias},
	TITLE = {Morse flow trees and {L}egendrian contact homology in 1-jet
	spaces},
	JOURNAL = {Geom. Topol.},
	FJOURNAL = {Geometry \& Topology},
	VOLUME = {11},
	YEAR = {2007},
	PAGES = {1083--1224},
	ISSN = {1465-3060},
	MRCLASS = {53D40 (57R17)},
	MRNUMBER = {2326943},
	MRREVIEWER = {Dragomir L. Dragnev},
	URL = {https://doi.org/10.2140/gt.2007.11.1083},
}

@article {ESS,
	AUTHOR = {Ekholm, Tobias and Etnyre, John and Sullivan, Michael},
	TITLE = {The contact homology of {L}egendrian submanifolds in {$\mathbb{
	R}^{2n+1}$}},
	JOURNAL = {J. Differential Geom.},
	FJOURNAL = {Journal of Differential Geometry},
	VOLUME = {71},
	YEAR = {2005},
	NUMBER = {2},
	PAGES = {177--305},
	ISSN = {0022-040X},
	MRCLASS = {53D35 (57R17)},
	MRNUMBER = {2197142},
	MRREVIEWER = {Joshua M. Sabloff},
	URL = {http://projecteuclid.org/euclid.jdg/1143651770},
}

@article {EESPxR,
	AUTHOR = {Ekholm, Tobias and Etnyre, John and Sullivan, Michael},
	TITLE = {Legendrian contact homology in {$P\times\mathbb{R}$}},
	JOURNAL = {Trans. Amer. Math. Soc.},
	FJOURNAL = {Transactions of the American Mathematical Society},
	VOLUME = {359},
	YEAR = {2007},
	NUMBER = {7},
	PAGES = {3301--3335},
	ISSN = {0002-9947},
	MRCLASS = {53D35 (53D40)},
	MRNUMBER = {2299457},
	MRREVIEWER = {Michael J. Usher},
	DOI = {10.1090/S0002-9947-07-04337-1},
	URL = {https://doi.org/10.1090/S0002-9947-07-04337-1},
}

@article {CEL,
	AUTHOR = {Cieliebak, Kai and Ekholm, Tobias and Latschev, Janko},
	TITLE = {Compactness for holomorphic curves with switching {L}agrangian
	boundary conditions},
	JOURNAL = {J. Symplectic Geom.},
	FJOURNAL = {The Journal of Symplectic Geometry},
	VOLUME = {8},
	YEAR = {2010},
	NUMBER = {3},
	PAGES = {267--298},
	ISSN = {1527-5256},
	MRCLASS = {53D40 (53D12 53D42 53D45)},
	MRNUMBER = {2684508},
	MRREVIEWER = {Lenhard L. Ng},
	URL = {http://projecteuclid.org/euclid.jsg/1283865584},
}

@article{A,
  title={Topological invariants of knots and links},
  author={Alexander, James W},
  journal={Transactions of the American Mathematical Society},
  volume={30},
  number={2},
  pages={275--306},
  year={1928}
}

@article{PT,
  title={Invariants of links of Conway type},
  author={Przytycki, J\'{o}zef  and Traczyk, Pawel},
  journal={Kobe journal of mathematics},
  volume={4},
  pages={115--139},
  year={1988},
  publisher={????}
}

@article{RT,
  title={Higher genus symplectic invariants and sigma models coupled with gravity},
  author={Ruan, Yongbin and Tian, Gang},
  journal={Inventiones mathematicae},
  volume={130},
  number={3},
  pages={455--516},
  year={1997},
  publisher={Springer}
}

@book{HWZ,
  title={Polyfold and {F}redholm theory},
  author={Hofer, Helmut and Wysocki, Krzysztof and Zehnder, Eduard},
  year={2021},
  publisher={Springer Nature}
}

@article{HWZ-models,
  title={Sc-smoothness, retractions and new models for smooth spaces},
  author={Hofer, Helmut and Wysocki, Kris and Zehnder, Eduard},
  journal={Discrete \& Continuous Dynamical Systems-A},
  volume={28},
  number={2},
  pages={665},
  publisher={American Institute of Mathematical Sciences},
  year={2010}
}

@book{HWZ-GW,
  title={Applications of polyfold theory I: the polyfolds of Gromov--Witten theory},
  author={Hofer, Helmut and Wysocki, Kris and Zehnder, Eduard},
  year={2017},
  publisher={American Mathematical Society}
}

@article {HWZ-perturb,
	AUTHOR = {Hofer, Helmut and Wysocki, Kris and Zehnder, Eduard},
	TITLE = {A general {F}redholm theory. {III}. {F}redholm functors and
	polyfolds},
	JOURNAL = {Geom. Topol.},
	FJOURNAL = {Geometry \& Topology},
	VOLUME = {13},
	YEAR = {2009},
	NUMBER = {4},
	PAGES = {2279--2387},
	ISSN = {1465-3060},
	MRCLASS = {53D42 (46T99 47A53 53D40 53D45 57R17 58B15)},
	MRNUMBER = {2515707},
	MRREVIEWER = {Tobias Ekholm},
	DOI = {10.2140/gt.2009.13.2279},
	URL = {https://doi.org/10.2140/gt.2009.13.2279},
}

@article{Schmaltz,
  title={The {G}romov-{W}itten axioms for symplectic manifolds via polyfold theory},
  author={Schmaltz, Wolfgang},
  journal={arXiv:1912.13374}
}

@article{Pardon,
  title={An algebraic approach to virtual fundamental cycles on moduli spaces of pseudo-holomorphic curves},
  author={Pardon, John},
  journal={Geometry \& Topology},
  volume={20},
  number={2},
  pages={779--1034},
  year={2016},
  publisher={Mathematical Sciences Publishers}
}

@book{FOOO,
  title={Lagrangian intersection {Floer} theory: anomaly and obstruction},
  author={Fukaya, Kenji and Oh, Yong-Geun and Ohta, Hiroshi and Ono, Kaoru},
  year={2010},
  publisher={American Mathematical Soc.}
}

@article{Pandharipande-degenerate,
  title={Hodge integrals and degenerate contributions},
  author={Pandharipande, Rahul},
  journal={Communications in Mathematical Physics},
  volume={208},
  number={2},
  pages={489--506},
  year={1999},
  publisher={Springer}
}

@article{Ionel-Parker-GV,
  title={The {G}opakumar-{V}afa formula for symplectic manifolds},
  author={Ionel, Eleny-Nicoleta and Parker, Thomas},
  journal={Annals of Mathematics},
  volume={187},
  number={1},
  pages={1--64},
  year={2018},
  publisher={Department of Mathematics, Princeton University Princeton, New Jersey, USA}
}

@article{Pandharipande-Solomn-Tessler,
  title={Intersection theory on moduli of disks, open {KdV} and {V}irasoro},
  author={Pandharipande, Rahul and Solomon, Jake and Tessler, Ran},
  journal={Geometry \& Topology},
  volume={28},
  number={6},
  pages={2483--2567},
  year={2024},
  publisher={Mathematical Sciences Publishers}
}

@article{OV,
	author         = "Ooguri, Hirosi and Vafa, Cumrun",
	title          = "{Knot invariants and topological strings}",
	journal        = "Nucl.Phys.",
	volume         = "B577",
	pages          = "419-438",
	doi            = "10.1016/S0550-3213(00)00118-8",
	year           = "2000"
}

@incollection{Witten,
  title={Chern-Simons gauge theory as a string theory},
  author={Witten, Edward},
  booktitle={The Floer memorial volume},
  pages={637--678},
  year={1995},
  publisher={Springer}
}

@article{GV-MtopII,
 title={M-theory and topological strings, {II}},
 author={Gopakumar, Rajesh and Vafa, Cumrun},
 journal={arXiv:hep-th/9812127}
}

@article{M,
  title={Stable pairs and the {HOMFLY} polynomial},
  author={Maulik, Davesh},
  journal={Inventiones mathematicae},
  volume={204},
  number={3},
  pages={787--831},
  year={2016},
  publisher={Springer}
}

@article{Abouzaid-McLean-Smith,
  title={Complex cobordism, Hamiltonian loops and global Kuranishi charts},
  author={Abouzaid, Mohammed and McLean, Mark and Smith, Ivan},
  journal={arXiv preprint arXiv:2110.14320},
  year={2021}
}

@article{Hirschi-Swaminathan,
  title={Global Kuranishi charts and a product formula in symplectic Gromov--Witten theory},
  author={Hirschi, Amanda and Swaminathan, Mohan},
  journal={Selecta Mathematica},
  volume={30},
  number={5},
  pages={87},
  year={2024},
  publisher={Springer}
}

@article{Siebert,
  title={Symplectic Gromov-Witten invariants},
  author={Siebert, Bernd},
  journal={London Mathematical Society Lecture Note Series},
  pages={375--424},
  year={1999},
  publisher={Cambridge University Press}
}

\end{document}